\documentclass[11pt, reqno, oneside]{amsart}

\usepackage{amssymb, amsfonts, mathrsfs, verbatim, graphicx,graphics, epsfig, mathtools, bm, bbm, mathabx}

\usepackage{soul} %------------------For HighLight----------------------------
\usepackage{microtype} %-----Improve spacing

\usepackage{color, xcolor}

\usepackage{abstract}
\usepackage[all]{xy}	

\usepackage[left= 3.4 cm, right= 3.4 cm, top= 2.8 cm, bottom=2.7 cm, foot= 1.2 cm, marginparwidth=2.7 cm, marginparsep=0.5 cm]{geometry}

\usepackage[inline]{enumitem}

\usepackage{hyperref}
\hypersetup{colorlinks}
\definecolor{darkred}{rgb}{0.5,0,0}
\definecolor{darkgreen}{rgb}{0,0.5,0}
\definecolor{darkblue}{rgb}{0,0,0.5}
\hypersetup{colorlinks, linkcolor=darkblue, filecolor=darkgreen, urlcolor=darkred, citecolor=darkblue}
\makeatletter % `@' now normal "letter"
\@addtoreset{equation}{section}
\makeatother  % `@' is restored as "non-letter"

\numberwithin{equation}{section}

\numberwithin{figure}{section}

\usepackage{tikz-cd}
\usepackage{tikz}
\usetikzlibrary{arrows}
\usepackage{float}

\usetikzlibrary{decorations.pathreplacing}
\usepackage{etex}
\makeatletter
\def\@tocline#1#2#3#4#5#6#7{\relax
  \ifnum #1>\c@tocdepth % then omit
  \else
    \par \addpenalty\@secpenalty\addvspace{#2}%
    \begingroup \hyphenpenalty\@M
    \@ifempty{#4}{%
      \@tempdima\csname r@tocindent\number#1\endcsname\relax
    }{%
      \@tempdima#4\relax
    }%
    \parindent\z@ \leftskip#3\relax \advance\leftskip\@tempdima\relax
    \rightskip\@pnumwidth plus4em \parfillskip-\@pnumwidth
    #5\leavevmode\hskip-\@tempdima
      \ifcase #1
       \or\or \hskip 1em \or \hskip 2em \else \hskip 3em \fi%
      #6\nobreak\relax
    \dotfill\hbox to\@pnumwidth{\@tocpagenum{#7}}\par
    \nobreak
    \endgroup
  \fi}
\makeatother

\newtheorem{thm}{Theorem}[section]
\newtheorem{cor}[thm]{Corollary}

\newtheorem{conjecture}[thm]{Conjecture}
\newtheorem{prop}[thm]{Proposition}
\newtheorem{lemma}[thm]{Lemma}

\theoremstyle{definition}
\newtheorem{defn}[thm]{Definition}

\theoremstyle{remark}
\newtheorem{rem}[thm]{Remark}
\newtheorem{hyp}[thm]{Hypothesis}

\newcounter{notes}
{\end{list}}

\makeatletter
\@addtoreset{thm}{section}
\makeatother

\newcommand\qu{/\kern-.7ex/} % Categorical quotients
\renewcommand{\setminus}{\smallsetminus}

\newcommand{\immerse}{\looparrowright}

\newcommand{\beq}{\begin{equation}}
\newcommand{\eeq}{\end{equation}}
\newcommand{\beqn}{\begin{equation*}}
\newcommand{\eeqn}{\end{equation*}}
\newcommand{\ov}{\overline}
\newcommand{\mb}{\mathbb}

\newcommand{\mc}{\mathcal}
\newcommand{\mf}{\mathfrak}

\newcommand{\A}{{\bm A}}
\renewcommand{\a}{{\bm a}}
\newcommand{\M}{{\bm M}}
\newcommand{\N}{{\bm N}}

\renewcommand{\P}{{\bm P}}
\renewcommand{\H}{{\bm H}}
\newcommand{\C}{{\bm C}}

\newcommand{\tr}{{\rm tr}}

\newcommand{\uds}[1]{\underline{\smash{#1}}}

\title[Compactness for ASD Equation with Translation Symmetry]{A Compactness Theorem for $SO(3)$ Anti-Self-Dual Equation with Translation Symmetry}

\begin{document}

\author[Xu]{Guangbo Xu}
\address{
Department of Mathematics\\
Texas A{\&}M University\\
College Station, TX 77843 USA
}
\email{guangboxu@math.tamu.edu}

\date{\today}

%------------------TOC----------------------

\setcounter{tocdepth}{1}

% set the table of contents up to sections

\maketitle

\begin{abstract}
Motivated by the Atiyah--Floer conjecture, we consider $SO(3)$ anti-self-dual instantons on the product of the real line and  a three-manifold with cylindrical end. We prove a Gromov--Uhlenbeck type compactness theorem, namely, any sequence of such instantons with uniform energy bound has a subsequence converging to a type of singular objects which may have both instanton and holomorphic curve components. In particular, we prove that holomorphic curves that appear in the compactification must satisfy the Lagrangian boundary condition, a claim which has been long believed in the literature. This result is the first step towards constructing a natural bounding cochain proposed by Fukaya for the $SO(3)$ Atiyah--Floer conjecture. 
\end{abstract}

\tableofcontents

\section{Introduction}

\subsection{The Atiyah--Floer conjecture}

In 1980s Floer \cite{Floer_instanton}\cite{Floer_intersection}\cite{Floer_CMP} introduced several important invariants of different types of geometric objects. These invariants are now generally called the Floer (co)homology. The Atiyah--Floer conjecture \cite{Atiyah_Floer_conjecture} asserts that two such invariants, the {\it instanton Floer homology} of a three-dimensional manifold and the {\it Lagrangian intersection Floer homology} associated to a splitting of the three-manifold, are isomorphic. This conjecture has become a central problem in the field of symplectic geometry, gauge theory, and low-dimensional topology. The principle underlying the Atiyah--Floer conjecture has also motivated a number of important constructions in low-dimensional topology, such as the Heegaard--Floer homology \cite{OZ}. There are two principal versions of the Atiyah--Floer conjecture, the $SU(2)$ case and the $SO(3)$ case, corresponding respectively to the two choices of the gauge group. Despite many progresses made in recent years, the general cases of both versions are still open. On the level of Euler characteristics, the Atiyah--Floer conjecture was proved by Taubes \cite{Taubes_1990}.

Let us briefly review Atiyah's intuitive argument \cite{Atiyah_Floer_conjecture} leading to the identification of the two Floer homologies. Let $M$ be a closed oriented three-manifold. Let $\Sigma \subset M$ be an embedded surface separating $M$ into two pieces $M^-$ and $M^+$ which share the common boundary $\Sigma$ (see Figure \ref{figure1}). Consider a $G$-bundle $P \to M$ where $G$ is either $SU(2)$ or $SO(3)$. The moduli space of flat connections on $P|_\Sigma$, denoted by $R_\Sigma$, is naturally a (singular) symplectic manifold. Inside $R_\Sigma$ there are two Lagrangian submanifolds $L_{M^-}$ and $L_{M^+}$ associated to the splitting, i.e., the set of gauge equivalence classes of flat connections on $\Sigma$ which can be extended to a flat connection on $P|_{M^+}$ resp. $P|_{M^-}$. The generators of the instanton Floer chain complex, which are gauge equivalence classes of flat connections on $M$, correspond naturally to intersections of the two Lagrangian submanifolds. These intersection points are generators of the Lagrangian Floer chain complex. On the other hand, the differential map of the instanton Floer homology $I(M, P)$ is defined by counting solutions to the anti-self-dual equation (the ASD equation) on the product ${\bm R} \times M$ of the real line ${\bm R}$ and $M$. This equation depends the metric on $M$ while the resulting homology is independent of the metric. If one ``stretches the neck,'' namely one chooses a family of metrics $g_T$ on $M$ such that a fixed neighborhood of $\Sigma$ is isometric to $[-T, T] \times \Sigma$, then one expects that solutions to the ASD equation converge as $T \to \infty$ to holomorphic maps $u: [-1, 1] \times {\bm R} \to R_\Sigma$ with boundary condition $u(\{\pm 1\} \times {\bm R}) \subset L_{M^\pm}$---the counting of these holomorphic maps defines the differential map of the Lagrangian intersection Floer chain complex with resulting homology group $HF(L_{M^-}, L_{M^+})$. The correspondence between instantons and holomorphic strips shows that the two Floer chain complexes, and hence the homology groups, should be isomorphic.

\begin{figure}[ht]
\centering
\includegraphics{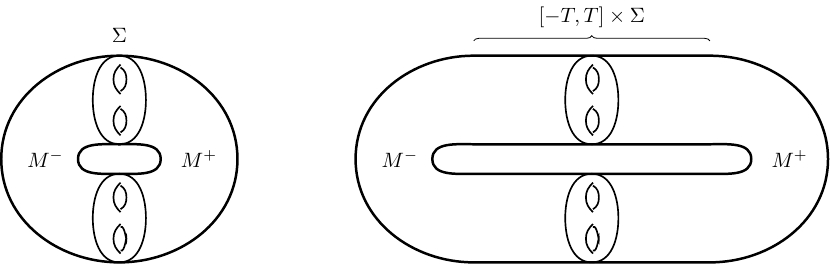}

\caption{Splitting and neck-stretching along an embedded surface.}

\label{figure1}

\end{figure}

This paper is motivated by the $SO(3)$ case of the Atiyah--Floer conjecture. We first remark on a crucial difference between the $SO(3)$ case and the $SU(2)$ case. When $G = SU(2)$, the moduli space of flat connections over a surface has singularities corresponding to reducible connections. This fact makes the Lagrangian Floer homology $HF(L_{M^-}, L_{M^+})$ difficult to define (see \cite{Manolescu_Woodward} for an equivariant construction of this Floer homology). When $G = SO(3)$, the moduli space $R_\Sigma$ of flat connections on a nontrivial $SO(3)$-bundle over a surface is smooth. Consequently the Lagrangian Floer homology can be defined in the traditional way and the corresponding Atiyah--Floer conjecture has been proved in certain cases. For example, Dostoglou--Salamon \cite{Dostoglou_Salamon} proved the $SO(3)$ case for mapping cylinders. 

Another motivation of this paper comes from the original neck-stretching argument sketched above. As people become more interested in the alternative approach using Lagrangian boundary conditions for the instanton equation (see \cite{Salamon_ICM} \cite{Fukaya_1997, Fukaya_1998} \cite{Wehrheim_CMP_1, Wehrheim_CMP_2, Wehrheim_2005} \cite{Salamon_Wehrheim_2008} \cite{Fukaya_2015} \cite{Daemi_Fukaya_2018}), the neck-stretching argument is more or less abandoned. Furthermore, the neck-stretching argument can provide a direct comparison between the moduli spaces, potentially leading to Atiyah--Floer type correspondences for more refined invariants. We would like to see how far Atiyah's original idea can go beyond the situation of \cite{Dostoglou_Salamon}. Meanwhile, the analytic problems involved in the neck-stretching limit have their own interests and deserve to be explored. 

\subsection{Bounding cochain on the symplectic side}

A more direct motivation of this paper is a conjecture of Fukaya \cite[Conjecture 6.7]{Fukaya_2018}. This conjecture is related to an additional complexity in the study of the Atiyah--Floer conjecture, that is, the regularity of the Lagrangian submanifolds. From now on we assume $G = SO(3)$, which grants a smooth moduli space $R_\Sigma$ of flat connections on the nontrivial $SO(3)$-bundle over the surface $\Sigma$. If both Lagrangians $L_{M^+}$ and $L_{M^-}$ are embedded and intersect transversely, then one can define $HF(L_{M^-}, L_{M^+})$ in a standard way as in \cite{Floer_intersection} \cite{Oh_1993, Oh_1993_2}, as the two Lagrangians are both monotone. In general, however, the natural maps $L_{M^\pm} \to R_\Sigma$ are not embeddings. After a generic perturbation, one can only achieve immersions $L_{M^\pm} \immerse R_\Sigma$ which satisfy the monotonicity condition in a weak sense. In this situation we need to apply the more complicated construction of immersed Lagrangian Floer homology developed by Akaho--Joyce \cite{Akaho_Joyce_2010} with the appearance of {\it bounding cochains}. A general Lagrangian immersion may not have bounding cochains; even it has, the Floer homology involving the immersed Lagrangian depends on the choice of a bounding cochain.

The necessity of considering bounding cochains in the Atiyah--Floer conjecture can also be seen via a closer look at the neck-stretching process. As one stretches the neck, the energy of ASD instantons can be distributed in three parts, ${\bm R} \times M^-$, ${\bm R} \times M^+$, and ${\bm R} \times [-T, T] \times \Sigma$ (corresponding to the left, the right, and the middle parts of the second picture of Figure \ref{figure1}). The energy stored on the left and on the right can form ASD instantons on ${\bm R} \times M_\infty^-$ and ${\bm R} \times M_\infty^+$, where $M_\infty^\pm$ is the completion of $M^\pm$ by adding the cylindrical end $\Sigma \times [0, +\infty)$. On the other hand, the energy stored in the middle part produces holomorphic strips in $R_\Sigma$ as predicted by the Atiyah--Floer conjecture. Hence a general limiting object could be a complicated configuration having components corresponding to either instantons over ${\bm R} \times M_\infty^\pm$ or holomorphic strips (see Figure \ref{figure2}); we ignore bubbling of instantons over ${\bm R}^4$, instantons over $\C \times \Sigma$, and holomorphic spheres, as they happen in high codimensions. While in the case when $L_{M^\pm}$ are immersed, instantons over ${\bm R} \times M_\infty^\pm$ may give nontrivial contributions which happen in codimension zero.

\begin{figure}[ht]
\centering

\includegraphics{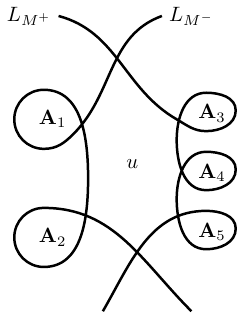}

\caption{A typical limiting configuration in the adiabatic limit: $u$ is a holomorphic strip in $R_\Sigma$ with boundary in $L_{M^-}$ and $L_{M^+}$, $\A_1$ and $\A_2$ are ASD instantons over ${\bm R} \times M_\infty^-$, $\A_3$, $\A_4$, $\A_5$ are ASD instantons over ${\bm R} \times M_\infty^+$. The limits of the instantons $\A_i$ are double points of the Lagrangian immersions $L_{M^\pm}$.}

\label{figure2}
\end{figure}

The geometric picture discussed above suggests that one has to modify the Lagrangian Floer chain complex to match with the instanton Floer chain complex. The differential map of the modified chain complex counts not only usual holomorphic strips, but also strips with additional boundary constraints associated with the instantons over ${\bm R} \times M_\infty^\pm$. This kind of boundary constraints can be regarded as bounding cochains. As introduced in \cite{FOOO_Book}, for any cochain model chosen for the Lagrangian Floer theory (such as Morse cochains), a bounding cochain is a cochain of odd degree which can cancel all contributions of disk bubbling; these disk bubbles may obstructed $d^2 = 0$ in the original Floer chain complex. In a symplectic manifold, if $L_-$, $L_+$ are Lagrangian submanifolds, then a choice of a pair of bounding cochains $b_-$, $b_+$ leads to a deformed complex $CF((L_-, b_-), (L_+, b_+))$. The differential map of the deformed complex counts holomorphic strips with boundary ``insertions,'' i.e., points on the two boundary components satisfying geometric constraints prescribed by the cochains $b_-$ and $b_+$ (see Figure \ref{figure3}).

\begin{figure}[ht]
\centering
\includegraphics{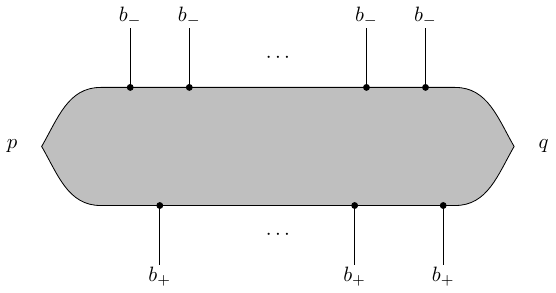}

\caption{A holomorphic strip with boundary insertions lying in constraints given by the bounding cochains. }
\label{figure3}
\end{figure}

The following conjecture of Fukaya summarizes the above discussion and directly motivates the work of the current paper. 

\begin{conjecture}\cite[Conjecture 6.7]{Fukaya_2018}\label{conj11}
Let $M$ be a three-manifold with a cylindrical end isometric to $\Sigma \times [0, +\infty)$ and $P \to M$ be an $SO(3)$-bundle whose restriction to every connected component of $\Sigma$ is nontrivial. Let $L_M$ be the moduli space of flat connections over $M$. Suppose the natural map $L_M\to R_\Sigma$ is an immersion with transverse double points. Then ``counting'' instantons on ${\bm R} \times M$ defines a bounding cochain
\beqn
b_M  \in CF (L_M). 
\eeqn
of the $A_\infty$-algebra associated to the immersed Lagrangian $L_M \immerse R_\Sigma$. 
\end{conjecture}

Indeed, for the situation of immersed Lagrangians considered in \cite{Akaho_Joyce_2010}, the cochain model has summands corresponding to the double points of the immersion. Hence {\it a priori} a bounding cochain is in general a linear combination of ordinary cochains on $L_M$ and double points of the immersion. As observed by Fukaya \cite{Fukaya_2018}, due to a weak version of monotonicity, the bounding cochain $b_M$ in the above conjecture is expected to be a linear combination of only double points. 

A refined version of the $SO(3)$ Atiyah--Floer conjecture can be stated as follows. 

\begin{conjecture} [The $SO(3)$ Atiyah--Floer conjecture]\label{conj12}
For a closed three-manifold $M$ with a suitable $SO(3)$-bundle $P \to M$ and a suitable splitting $M = M^- \cup M^+$, there is a natural isomorphism of abelian groups
\beq\label{eqn11}
I (M, P) \cong HF( (L_{M^-}, b_{M^-} ), (L_{M^+}, b_{M^+} ) ).
\eeq
\end{conjecture}

\begin{rem}
In \cite{Fukaya_2015, Fukaya_2018} a different strategy of proving the existence of a bounding cochain was sketched. Instead of considering instantons over ${\bm R} \times M$ where $M$ has a cylindrical end, consider the ASD equation over ${\bm R} \times M_0$ where $M_0$ is the corresponding manifold with boundary, imposing the boundary condition given by a testing Lagrangian $L \subset R_\Sigma$. The advantage of this approach is that it can avoid certain difficult analysis associated to the ASD equation on ${\bm R} \times M$, while a simpler moduli space is enough to produce a chain map. However, this approach lacks a direct comparison between the moduli space of instantons and moduli space of holomorphic strips as indicated by the straightforward neck-stretching argument. Such a comparison between moduli spaces can be useful in establishing relations between more refined invariants. The approach of using Lagrangian boundary conditions have also been adopted to solve the Atiyah--Floer conjecture, see \cite{Salamon_ICM} \cite{Fukaya_1997, Fukaya_1998, Fukaya_2015} \cite{Wehrheim_2005, Wehrheim_CMP_1, Wehrheim_CMP_2} \cite{Salamon_Wehrheim_2008} \cite{Daemi_Fukaya_2018}.
\end{rem}

\subsection{Main results of this paper}

The purpose of this paper is to take the first step towards the resolution of Conjecture \ref{conj11}, namely, to compactify the moduli space of ASD instantons over ${\bm R} \times M$ where $M$ is a three-manifold with cylindrical end. More precisely, given a sequence of ASD instantons $\A_i$ over ${\bm R} \times M$ with uniformly bounded energy, we study the possible limiting configurations as $i \to \infty$. There are several phenomena preventing $\A_i$ from converging to an ASD instanton. (i) As in the usual situation of the ASD equation, energy may concentrate in small scales and bubble off instantons on ${\bm R}^4$. (ii) Since ${\bm R} \times M$ is noncompact, the energy may concentrate at different regions of the same scale which move apart from each other. This is similar to the situation in Morse theory, where a sequence of gradient lines can converge to a broken gradient line. There can also be instantons over $\C \times \Sigma$ appearing as energy may escape in the noncompact direction of $M$. (iii) A nontrivial amount of energy may escape from any finite region of ${\bm R} \times M$ and spread over larger and larger domains; after rescaling such energy form either holomorphic spheres or holomorphic disks in $R_\Sigma$. In general a combination of these phenomena can happen in the limit. The hierarchy of different speeds of energy concentration or spreading is captured by the combinatorial type of the limiting object described by a tree. See Figure \ref{figure4} for a typical configuration of the limiting object, which we will call stable scaled instantons.

\begin{figure}[ht]
\centering

\includegraphics{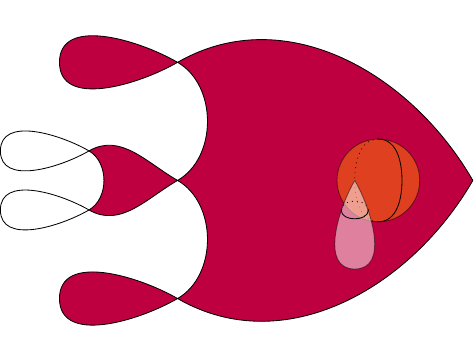}

\caption{A typical stable scaled instanton with eight components. The dark regions are holomorphic curve components while the corners are where the curves meet the double points of immersed Lagrangian. The white regions are instanton components: the two on the left are instantons over ${\bm R} \times M$ and the one on the right (the teardrop) is an instanton over ${\bm C} \times \Sigma$. }
\label{figure4}
\end{figure}

One can see that the limiting configurations are very similar to objects appearing in the adiabatic limit of the {\it symplectic vortex equation} (see \cite{Gaio_Salamon_2005} and \cite{Ziltener_thesis, Ziltener_book} for the closed case and \cite{Wang_Xu} \cite{Woodward_Xu} for the case with Lagrangian boundary condition). Combinatorially these objects are also similar to certain objects appearing in the compactification of pseudoholomorphic quilts studied by \cite{Wehrheim_Woodward_2015} and \cite{Bottman_Wehrheim_2018}. To describe such limiting objects, we need to define certain singular configurations which have components corresponding to energy concentrations in different scales (see \cite[Section 4]{Wang_Xu} and Section \ref{section6} of the current paper). Having this picture in mind, in this paper we define the notion of {\it stable scaled instantons} (see Definition \ref{defn62}) as the expected limiting objects, and a Gromov--Uhlenbeck type convergence (see Definition \ref{defn63}). Then we can state our main theorem as follows. 

\begin{thm}\label{thm14}
Let $M$ be a three-manifold with cylindrical end and $P \to M$ be an $SO(3)$-bundle. Suppose $(M, P)$ satisfies assumptions of Conjecture \ref{conj11}. Then given a sequence of anti-self-dual instantons on ${\bm R} \times P \to {\bm R} \times M$ with uniformly bounded energy, there is a subsequence which converges modulo gauge transformation and translation to a stable scaled instanton (in the sense of Definition \ref{defn63}). 
\end{thm}

Besides proving the above compactness theorem, using the same method and more simplified argument, we can also prove a counterpart for instantons over $\C \times \Sigma$.

\begin{thm}\label{thm15}
Let $\Sigma$ be a compact Riemann surface (not necessarily connected) and $Q \to \Sigma$ be an $SO(3)$-bundle which is nontrivial over every connected component of $\Sigma$. Then given a sequence of anti-self-dual instantons on $\C \times Q \to \C \times \Sigma$ with uniformly bounded energy, there is a subsequence which converges modulo gauge transformation and translation to a stable scaled instanton.
\end{thm}

%From the argument of proving the above two compactness theorems it is relatively easier to prove a compactness theorem for the adiabatic limit associated to neck-stretching. This extension is left to interesting readers.

We remark that certain compactness problems in gauge theory with respect to adiabatic limit or neck-stretching which are of similar nature have been considered by other people, for example Chen \cite{Chen_1998}, Nishinou \cite{Nishinou_2010}, and Duncan \cite{Duncan_thesis, Duncan_2012}. Comparing to these previous works, the main contribution of this paper is the treatment of the compactness problem near the ``boundary'' (the compact part of the three-manifold with cylindrical ends). The argument is based on the isoperimetric inequality (Theorem \ref{isoperimetric}), the annulus lemma (Proposition \ref{annulus} and extensions), and the boundary diameter estimate (Lemma \ref{diameter2}). The method of using boundary diameter estimate to establish boundary compactness would also be useful in other situations. For example, for the compactness problem about the strip-shrinking limit of pseudoholomorphic quilts, this method potentially leads to a simplified argument as opposed to the method of \cite{Bottman_Wehrheim_2018} which appeals to hard elliptic estimates over varying domains. Another contribution of the current paper is to define the correct notion of singular configurations (stable scaled instantons) that may appear in the limit and provide detailed argument of constructing the limiting bubble tree. Last but not the least, a modification of our construction will lead to a proof of a compactness theorem about the neck-stretching limit for the $SO(3)$ instanton equation. The details will be completed in future works. 

This paper is organized as follows. In Section \ref{section2} we recall basic notions and facts about the anti-self-dual equation, holomorphic curves, state the main assumption of the three-manifold, and recall a few technical results. In Section \ref{section3} we recall a basic compactness theorem for the rescaled ASD equation over the product of two surfaces in the adiabatic limit and prove a refinement of an interior estimate of Dostoglou--Salamon. In Section \ref{section4} we (re)prove an isoperimetric inequality for a closed three-manifold and the annulus lemma, and establish a boundary diameter estimate. In Section \ref{section5} we prove the compactness modulo energy blowup theorem for the ASD equation over the noncompact four-manifold ${\bm R} \times M$. In Section \ref{section6} we state the main theorem in technical terms. In Section \ref{section7} we finish the proof of the main theorem (Theorem \ref{thm14} and Theorem \ref{thm15}). In the appendix we provide a complete proof of Theorem \ref{thm33}.

\subsection*{Acknowledgments}

The author would like to thank Professor Kenji Fukaya for suggesting this problem, for many stimulating discussions, and for his warm encouragement and generous support. The author would like to thank Simons Center for Geometry and Physics for creating a wonderful environment for mathematical research. The author would like to thank Donghao Wang, David Duncan, and Chris Woodward for helpful discussions. The author also thanks the anonymous referee for many valuable suggestions to improve the paper.

This work is partially supported by the Simons Collaboration Grant on Homological Mirror Symmetry.

\section{Preliminaries}\label{section2}

In this paper we study gauge theory for $SO(3)$-bundles over manifolds of dimension at most $4$. Let $U$ be such a smooth manifold and $P \to U$ be a smooth $SO(3)$-bundle. As in the usual treatment of $SO(3)$-gauge theory, we modify the definition of gauge transformations as follows. The conjugation of $SO(3)$ can be extended to an $SO(3)$-action on $SU(2)$. A gauge transformation on $P$ is regarded as $SU(2)$-valued, i.e., a map $g: P \to SU(2)$ satisfying 
\beqn
g(p h) = h^{-1} g(p) h,\ \forall h \in SO(3),\ p \in P.
\eeqn
Since $\mf{so}(3) \cong \mf{su}(2)$, such $SU(2)$-valued gauge transformations act on $SO(3)$-connections in the usual way. Let ${\mc A}(P)$ be the space of smooth connections on $P$ and ${\mc G}(P)$ the space of $SU(2)$-valued smooth gauge transformations. The gauge equivalence class of a connection $A\in {\mc A}(P)$ is usually denoted by $[A]$.

In this paper, when there is no extra explanation, the sequential convergence of smooth objects are always regarded as convergence in the $C^\infty_{\rm loc}$-topology.

\subsection{Chern--Simons functional and the anti-self-dual equation}

\subsubsection{The Chern--Simons functional}\label{subsection21}

The instanton Floer cochain complex can be formally viewed as the Morse cochain complex for the Chern--Simons functional. Let $\M$ be a smooth oriented four-manifold. Let $\P \to \M$ be a smooth $SO(3)$-bundle. By Chern--Weil theory, the Pontryagin class can be represented by the differential form 
\beqn
p_1(\A) = - \frac{1}{2 \pi^2} {\rm tr} ( F_\A \wedge F_\A) \in \Omega^4(\M)\footnote{Here the trace operator ${\rm tr}: \mf{so}(3) \to {\mb R}$ is defined via the spin representation of $\mf{so}(3)$.}
\eeqn
for any smooth connection $\A$ on $\P$. When $\M$ is closed, the integral of this differential form is an integer and is a topological invariant. On the other hand, suppose $\M$ has a nonempty boundary $\partial \M \cong M$ where $M$ inherits a natural orientation. Denote $P = \P|_M$. Then for any smooth connection $A \in {\mc A}(P)$, for any smooth extension $\A$ of $A$ to the interior, the Chern--Weil integral
\beqn
\frac{1}{2} \int_\M \tr( F_\A \wedge F_\A)
\eeqn
only depends on the boundary restriction $A$. Denote this integral by 
\beqn
{\it CS}_\P(A)\in {\bf R}
\eeqn
which is invariant under gauge transformations on $P$ which extend to $\P$. More generally, if $M$ has several components $M_1, \ldots, M_k$ and $A_1, \ldots, A_k$ are the restrictions of $A$, then we also denote this action by
\beqn
{\it CS}_\P(A_1, \ldots, A_k).
\eeqn

For the purpose of this paper, we define the Chern--Simons functional on three-manifolds in a relative perspective. Let $N$ be a closed oriented three-manifold and $P\to N$ be an $SO(3)$-bundle. For a reference connection $A_0$ and any $A \in {\mc A}(P)$, define the {\bf relative} Chern--Simons action of $A$ to be 
\beqn
{\it CS}^{\rm rel}_{A_0}(A):= {\it CS}^{\rm rel}(A_0, A):= {\it CS}_{\bm P} (A_0, A)
\eeqn
where $\P = [0,1]\times P$ is the product bundle over $[0,1]\times N$. The relative action has the following more explicit expression: 
\beq\label{eqn21}
\int_M {\rm tr} \Big[F_{A_0} \wedge (A - A_0) +  \frac{1}{2} d_{A_0} (A - A_0) \wedge (A - A_0) + \frac{1}{3} ( A - A_0) \wedge (A - A_0) \wedge (A - A_0) \Big].
\eeq
Notice that ${\it CS}^{\rm rel}_{A_0}( A)$ is invariant under the action of the identity component ${\mc G}_0(P) \subset {\mc G}(P)$. For each $A \in {\mc A}(P)$ and a deformation $\alpha \in \Omega^1({\rm ad}P)$, the directional derivative of the Chern--Simons functional in the direction of $\alpha$ (which is independent of the choice of the reference connection $A_0$) is 
\beqn
( D_A {\it CS}^{\rm rel}_{A_0})(\alpha) = \int_M {\rm tr}( F_A \wedge \alpha).
\eeqn
Hence critical points of the Chern--Simons functional are flat connections.

\subsubsection{The anti-self-dual equation}

Suppose $\M$ is equipped with a Riemannian metric. A connection $\A \in {\mc A} (\P )$ is called an {\bf anti-self-dual connection} (ASD connection for short\footnote{When the domain is ${\bm R}^4$, or the product of the complex plane and a closed surface, or the product of the real line and a three-manifold with cylindrical end, we call an ASD connection an instanton.}) if 
\beqn
F_{\A} + *_4 F_{\A} = 0 \in \Omega^2( \M, {\rm ad} \P). 
\eeqn
Here $F_{\A}$ is the curvature of $\A$ and $*_4$ is the Hodge star operator on differential forms on $\M$. When $\M$ is noncompact, we also impose the finite energy condition. Then define the {\bf Yang--Mills functional}, also called the {\bf energy}, of a connection $ \A$ by 
\beqn
E ( \A):= \frac{1}{2} \int_{ \M} | F_{ \A} |^2.
\eeqn
Here the norm of the curvature is induced from the metric on $\M$ and the Killing metric on the Lie algebra. We always assume that ASD connections have finite energy. 

A particular case is when $\M = {\bm R}^4$ which is equipped with the standard Euclidean metric. We call an ASD connection over ${\bm R}^4$ an {\bf ${\bm R}^4$-instanton}.

The ASD equation can be viewed as the gradient flow equation of the Chern--Simons functional. A direct consequence of this perspective is the following energy identity (see for example \cite[Equation (2.7)]{Donaldson_YM}).

\begin{lemma}\label{lemma21}
Let $\M$ be a compact oriented four-manifold with boundary and $\P \to \M$ be an $SO(3)$-bundle. Then for any Riemannian metric on $\M$ and any ASD connection on $\P$ with respect to this metric, one has
\beqn
E (\A) = {\it CS}_{\P} ( \A|_{\partial \M}).
\eeqn
\end{lemma}

\subsubsection{Convergence and compactness}

We discuss the topology of the space of connections and recall the celebrated Uhlenbeck compactness theorem. Let $\M$ be a manifold and $\P \to \M$ be a principal bundle. The convergence of smooth $SO(3)$-connections $\A_i \in{\mc A}(\P)$ towards a limit $\A_\infty \in {\mc A}(\P)$ (in the $C^\infty_{\rm loc}$ topology) means for any precompact open subset $K \subset \M$, $\A_i|_K$ converges uniformly with all derivatives to $A_\infty|_K$. We define the more general notion of convergence in the Uhlenbeck sense. Suppose $\M$ is equipped with a Riemannian metric. Let $\M_i\subset \M$ be an exhausting sequence of open subsets, meaning that every compact subset of $\M$ is contained in $\M_i$ for sufficiently large $i$. Let $\P_i \to \M_i$ be $SO(3)$-bundles and $\A_i \in {\mc A}(\P_i )$ be a sequence of ASD connections on $\P_i$. Let $\P_\infty \to \M$ be an $SO(3)$-bundle and $\A_\infty \in {\mc A}(\P_\infty)$ be an ASD connection on $\P$. Let ${\bm m}_\infty$ be a positive measure on $\M$ supported at finitely many points. 

\begin{defn}\label{defn22}
We say that $\A_i$ converges to $(\A_\infty, {\bm m}_\infty)$ {\bf in the Uhlenbeck sense} if
\begin{enumerate}
\item the sequence of functions $|F_{\A_i}|^2$ converge as measures to $| F_{\A_\infty}|^2 +  2 {\bm m}_\infty$, and

\item there are bundle isomorphisms $\rho_i: \P_\infty \to \P_i$ over $\M_i \setminus  {\rm Supp} {\bm m}_\infty$ such that $\rho_i^*  \A_i$ converges to $\A_\infty$.
\end{enumerate}
\end{defn}

This notion of convergence is independent of the choices of representatives in their gauge equivalence classes. Therefore if $\A_i$ and $(\A_\infty, {\bm m}_\infty)$ satisfy conditions of Definition \ref{defn22}, we will say that $[\A_i]$ converges to $([\A_\infty], {\bm m}_\infty)$ in the Uhlenbeck sense. Further, the measure ${\bm m}_\infty$ in the limit, called the {\bf bubbling measure}, is nonzero at ${\bm x} \in \M$ if and only if a nontrivial ${\bm R}^4$-instanton bubbles off in the limit. We know that the masses of ${\bm m}_\infty$ are in $4\pi^2 {\bf Z}_+$. The convergence implies the following energy identity: for any compact subset ${\bm K} \subset \M$ containing the support of ${\bm m}_\infty$, there holds
\beqn
\lim_{i \to \infty} E ( \A_i; {\bm K}) = E( \A_\infty; {\bm K}) + \int_{{\bm K}} {\bm m}_\infty.
\eeqn

We summarize the celebrated Uhlenbeck compactness theorem as follows. 

\begin{thm}{\rm (}cf. \cite[Section 4.4]{Donaldson_Kronheimer}{\rm )} \label{thm23} Let $\M$, $\M_i$, $\P_i$, $\A_i$ be as above. Suppose the energy of $\A_i$ is uniformly bounded from above, i.e, 
\beqn
\limsup_{i \to \infty} E(\A_i) <  +\infty.
\eeqn
Then there exist a subsequence (still indexed by $i$ ), a positive measure ${\bm m}_\infty$ on $\M$ with finite support, an $SO(3)$-bundle $\P_\infty \to \M$, and an ASD connection $\A_\infty \in {\mc A}( \P_\infty)$, such that $\A_i$ converges to $(\A_\infty, {\bm m}_\infty)$ in the Uhlenbeck sense. In particular, there is a positive constant $\hbar >0$ with the following property: if $E(\A_i) < \hbar$, then a subsequence of $\A_i$ converges to a limiting ASD connection $\A_\infty$ without bubbling.
\end{thm}

\subsection{Product of two surfaces}\label{subsection22}

The Atiyah--Floer conjecture can be regarded as an effect of the reduction from 4D gauge theory to 2D sigma model observed by physicists \cite{BJSV}. Consider the special case that $\M = S \times \Sigma$ where $S$ and $\Sigma$ are oriented surfaces, equipped with a product metric. We assume for simplicity that $S$ is an open subset of either the complex plane ${\bm C}$ or the upper half plane ${\bm H}$, in which cases $S$ is equipped with the standard holomorphic coordinate $z = s + {\bm i}t$ and the standard flat metric. Let $Q \to \Sigma$ be an $SO(3)$-bundle and $\P  = S \times Q \to \M$ be the pullback of $Q$ via the projection $\M \to \Sigma$. Then we can write a connection $\A \in {\mc A}( \P )$ as 
\beqn
\A = d_S + \phi ds + \psi dt + B,
\eeqn
where $d_S$ is the exterior differential in $S$, $B: S \to {\mc A}(Q)$ is a smooth map,  and 
\beqn
\phi, \psi \in \Omega^0( S \times \Sigma, {\rm ad} Q ).
\eeqn

Now we look at the ASD equation with respect to the product metric. Introduce 
\begin{align}\label{eqn22}
&\ \A_s = \partial_s B - d_B \phi,\ &\ \A_t = \partial_t B - d_B \psi.
\end{align}
and
\begin{align}\label{eqn23}
&\ \kappa_\A = \partial_s \psi - \partial_t \phi + [\phi, \psi],\ &\ \mu_\A = F_B.
\end{align}
Then the ASD equation can be written in the local form 
\begin{align}\label{eqn24}
&\ \A_s + * \A_t = 0,\ &\ \kappa_\A + * \mu_\A = 0.
\end{align}
Here $*$ is the Hodge star on $\Sigma$. 

\begin{rem}
The equation \eqref{eqn24} can be viewed as an infinite dimensional version of the symplectic vortex equation introduced by Cieliebak--Gaio--Salamon \cite{Cieliebak_Gaio_Salamon_2000} and Mundet \cite{Mundet_thesis, Mundet_2003}. Indeed this perspective is one of the motivation of \cite{Cieliebak_Gaio_Salamon_2000} to propose the symplectic vortex equation. 
\end{rem}

\subsection{Holomorphic curves}

We recall a few basic facts about pseudoholomorphic maps from Riemann surfaces to almost complex manifolds. In this subsection, $(X, J)$ always denotes a {\it compact} almost complex manifold. We fix a Riemannian metric $h$ on $X$. Let $S$ be a smooth Riemann surface with possibly nonempty boundary. A {\bf $J$-holomorphic map} from $S$ to $X$ is a continuous map $u: S \to X$ which is smooth in the interior and satisfies the Cauchy--Riemann equation 
\beqn
\ov\partial_J u:= \frac{1}{2} \left( \frac{\partial u}{\partial s} + J \frac{\partial u}{\partial t} \right) d \bar z = 0.
\eeqn
Here $z = s + {\bm i} t$ is a local holomorphic coordinate on $S$. In this paper $S$ is always an open subset of either the complex plane ${\bm C}$ or the upper half plane ${\bm H}$. The energy of $u$ is 
\beqn
E(u; S):= \frac{1}{2} \| du \|_{L^2(S)} = \frac{1}{2} \int_S |du|_h^2 ds dt.
\eeqn
Here $|\cdot|_h$ is the norm with respect to the fixed Riemannian metric $h$. When $S$ is understood from the context, we abbreviate $E(u; S)$ by $E(u)$.

\subsubsection{Compactness}

We would like to define the notion of convergence of $J$-holomorphic maps over a bordered surface without any appropriate boundary condition. Let $S \subset \H$ be an open subset and $S_i \subset S$ be an exhausting sequence of open subsets.

\begin{defn}\label{defn25}
Let $u_i: S_i \to X$ be a sequence of $J$-holomorphic maps. Let $u_\infty: S \to X$ be another $J$-holomorphic map. Then we say that $u_i$ {\bf converges} to $u_\infty$ if $u_i$ converges to $u_\infty$ in $C_{\rm loc}^0(S)\cap C_{\rm loc}^\infty( S \cap  {\rm Int} \H)$. 
\end{defn}

Now we recall a compactness result about holomorphic maps on bordered surfaces without imposing a boundary condition and give a proof.

\begin{prop}\label{prop26}
Let $u_i: S_i \to X$ be a sequence of $J$-holomorphic maps such that for all compact sets $K \subset S$ there holds
\beqn
\limsup_{i \to \infty} E(u_i; K) < +\infty.
\eeqn
\begin{enumerate}

\item Assume there is no energy blowup in the interior, namely, for all $z \in S \cap {\rm Int} {\bm H}$, one has 
\beq\label{eqn25}
\lim_{r \to 0} \limsup_{i \to \infty} E(u_i; B_r(z) \cap S_i) = 0.
\eeq
Then there exist a subsequence (still indexed by $i$) and a holomorphic map $u_\infty: S \cap {\rm Int} {\bm H} \to X$ such that $u_i$ converges to $u_\infty$ in $C^\infty_{\rm loc}(S \cap {\rm Int} {\bm H})$. 

\item In addition, suppose for each $z \in \partial S$ there holds
\beq\label{eqn26}
\lim_{r \to 0} \limsup_{i \to \infty} {\rm diam}( u_i(B_r^+(z) \cap S_i)) = 0
\eeq
(here $B_r^+(z)$ is the intersection of the radius $r$ open disk $B_r(z) \subset {\bm C}$ with the upper half plane ${\bm H}$). Then the limit $u_\infty$ extends continuously to $S \cap \partial {\bm H}$ and $u_i$ converges to $u_\infty$ in the sense of Definition \ref{defn25}.
\end{enumerate}
\end{prop}

\begin{proof}
Part (a) is the classical Gromov compactness result (see for example \cite{McDuff_Salamon_2004} \cite{IS_compactness}). For part (b), we first show that $u_\infty$ has limits at all boundary points. Choose $z \in \partial S$. By \eqref{eqn26}, for any $\epsilon>0$, there exists an $r>0$ such that
\beqn
\limsup_{i \to \infty}  {\rm diam}( u_i(B_r^+(z))) < \epsilon.
\eeqn
Then for sufficiently large $i$, for all $z', z'' \in B_r^+(z) \cap S \cap {\rm Int} {\bm H}$, there holds
\beqn
d( u_i(z'), u_i(z'')) < \epsilon.
\eeqn
Since $u_i(z') \to u_\infty(z')$ and $u_i(z'') \to u_\infty(z'')$ as $i \to \infty$. It implies that 
\beqn
z', z''\in B_r^+(z) \Longrightarrow d( u_\infty(z'), u_\infty(z'')) < \epsilon.
\eeqn
Hence $u_\infty$ has limits at all boundary points. It is a similar argument to show that the boundary limits define a continuous extension of $u_\infty$ and $u_i$ converges to $u_\infty$ in $C^0_{\rm loc}(S)$. We leave the details to the reader. \end{proof}

\subsubsection{Boundary regularity}

Now assume that $L \subset X$ is a {\bf totally real} submanifold of $(X, J)$, i.e., for all $p \in L$, $T_p X = T_p L \oplus J(T_p L)$.

\begin{lemma}\label{regularity}
Let $S \subset {\bm H}$ be an open subset, $u: S \to X$ be a $J$-holomorphic map satisfying the boundary condition $u(\partial S) \subset L$. Then $u$ is also smooth on the boundary. 
\end{lemma}

\begin{proof}
The proof is essentially the same as \cite[Theorem B.4.1]{McDuff_Salamon_2004} where the map $u$ is assumed to be $W^{1,p}$ for some $p>2$. The same argument works if we assume $u$ is smooth in the interior and continuous along the boundary. Suppose ${\rm dim} X = 2n$. For any $p \in L$, there exists a coordinate chart $\psi: U \to {\mb R}^{2n}$ such that
\begin{align*}
&\ \psi(U \cap L) = \psi(U) \cap ({\mb R}^n \times \{0\}),\ &\ J_0 d\psi(x) = d\psi(x) J(x)\ \forall x\in U\cap L.
\end{align*}
Here $J_0$ is the standard complex structure on ${\mb R}^{2n}$. We may assume that the image of $u$ is contained in such a coordinate chart. Hence $u$ is a solution to 
\beqn
\partial_s u + J(u) \partial_t u = 0,\ u( \partial S) \subset {\mb R}^n \times \{0\}
\eeqn
where $J$ is a smooth almost complex structure on ${\mb R}^{2n}$ making ${\mb R}^n \times \{0\}$ totally real. Then the continuity of $u$ and the boundary condition implies that
\beqn
\int_S \langle \partial_s \phi + J(u)^T \partial_t \phi, u \rangle + \int_S \langle \phi, (\partial_t J(u)) u \rangle = 0
\eeqn
for all test functions $\phi \in C_0^\infty(S, {\mb R}^{2n})$ satisfying $\phi(S \cap \partial {\bm H}) \subset {\mb R}^n \oplus \{0\}$. Namely, $u$ is a weak solution to the Cauchy--Riemann equation with the totally real boundary condition. Then by \cite[Proposition B.4.9]{McDuff_Salamon_2004}, $u$ is smooth.
\end{proof}

\subsubsection{Immersed Lagrangian boundary condition}

We recall basic notions of pseudoholomorphic curves with an immersed Lagrangian boundary condition. We fix a compact symplectic manifold $(X, \omega)$. A {\bf Lagrangian immersion} is a smooth immersion $\iota: L \looparrowright X$ such that $\iota^* \omega = 0$ and such that ${\rm dim} X = 2{\rm dim} L$. We assume $L$ is compact. We assume that $L$ only {\bf has transverse double points}, which means the following.
\begin{enumerate}

\item For each $x \in \iota(L)$, $\# \iota^{-1}(x) \leq 2$. Each $x \in \iota(L)$ with $\# \iota^{-1}(x) = 2$ is called a {\bf double point} of the immersed Lagrangian. 

\item The map $\iota\times \iota: L \times L \to X \times X$ is transverse to the diagonal $\Delta_X\subset X \times X$ away from the diagonal $\Delta_L \subset L \times L$.
\end{enumerate}
%\beqn
%\iota \times \iota: L \times L \to X \times X
%\eeqn
%is assumed to be transverse to the diagonal $\Delta_X \subset X \times X$ away from the diagonal $\Delta_L \subset L \times L$. \textcolor{blue}{Each element of 
%\beqn
%(\iota\times \iota)^{-1}( \Delta_X) \setminus \Delta_L 
%\eeqn
%is called an {\bf ordered double point}, denoted by $(p, q)$. The image $\iota(p) = \iota(q)$ is called a {\bf double point}}
The compactness of $L$ implies that there are finitely many double points of $\iota(L)$. Elements of the set 
\beqn
R_L: = \{ (p, q) \in L \times L\ |\ \iota(p) = \iota(q),\ p \neq q \}.
\eeqn
are called {\bf ordered double points}. The map $(p, q) \mapsto (q, p)$ which preserves the set $\Delta_L \sqcup R_L$ is called the {\bf transpose}. 

Now we define the notion of holomorphic curves with boundary lying in the immersed Lagrangian. One can see that this notion coincides with that in \cite{Akaho_Joyce_2010} after ordering the set of marked points $W$. Fix an {\bf $\omega$-compatible} almost complex structure $J$ on $(X, \omega)$, namely, the bilinear form on $TX$ defined by 
\beqn
g(\xi, \xi'): = \omega(\xi, J\xi')
\eeqn
is a Riemannian metric on $X$. %We also assume that $L$ is {\bf totally real} with respect to $J$, namely, for every $x \in L$, 
%\beqn
%T_{\iota(x)} X = d \iota( T_x L) \oplus J d\iota(T_x L).
%\eeqn

\begin{defn}(cf. \cite[Definition 4.2]{Akaho_Joyce_2010})\label{defn28}
Let $S$ be a Riemann surface with possibly nonempty boundary. A {\bf marked $J$-holomorphic map} from $S$ to $X$ with boundary in $\iota(L)$ is a triple
\beqn
{\bm u}= (u, W, \gamma)
\eeqn
where $u: S \to  X$ is a $J$-holomorphic map with $u(\partial S) \subset \iota(L)$, $W \subset S$ is a finite subset, and $\gamma: \partial S \setminus W \to L$ is a continuous map, satisfying the following {\bf boundary condition}:
\beqn
u|_{\partial S \setminus W} = \iota \circ \gamma.
\eeqn

Given such a marked $J$-holomorphic map, for each $w \in W \cap \partial S$, there is a local holomorphic coordinate chart $U_w \cong \phi_w(U_w) \subset {\bm H}$. The boundary condition implies that the limit
\beqn
{\rm ev}_w({\bm u}):= ( \lim_{s \to 0-} \gamma(\phi_w^{-1}(s)), \lim_{s \to 0+} \gamma( \phi_w^{-1}(s))  ) \in \Delta_L \cup R_L\subset L \times L
\eeqn
exists. We call this limit the {\bf evaluation} of ${\bm u}$ at $w$. If ${\rm ev}_w({\bm u}) \in R_L$, we call $w$ a {\bf switching point} of ${\bm u}$. On the other hand, the evaluation of ${\bm u}$ at an interior marking $w \in W \cap {\rm Int}S$ is the value ${\rm ev}_w({\bm u}) = u(w) \in X$. 
\end{defn}

A corollary of the boundary regularity result Lemma \ref{regularity} is that the boundary map $\gamma$ is necessarily smooth.

\begin{lemma}\label{lemma29}
Let ${\bm u} = (u, W, \gamma)$ be as in Definition \ref{defn28}. Then $\gamma$ is smooth.
\end{lemma}

\begin{proof}
As $J$ is $\omega$-compatible, the immersion $\iota: L \immerse X$ is totally real. Then as the immersion has only transverse double points, locally we can view $u$ as a $J$-holomorphic map with boundary lying in an embedded totally real submanifold and $\iota$ is the boundary restriction of $u$. Hence by Lemma \ref{regularity}, $u$ is smooth near boundary points which are not in $W$. Hence $\gamma$ is smooth.
\end{proof}

We also allow a nontrivial Dirac measure in the datum of a holomorphic curve. More precisely, if $m: W \to {\bf R}_+$ is a function, regarded as a positive measure on $S$ whose support is contained in $W$, then we call the tuple 
\beqn
\tilde {\bm u} = (u, W, \gamma, m)
\eeqn
a {\bf marked holomorphic curve with mass}. When $\partial S = \emptyset$, we simplify the notation as $\tilde {\bm u} = (u, W, m)$.

\subsection{Flat connections on three-manifolds}\label{subsection24}

Now we introduce the basic assumptions on the three-manifolds. Let $M_0$ be a connected, oriented three-manifold with a nonempty and not necessarily connected  boundary $\partial M_0 \cong \Sigma$. Let $M$ be the completion, i.e., 
\beq\label{eqn27}
M:= M_0 \cup ( [0, +\infty) \times \Sigma),
\eeq
where the two parts are glued along the common boundary. Then we always identify $M_0$ with a closed subset of $M$. Let $P \to M$ be an $SO(3)$-bundle. Let $P_0 \to M_0$ be the restriction of $P$ to $M_0 \subset M$ and 
\beqn
Q \to \Sigma
\eeqn
be the restriction of $P_0$ to the boundary. Let $L_M$ be the moduli space of gauge equivalence classes of flat connections on $P_0$, i.e., 
\beqn
L_M:= \Big\{ A \in {\mc A}(P_0)\ |\ F_A = 0 \Big\}/ {\mc G}(P_0).
\eeqn
Let $R_\Sigma$ be the moduli space of gauge equivalence classes of flat connections on $Q$, i.e., 
\beqn
R_\Sigma:= \Big\{ B \in {\mc A}(Q)\ |\ F_B = 0\Big\}/ {\mc G}(Q).
\eeqn
Both $L_M$ and $R_\Sigma$ have natural topology. There is a natural continuous map 
\beqn
\iota: L_M \to R_\Sigma
\eeqn
induced by boundary restriction.

\subsubsection{Transversality assumption}

Now we consider moduli spaces of flat connections on $M$ and $\Sigma$. We impose certain extra conditions to guarantee that these moduli spaces are smooth. For any flat connection $A \in {\mc A}(P_0 )$, the covariant derivative $d_A$ makes ${\rm ad} P_0 $ a flat bundle with a twisted de Rham complex 
\beqn
\xymatrix{ \Omega^0( M_0, {\rm ad}P_0) \ar@{^{(}->}[r]^{d_A } & \Omega^1( M_0, {\rm ad} P_0 ) \ar[r]^{d_A} & \Omega^2 ( M_0, {\rm ad} P_0 ) \ar@{->>}[r]^{d_A } & \Omega^3( M_0, {\rm ad} P_0)}.
\eeqn
Similarly, when $B$ is a flat connection on $Q$, there is a complex 
\beqn
\xymatrix{ \Omega^0(\Sigma, {\rm ad} Q) \ar@{^{(}->}[r]^{d_B} & \Omega^1( \Sigma, {\rm ad} Q) \ar@{->>}[r]^{d_B} & \Omega^2(\Sigma, {\rm ad} Q)}.
\eeqn
When $A |_{\Sigma} = B$, one can form the relative complex with 
\beqn
\Omega^k_{\rm rel}({\rm ad} P_0) = \Omega^k( M_0, {\rm ad} P_0 ) \oplus \Omega^{k-1}( \Sigma, {\rm ad} Q)
\eeqn
and differential $d_{A, B}$, which is defined as 
\beqn
d_{A, B} \left[ \begin{array}{c} \alpha\\ h \end{array} \right] = \left[ \begin{array}{cc} d_A & 0 \\  - r^* & d_B \end{array} \right] \left[ \begin{array}{c} \alpha\\ h \end{array} \right].
\eeqn
Here $r^*: \Omega^k( M_0, {\rm ad} P_0 ) \to \Omega^k( \Sigma, {\rm ad} Q)$ is the pullback. Then there is a long exact sequence
\beq\label{eqn28}
\xymatrix{ \cdots H^0 (d_A ) \ar[r] & H^0 ( d_B ) \ar[r] & H^1( d_{A, B} ) \ar[r] & H^1( d_A ) \ar[r] & H^1( d_B ) \cdots}
\eeq

We assume the following conditions throughout this paper. 

\begin{hyp}\label{hyp210}  The three-manifold with boundary $M_0$ and the $SO(3)$-bundle $P_0 \to M_0$ satisfy the following conditions. 
\begin{enumerate}
\item For any flat connection $B$ on $Q$, $H^0(d_B) \oplus H^2(d_B)$ vanishes. This implies that $R_\Sigma$ is a smooth manifold with
\beqn
{\rm dim} R_\Sigma = - 3 \chi(\Sigma).
\eeqn

\item For any flat connection $A$ on $P_0$, $H^2(d_A) = 0$. It follows that the moduli space $L_M$ of flat connections on $P$ is smooth and
\beqn
{\rm dim} L_M = - \frac{3}{2} \chi(\Sigma)
\eeqn
(see \cite[Equation (2.3)]{Fukaya_2018}, also the index formula of Yoshida \cite[Lemma 4.2]{Yoshida_1991} for the $SU(2)$ case).

\item For any flat connection $A$ on $P_0$ whose boundary restriction is $B$, Poincar\'e duality and item (b) above imply that $H^1(d_{A, B}) \cong H^2(d_A) = 0$. Then it follows from the long exact sequence \eqref{eqn28} that the map $H^1(d_A) \to H^1(d_B)$ is injective, hence the natural map $L_M \to R_\Sigma$ is an immersion (cf. \cite[Proof of lemma 2.4]{Fukaya_2018}. We assume that the immersion $\iota: L_M \immerse R_\Sigma$ has transverse double points.

%For any flat connection $A$ on $P_0$ whose boundary restriction is $B$, the map $H^1( d_A ) \to H^1(d_B)$ is injective. This implies that the natural map $L_M \to R_\Sigma$ is an immersion. 

%\item The immersion $\iota: L_M \immerse R_\Sigma$ has transverse double points. 
\end{enumerate}
\end{hyp}

\begin{lemma}
Item (a) of Hypothesis \ref{hyp210} holds if and only if $Q$ is nontrivial over each connected component of $\Sigma$. In this case $\Sigma$ has an even number of connected components. 
\end{lemma}

\begin{proof}
If $Q|_{\Sigma_i} \cong \Sigma_i \times SO(3)$ for some component $\Sigma_i \subset \Sigma$, then for a flat connection $B$ on $Q$ which is trivial on $\Sigma_i$, the constant sections of ${\rm ad} (Q|_{\Sigma_i}) \cong \Sigma_i \times \mf{so}(3)$ are contained in $H^0(d_B)$. On the other hand, suppose $Q$ is nontrivial over a component $\Sigma_0 \subset \Sigma$, we claim that $H^0(d_{B_0}) \oplus H^2(d_{B_0}) = 0$ where $B_0$ is the restriction of a flat connection $B$ on $Q$ to $\Sigma_0$. By Poincar\'e duality, one only needs to show $H^0(d_{B_0}) = 0$. Suppose on the contrary that $H^0(d_{B_0}) \neq 0$. Then there exists a nonzero section $\xi \in \Gamma({\rm ad} Q|_{\Sigma_0})$ such that $d_{B_0} \xi = 0$. Then $\xi$ is parallel and hence reduces the structure group of $Q|_{\Sigma_0}$ from $SO(3)$ to $SO(2)$. As a flat $SO(2)$-bundle over $\Sigma_0$ is topologically trivial, this contradicts the assumption that $Q|_{\Sigma_0}$ is nontrivial. Therefore if $Q$ is nontrivial over every component of $\Sigma$, one has $H^0(d_B) \oplus H^2(d_B) = 0$ for all flat connection $B$ on $Q$.

%it is well-known that any flat connection $B$ on $Q$ has no stabilizer in the group of $SU(2)$-valued gauge transformations except for $\pm 1$. A typical argument is to identify connections on $Q$ with connections on certain $U(2)$-bundle over $\Sigma$ with a fixed determinant (see for example \cite[Proposition 1.6]{Donaldson_Braam}). In particular there is no nontrivial infinitesimal stabilizer and hence $H^0(d_B) = 0$. By Poincar\'e duality $H^2(d_B)$ also vanishes. Hence (a) is equivalent to the nontriviality of $Q$ over each $\Sigma_i$. 

To show that $\Sigma$ has an even number of connected components, consider the exact sequence in ${\bf Z}_2$ coefficients
\beqn
\xymatrix{ \cdots \ar[r] & H^2(M_0) \ar[r] & H^2(\partial M_0) \ar[r] & H^3(M_0, \partial M_0) \ar[r] & \cdots}
\eeqn
It follows that the second Stiefel--Whitney class $w_2(Q)$, which is the image of $w_2(P) \in H^2(M_0)$, is sent to zero in $H^3(M_0, \partial M_0) \cong {\bf Z}_2$. On the other hand, since $Q$ is nontrivial over each component $\Sigma_i$, $w_2(Q)$ restricts to the generator of $H^2(\Sigma_i; {\bf Z}_2)$, while each generator is sent to the generator of $H^3(M_0, \partial M_0) \cong {\bf Z}_2$. Hence $\Sigma$ has an even number of connected components. 
\end{proof}

\begin{rem}
In general Item (b) and (c) of Hypothesis \ref{hyp210} may not hold. However, one can perturb the Chern--Simons functional by the so-called holonomic perturbation supported away from the boundary, so that critical points of the perturbed Chern--Simons functional (i.e., certain perturbed flat connection) are non-degenerate in the Bott sense so Item (b) holds (see the $SU(2)$ case in \cite{Herald_1994}); at the same time, the perturbation can also be made such that the immersion $\iota: L_M \immerse R_\Sigma$ has only transverse self-intersections.
\end{rem}

Pairs of flat connections on two three-manifolds with boundary induce certain flat connections on a closed three-manifold. Let $N^-$, $N^+$ be connected oriented three-manifolds with boundary such that 
\beqn
\partial N^- \cong \Sigma \cong \partial (N^+)^{\rm op}.
\eeqn
Here $(N^+)^{\rm op}$ is a copy of $N^+$ with the reversed orientation. Let $P_{N^-} \to N^-$, $P_{N^+} \to N^+$, and $Q \to \Sigma$ be $SO(3)$-bundles such that
\beqn
P_{N^-}|_{\partial N^-} \cong Q \cong P_{N^+}|_{\partial N^+}.
\eeqn
Suppose $(P_{N^-}, N^-)$ and $(P_{N^+}, N^+)$ both satisfy Hypothesis \ref{hyp210}. Then one obtains Lagrangian immersions
\begin{align*}
&\ \iota^-: L_{N^-} \immerse R_\Sigma,\ &\ \iota^+: L_{N^+} \immerse R_\Sigma.
\end{align*}
We assume in addition that 

\begin{hyp}\label{hyp213}
The two three-manifolds with boundary $N^-$, $N^+$ and the $SO(3)$-bundles $P_{N^\pm} \to N^\pm$ satisfy the following condition. 
\begin{itemize}
\item The immersions $\iota^-: L_{N^-} \immerse R_\Sigma$ and $\iota^+: L_{N^+} \immerse R_\Sigma$ intersect cleanly.
\end{itemize}
\end{hyp}

Define a closed three-manifold together with an $SO(3)$-bundle as follows. Let $N$ be the closed three-manifold defined by 
\beq\label{eqn29}
\check N:= N^- \cup N^+:= N^- \cup N^+.
\eeq
Here we identify the common boundaries $\partial N^- \cong \partial (N^+)^{\rm op}$. The bundles $P_{N^\pm}$ can be glued similarly to give an $SO(3)$-bundle
\beq\label{eqn210}
P_{\check N}:= P_{\check N} |_{N^-} \cup P_{\check N} |_{N^+}:= P_{N^-} \cup P_{N^+}.
\eeq
Let $L_{\check N}$ be the moduli space of gauge equivalence classes of flat connections on $P_{\check N}$. Then $L_{\check N}$ is a compact manifold (with possibly varying dimensions) with a diffeomorphism
\beqn
L_{\check N} \cong (\iota^- \times \iota^+)^{-1} (\Delta_{R_\Sigma}) \subset L_{N^-} \times L_{N^+}.
\eeqn

In practice we need a slightly different construction. Define 
\beqn
N:= N^- \cup [0, \pi] \times \Sigma \cup N^+
\eeqn
where we glue $\partial N^-$ with $\{0\}\times \Sigma$ and glue $\partial N^+$ with $\{\pi\}\times \Sigma$. We denote $N^{\rm neck} = [0, \pi]\times \Sigma$. Obviously $N$ is diffeomorphic to $\check N$. Then one can construct a similar bundle $P_N \to N$ whose restriction to the neck region is $[0, \pi]\times Q$. Notice that this is only a special case of the construction of $P_{\check N} \to \check N$. Indeed if we set $N_\times^- = N^- \sqcup N^+$, $N_\times^+ = N^{\rm neck}$. Then the manifold obtained by gluing $N_\times^-$ and $N_\times^+$ along the common boundary is exactly $N$.

\begin{rem}\label{rem214}
There are two special situations when Hypothesis \ref{hyp213} is satisfied. The first special case is when $N^- \cong (N^+)^{\rm op} \cong M_0$ and $P_{N^-} \cong P_{N^+} \cong P_0$. In this case $N$ is diffeomorphic to the doubling of $M_0$, denoted by $M^{\rm double}$ and one has
\beqn
L_N \cong L_{\check N} \cong \Delta_{L_M} \cup R_{L_M}\subset L_M \times L_M.
\eeqn
The second special case is when $N^- \cong (N^+)^{\rm op} \cong [0, \pi] \times \Sigma$ whose boundary is two copies of $\Sigma$ and $P_{N^-} \cong P_{N^+} \cong [0, \pi]\times Q$. In this case $N$ is diffeomorphic to $S^1 \times \Sigma$ and $L_N$ is diffeomorphic to $R_\Sigma$. 
\end{rem}

It is convenient to allow certain piecewise smooth connections. Define
\beqn
{\mc A}^{\rm p.s.}(P_{\check N} ):= \big\{  (A_-, A_+) \in {\mc A}(P_{N^-} \sqcup P_{N^+}) \ |\  A_-|_{\partial N^-} = A_+|_{\partial N^+} \big\}
\eeqn
whose elements are called {\bf piecewise smooth connections}. Define the space of piecewise smooth gauge transformations ${\mc G}^{\rm p.s.}(P_{\check N} )$ in a similar way. Then one has
\beqn
L_{\check N} \cong \big\{ A \in {\mc A}^{\rm p.s.}(P_{\check N} )\ |\ F_A = 0 \}/ {\mc G}^{\rm p.s.}(P_{\check N}).
\eeqn

\subsubsection{Almost flat connections on $M_0$}

Now we turn to the analytical part of the gauge theory. For all three-manifolds with boundary considered in this paper, say $M_0$ for example, we fix a Riemannian metric on $M_0$ such that a neighborhood of the boundary is isometric to $(-\epsilon, 0] \times \partial M_0$. The metric on $M_0$ induces a metric on its completion \eqref{eqn27} which is of the product type on the cylindrical end. When discussing a pair of three-manifolds $N^-$ and $N^+$ sharing the same boundary, we assume the boundary restrictions of the metrics are isometric. The pair of metrics induce a metric on the closed manifold $N$ defined by \eqref{eqn210} which is of product type over the neck.

The differentiation of gauge fields depends on the choice of a covariant derivative. For any $SO(3)$-bundle $P$ over a Riemannian manifold $U$, for any connection $A \in {\mc A}(P)$, let $W_A^{k,p}$ be the corresponding Sobolev norm on sections of ${\rm ad} P$ or $\Lambda^1 \otimes {\rm ad} P$. For example, if $s \in \Omega^1(U, {\rm ad} P)$, then 
\beqn
\| s \|_{W_A^{k,p}(U)}:= \sum_{l=0}^k \| (\nabla^A)^l s \|_{L^p(U)}.
\eeqn
Here $\nabla^A: \Omega^1(U, {\rm ad} P) \to \Omega^1(U, T^* U \otimes {\rm ad} P)$ is the covariant derivative induced from $A$ and the Levi--Civita connection on $U$. However, the Banach topology on these spaces such as $W^{k,p}_A(U, {\rm ad} P)$ do not depend on the choice of a smooth connection $A$. Hence we often write $W^{k,p}(U, {\rm ad} P)$ instead of $W^{k,p}_A(U, {\rm ad} P)$ when the norm is not emphasized. 

The following lemma shows that near an almost flat connection on the three-manifold with boundary there is always a flat connection. It essentially follows from the transversality assumption Hypothesis \ref{hyp210} and the implicit function theorem. 

\begin{lemma}\label{lemma215}
Let $p \geq 2$. There exist $\epsilon = \epsilon_p>0$ and $C = C_p>0$ satisfying the following properties. Let $A$ be a smooth connection on $P_0 \to M_0$ with 
\beqn
\| F_A\|_{L^p(M_0)} \leq \epsilon.
\eeqn
Then there exists a flat connection $A^*$ on $P_0$ of regularity $W^{1,p}$ satisfying 
\beqn
\| A - A^* \|_{W_A^{1,p}(M_0)} \leq C \| F_A \|_{L^p(M_0)}.
\eeqn
\end{lemma}

\begin{proof}
Consider the Banach space 
\begin{multline*}
W^{1,p}_{A,\partial}( M_0, (\Lambda^0 \oplus \Lambda^1) \otimes {\rm ad} P_0) \\
 = \Big\{ (f, \alpha) \in W_A^{1,p}(M_0, (\Lambda^0 \oplus \Lambda^1) \otimes {\rm ad} P_0)\ \left| \ * \alpha|_{\partial M_0} = 0, f|_{\partial M_0} = 0  \Big\}. \right.
 \end{multline*}
Consider the linear operator
\beqn
D_A: W^{1,p}_{A, \partial} ( M_0, (\Lambda^0\oplus \Lambda^1) \otimes {\rm ad} P_0)  \to L^p(M_0,( \Lambda^0\oplus \Lambda^1) \otimes {\rm ad}P_0) 
\eeqn
defined by 
\beqn
D_A (f, \alpha) = ( d_A^* \alpha, * d_A \alpha + d_A f).
\eeqn
This is Fredholm with index $-\frac{3}{2} \chi(\Sigma)$. We claim that there exist $\epsilon = \epsilon_p > 0$ and $C = C_p>0$ such that when $\| F_A \|_{L^p(M_0)} \leq \epsilon$, $D_A$ is surjective and there is a right inverse $Q_A$ with $\| Q_A \| \leq C_p$. Suppose this is not the case, then there exist a sequence of connections $A_i$ with $\| F_{A_i}\|_{L^p} \to 0$ but $D_{A_i}$ is not surjective. Then by the weak Uhlenbeck compactness theorem in three dimensions (see \cite[Theorem A]{Wehrheim_Uhlenbeck}, which applies to both closed manifolds and manifolds with boundary) implies that a subsequece converges modulo gauge to a flat connection $A_\infty$ on $M_0$, and the convergence is weakly in $W^{1,p}$. This implies that $D_{A_i}$ converges to $D_{A_\infty}$ in operator norm (notice the Sobolev embedding $W^{1,2} \hookrightarrow L^6$ in dimension three). However, by Hypothesis \ref{hyp210}, $D_{A_\infty}$ is surjective as its kernel is the tangent space of $L_M$ at $[A_\infty]$. Hence for $i$ sufficiently large, $D_{A_i}$ is also surjective. This contradiction means as long as $\| F_A \|_{L^p}$ is sufficiently small, $D_A$ is surjective. Same argument further guarantees the existence of a right inverse with bounded norm as $L_M$ is compact. Then one can apply the implicit function theorem (see for example \cite[Proposition A.3.4]{McDuff_Salamon_2004}). More precisely, $D_A$ is the linearization of the nonlinear map ${\mc F}^A (f, \alpha) = ( d_A^* \alpha, * F_{A+\alpha} + d_{A+\alpha} f)$ at $(0, 0)$. One has
\beqn
{\mc F}^A (0, 0) = (0, * F_A)
\eeqn
whose norm is small. Then the implicit function theorem implies that there is a nearby connection $A^* = A + \alpha$ and $f$ such that 
\beqn
{\mc F}^A (f, \alpha) = (d_A^* \alpha, * F_{A+\alpha} + d_{A+\alpha} f ) = (0, 0).
\eeqn
Furthermore, by the Bianchi identity $d_{A+\alpha}^* (* F_{A+\alpha}) = 0$, integration by parts, and the boundary condition $f|_{\partial M_0} = 0$, we see that $*F_{A+\alpha}$ and $d_{A+\alpha} f$ are $L^2$-orthogonal. Hence $F_{A+\alpha} = 0$. So $A^*$ is a flat connection and the implicit function theorem implies that
\beqn
\| A^* - A \|_{W_A^{1,p}(M_0)} \leq C \| {\mc F}^A(0, 0)\| = C \| F_A \|_{L^p(M_0)}. \qedhere
\eeqn
\end{proof}

\begin{rem}
Throughout this paper, we adopt the convention that $C$ and $\epsilon$ represent constants which are allowed to vary from line to line. 
\end{rem}

\subsection{The representation variety}

We recall basic facts about the moduli space of flat connections over a surface, which we often call by the name representation variety. Let $\Sigma$ be a closed surface and $Q \to \Sigma$ be the $SO(3)$-bundle which is nontrivial over each connected component of $\Sigma$. The space ${\mc A}(Q)$ of smooth connections on $Q$ is an affine space modelled on $\Omega^1({\rm ad} Q)$. There is a symplectic form defined by 
\beq\label{eqn211}
\omega_{{\mc A}(Q)} (\alpha, \beta) = - \int_\Sigma {\rm tr} (\alpha \wedge \beta),\ \alpha, \beta \in \Omega^1({\rm ad} Q).
\eeq
The conformal class of the metric on $\Sigma$ defines an almost complex structure, i.e., 
\beqn
J_{{\mc A}(Q)} \alpha = * \alpha
\eeqn
where the Hodge star operator on 1-forms on $\Sigma$ only depends on the complex structure. $\omega_{{\mc A}(Q)}$ and $J_{{\mc A}(Q)}$ make ${\mc A}(Q)$ an (infinite dimensional) K\"ahler manifold. 

The space of gauge transformations ${\mc G}(Q)$ acts on ${\mc A}(Q)$ by pulling back connections. The action is Hamiltonian, with a moment map 
\beqn
\mu(B) = - F_B \in \Omega^2( {\rm ad} Q) \cong \Big( {\rm Lie} {\mc G}(Q) \Big)^*
\eeqn
The representation variety associated to $Q$ can be identified with the symplectic quotient 
\beqn
R_\Sigma \cong \mu^{-1}(0) / {\mc G}(Q).
\eeqn
The representation variety $R_\Sigma$ inherits a K\"ahler structure from the symplectic form $\omega_{{\mc A}(Q)}$ and the complex structure $J_{{\mc A}(Q)}$ (see \cite{Goldman_1984}). The associated K\"ahler metric on $R_\Sigma$ is called the $L^2$-metric. 

Under Hypothesis \ref{hyp210}, the immersion $\iota: L_M \immerse R_\Sigma$ is Lagrangian. Indeed, for any $a \in L_M$ represented by a flat connection $A \in {\mc A}_{\rm flat}(P_0)$, tangent vectors $\xi \in T_a L_M \cong H^1(d_A)$ are represented by $\alpha \in \Omega^1({\rm ad} P_0)$ satisfying $d_A \alpha = 0$. Then for any pair of tangent vectors $\xi, \xi'\in T_a L_M$ represented by $\alpha, \alpha' \in {\rm ker} d_A$ respectively, one has
\beqn
- \omega( \iota_* \xi, \iota_* \xi') = \int_\Sigma {\rm tr}( \alpha|_\Sigma\wedge \alpha'|_\Sigma) = \int_{M_0} d( {\rm tr}(\alpha \wedge \alpha') ) = \int_{M_0} {\rm tr}( d_A \alpha \wedge \alpha' - \alpha \wedge d_A \alpha') = 0.
\eeqn
(See also \cite{Yoshida_1991}\cite{Herald_1994} for discussions of the Lagrangian property in the $SU(2)$ case.)

\subsubsection{Projection onto the representation variety}

We would like to show that, analogous to the finite-dimensional situation, there is a natural map (which will be called the Narasimhan--Seshadri map) sending ``nearly flat'' connections on the $SO(3)$-bundle $Q \to \Sigma$ to a flat connection defined using complex gauge transformations. Moreover, we show the good behavior of this map by proving some estimates.

We briefly explain the idea of complex gauge transformations in the finite-dimensional setting. Suppose $V$ is a K\"ahler manifold acted by a compact Lie group $K$ with a moment map $\mu: V \to {\mf k}^*$. Suppose $0\in {\mf k}^*$ is a regular value of $\mu$. Because the $K$-action preserves the complex structure, the action extends to the complexification $G = K^{\mb C}$. Moreover, for each $x \in V$ sufficiently close to $\mu^{-1}(0)$, there exists a unique $\xi \in {\mf k}$ with $|\xi|$ small such that the ``purely imaginary'' translation $e^{{\bm i} \xi} x$ of $x$ lies in the zero locus of $\mu$. This defines a smooth $K$-equivariant map sending $x$ to $e^{{\bm i} \xi} x$. After composing with the projection $\mu^{-1}(0) \to \mu^{-1}(0)/K$ it becomes a holomorphic map onto the GIT quotient. If $\mu^{-1}(0)$ is compact, then one can prove estimates such as $|\xi| \leq C |\mu(x)|$ and $|d\xi/dx|\leq C$.

To discuss the infinite-dimensional analogue, we first need to extend gauge transformations to the complexified version. Similar discussions can be found in many literature. Here we follow \cite[Section 2]{Fukaya_1998} which treats the $SO(3)$-case. (For the $SU(2)$-case, see for example \cite[Section 6.1]{Donaldson_Kronheimer}.) Let $\hat Q \to \Sigma$ be the principal $U(2)$-bundle whose first Chern class is the fundamental class of $\Sigma$. The $SO(3)$-bundle associated to $\hat Q$ and the representation $U(2) \to U(2)/U(1) \cong SO(3)$ can then be identified with $Q$. Fix a connection $\nabla^{\det}$ on the determinant line bundle $\det \hat Q \to \Sigma$. As one has the $U(2)$-equivariant decomposition ${\mf u}(2) \cong \mf{so}(3) \oplus \mf{u}(1)$, each connection $B \in {\mc A}(Q)$ can be identified with a connection $\hat B = B \oplus  \nabla^{\det} \in {\mc A}(\hat Q)$, where the latter induces the connection $\nabla^{\det}$ on $\det \hat Q$. Let ${\mc A}(\hat Q, \nabla^{\det}) \subset {\mc A} (\hat Q)$ be the subset of those $U(2)$-connections.

From now on we shall not distinguish $B \in {\mc A}(Q)$ with the unitary connection $\hat B = B \oplus \nabla^{\rm det}$. Let $\hat E \to \Sigma$ be the rank 2 Hermitian vector bundle associated to $\hat Q$. Each connection $\hat B$ on $\hat Q$ induces a $\ov\partial$-operator 
\beqn
\ov\partial_{\hat B}: \Omega^0(\hat E) \to \Omega^{0,1}(\hat E)
\eeqn
(i.e. complex linear operators satisfying the Leibniz rule) on $\hat E$ by taking the $(0,1)$-part of the covariant derivative associated to $\hat B$. Let ${\mc D}''(E)$ be the space of $\ov\partial$-operators on $\hat E$ and ${\mc D}''(\hat E, \nabla^{\det}) \subset {\mc D}''(E)$ the subset of those which induce the $(0,1)$-part of $\nabla^{\rm det}$. Then one has the isomorphisms
\beqn
\Psi: {\mc A}(Q) \cong {\mc A}(\hat Q, \nabla^{\det}) \cong {\mc D}''(\hat E, \nabla^{\rm det}),\ B \mapsto \ov\partial_{\hat B}.
\eeqn
The complexification of the group of gauge transformations on $Q$ can be regarded as 
\beqn
{\mc G}(Q)^{\mb C} = \Gamma( \Sigma, Q \times_{SO(3)} SU(2))^{\mb C} = \Gamma( \Sigma, Q \times_{SO(3)} SL(2; {\mb C})).
\eeqn
Hence each $g \in {\mc G}(Q)^{{\mb C}}$ can be regarded as an automorphism of the complex vector bundle $\hat E$ whose determinant is the identity. The complex gauge transformation on ${\mc A}(Q)$ is then defined as 
\beqn
g^* B:= \Psi^{-1} ( g^{-1} \ov\partial_{\hat B} g).
\eeqn
In particular, a {\bf purely imaginary} gauge transformation, which has the form 
\beqn
g = e^{{\bm i} h},\ {\rm where}\ h \in \Gamma({\rm ad}Q)
\eeqn
acts by
\beqn
g^* B = B - \left( \partial_B e^{{\bm i} h} \right) e^{-{\bm i} h} + e^{-{\bm i} h} \left( \ov\partial_B e^{{\bm i} h} \right).
\eeqn
Here $\partial_B + \ov\partial_B$ is the covariant derivative on $E$ associated to $B$. The infinitesimal version of this action is 
\beqn
h \mapsto - * d_B h.
\eeqn
The curvature of the transformed connection is (see \cite[page 5]{Donaldson_1985} or \cite[page 11]{Duncan_2012})
\beq\label{eqn212}
F_{g^* B} = e^{-{\bm i} h} \left( F_B - \ov\partial_B \big( e^{2{\bm i} h} (\partial_B e^{- 2 {\bm i} h}) \big) \right) e^{{\bm i} h}.
\eeq

Following the terminology of Duncan \cite{Duncan_thesis, Duncan_2012}, we define a nonlinear map which assign to each almost flat connection on the surface to a flat connection via a unique imaginary gauge transformation. For $p>1$ and $\epsilon>0$, define 
\beqn
{\mc A}^{1,p}_\epsilon(Q) = \Big\{ B \in {\mc A}^{1,p}(Q)\ |\ \| F_B \|_{L^p(\Sigma)} < \epsilon\Big\}
\eeqn
and 
\beqn
{\mc A}^{1,p}_{\rm flat}(Q) = \Big\{ B \in {\mc A}^{1,p}(Q)\ | \ F_B = 0 \Big\}.
\eeqn

\begin{lemma}\label{lemma217}
Let $p \geq 2$. There exist positive constants $\epsilon = \epsilon_p$, $\epsilon' = \epsilon_p'$, and $C = C_p$ satisfying the following conditions. For each $B \in {\mc A}_\epsilon^{1,p}(Q)$, there exists a unique purely imaginary gauge transformation of the form $g = e^{ {\bm i} h_B}$ where $h_B \in W_B^{2,p}(\Sigma, {\rm ad}Q)$ such that 
\begin{align}\label{eqn213}
&\ g^* B \in {\mc A}_{\rm flat}^{1,p}(Q),\ &\ \| h_B \|_{W^{2,p}_B} < \epsilon'.
\end{align}
Moreover, $h_B\in W_{B}^{2,p}(\Sigma, {\rm ad}(Q))$ is a $C^1$ function of $B \in {\mc A}_{\epsilon}^{1,p}(Q)$ and there holds 
\beq\label{eqn214}
\| h_B \|_{W_B^{2,p}(\Sigma)} \leq C \| F_B \|_{L^p(\Sigma)}.
\eeq
\end{lemma}

\begin{proof}
We use the implicit function theorem. Consider the map
\beqn
W_B^{2,p}(\Sigma, {\rm ad}Q ) \ni h \mapsto  * F_{g^* B} \in L^p(\Sigma, {\rm ad}Q )
\eeqn
where $g = e^{ {\bm i} h}$. Its linearization at $h = 0$ is the linear operator
\beqn
\Delta_B:= d_B^* d_B: W_B^{2,p}(\Sigma, {\rm ad} Q) \to L^p(\Sigma, {\rm ad} Q).
\eeqn
This is a Fredholm operator of index zero. If $B$ is flat, then Hypothesis \ref{hyp210} implies that $\Delta_B$ is invertible. Since $\Delta_B$ depends on $B \in {\mc A}^{1,p}(Q)$ smoothly, the norm $W^{2,p}_B$ is gauge-invariant and varies smoothly with $B$, and $R_\Sigma$ is compact, there is a constant $C = C_p>0$ such that 
\beqn
\| \Delta_B h \|_{L^p(\Sigma)} \geq 2C \| h \|_{W_B^{2,p}(\Sigma)},\ \forall B \in {\mc A}_{\rm flat}^{1,p}(Q). 
\eeqn
Then when $\epsilon$ is sufficiently small, by Uhlenbeck's weak compactness, any $B \in {\mc A}_\epsilon^{1,p}(Q)$ is sufficiently close to a flat $W^{1,p}$-connection in the $W^{1,p}$-norm. It follows that
\beqn
\| \Delta_B h \|_{L^p(\Sigma)} \geq C \| h \|_{W_B^{2,p} (\Sigma)},\ \forall B \in {\mc A}_\epsilon^{1,p}(Q).
\eeqn
Then by applying the implicit function theorem (see for example \cite[Proposition A.3.4]{McDuff_Salamon_2004}), one obtains an appropriate $\epsilon_p'>0$ and a unique $h_B$ satisfying \eqref{eqn213}. Moreover, $h_B$ satisfies the estimate \eqref{eqn214}. The $C^1$-dependence of $h_B$ on $B$ is also a consequence of the implicit function theorem.
\end{proof}

\begin{defn} \label{nsmap}
Let $p\geq 2$ and $\epsilon_p$ be the one in Lemma \ref{lemma217}. Define the {\bf Narasimhan--Seshadri map}
\beq\label{eqn215}
{\it NS}_p: {\mc A}_{\epsilon_p}^{1,p}(Q) \to {\mc A}_{\rm flat}^{1,p}(Q)
\eeq
by 
\beqn
{\it NS}_p(B) = (e^{{\bm i} h_B})^* B.
\eeqn
\end{defn}

We know that each flat connection in ${\mc A}_{\rm flat}^{1,p}(Q)$ is gauge equivalent via a gauge transformation of class $W^{2,p}$ to a smooth flat connection. Then there is a homeomorphism
\beqn
{\mc A}_{\rm flat}^{1,p}(Q)/ {\mc G}^{2,p}(Q) \cong R_\Sigma. 
\eeqn
The composition of ${\it NS}_p$ with the projection ${\mc A}_{\rm flat}^{1,p}(Q) \to R_\Sigma$ is denoted by
\beqn
\ov{\it NS}_p: {\mc A}_{\epsilon_p}^{1,p}(Q) \to R_\Sigma.
\eeqn
By abusing names we still call $\ov{\it NS}_p$ the Narasimhan--Seshadri map. An important property of this map is that it maps instantons over $S \times \Sigma$ to holomorphic maps from $S$ to $R_\Sigma$.

\begin{prop}\label{prop219}
Let $S \subset {\bm C}$ be an open subset and $\A = d_S + \phi ds + \psi dt + B$ be a smooth connection on $S \times Q$ satisfying the first equation of \eqref{eqn24}, i.e.
\beqn
\A_s + * \A_t = 0.
\eeqn
Suppose for sufficiently small $\epsilon$ one has $B(z) \in {\mc A}_\epsilon^{1,2}(Q)$ for all $z \in S$, then the map $u: S \to R_\Sigma$ defined by $u(z) = \ov{\it NS}_2 (B(z))$ is holomorphic.
\end{prop}

\begin{proof}
By the definition of ${\it NS}_2$, there exists a complex gauge transformation $g$ on $S \times Q \to S \times \Sigma$ whose restriction to each fibre $\{z \}\times Q$ is the purely imaginary gauge transformation $g_z = e^{{\bm i}h_{B(z)}}$ provided by Lemma \ref{lemma217} such that ${\it NS}_2(B(z)) = g_z^* B(z)$. Suppose $g^* \A = \A' + \phi' ds + \psi' dt + B'$. Since complex gauge transformations preserve the first equation of \eqref{eqn24} (see \cite[Corollary 2.24]{Fukaya_1998}), one still have
\beq\label{eqn216}
0 = \A_s' + * \A_t' =  \partial_s B' - d_{B'} \phi' + * ( \partial_t B' - d_{B'} \psi').
\eeq
For each $z \in S$, $u(z)$ is represented by the flat connection $B'(z)$ and one can identify
\beqn
T_{u(z)} R_\Sigma \cong \{ \beta \in \Omega^1({\rm ad} Q)\ |\ d_{B'(z)} \beta = d_{B'(z)} * \beta = 0 \}
\eeqn
where the complex structure $J$ on the left is identified with the Hodge star on the right. Let $\pi_{B'(z)}$ be temporarily the orthogonal projection onto $d_{B'(z)}$-harmonic 1-forms. Then with respect to the above identification, one has 
\begin{multline*}
\partial_s u(z) + J \partial_t u(z) = \pi_{B'(z)} ( \partial_s B'(z)) + * \pi_{B'(z)} (\partial_t B'(z)) = \pi_{B'(z)} ( \partial_s B'(z) + * \partial_t B'(z)) \\
= \pi_{B'(z)} ( \partial_s B'(z) - d_{B'(z)} \phi' + * ( \partial_t B'(z) - d_{B'(z)} \psi')) = 0.
\end{multline*}
So $u$ is a holomorphic map.
\end{proof}

\subsubsection{Energy identity for pseudoholomorphic curves in the representation variety}

We show that the energy of holomorphic curves in the representation variety can be expressed in terms of the Chern--Simons functional. For simplicity we only discuss a special case but the general formula can be obtained using the same argument. Let $P_{N^\pm} \to N^\pm$ be a pair of $SO(3)$-bundles over three-manifolds with boundary which satisfy Hypothesis \ref{hyp210} and Hypothesis \ref{hyp213}. Let the domain of the holomorphic curve be the strip
\beqn
S_{[a, b]}:= [a,b] \times [-1, 1]
\eeqn
whose coordinates are $(s, t)$. Let $\partial^\pm S_{[a, b]} \subset \partial S_{[a, b]}$ be the $\pm 1$ side of the boundary. Denote 
\beq\label{eqn217}
\N_{[a,b]}:= (\partial^- S_{[a,b]}\times N^-) \cup (S_{[a, b]}\times \Sigma) \cup (\partial^+ S_{[a,b]} \times N^+)
\eeq
which is a four-manifold with boundary $\partial \N_{[a,b]} = \{a, b\}\times N$. Define an $SO(3)$-bundle $\P \to \N_{[a,b]}$ as 
\beqn
\P:= (\partial^- S_{[a,b]}\times P_{N^-} ) \cup (S_{[a, b]}\times Q) \cup (\partial^+ S_{[a,b]} \times P_{N^+}).
\eeqn
A piecewise smooth connection on ${\bm P}$ is a continuous connection whose restriction to each of  the three parts in \eqref{eqn217} is smooth.

Now consider holomorphic maps from $S_{[a, b]}$ to $R_\Sigma$ with the two boundary components $\partial^\pm S_{[a,b]}$ mapped into the two immersed Lagrangians $\iota^\pm(L_{N^\pm}) \subset R_\Sigma$ respectively. Such an object can be described by a triple ${\bm u} = (u, \gamma^-, \gamma^+)$ where $u: S_{[a,b]}\to R_\Sigma$ is a holomorphic map and $\gamma^\pm: \partial^\pm S_{[a,b]} \to L_{N^\pm}$ are continuous maps satisfying
\beqn
u|_{\partial^\pm S_{[a,b]}} = \iota^\pm \circ \gamma^\pm.
\eeqn
By Lemma \ref{lemma29}, $u$ is indeed smooth along the boundary and $\gamma^\pm$ are also smooth.

\begin{defn}\label{lift}
For a holomorphic map ${\bm u} = (u, \gamma^-, \gamma^+)$ as described above, a {\bf lift} of ${\bm u}$ is a piecewise smooth connection ${\bm A}$ on $\P_{[a, b]}\to \N_{[a,b]}$ satisfying the following conditions.

\begin{enumerate}
\item Over $[a, b]\times N^\pm$, if we write ${\bm A} = d_s + \eta^\pm(s) + A^\pm_s$, then for all $s \in [a,b]$
\begin{align*}
&\ A^\pm_s\in {\mc A}_{\rm flat}(P_{N^\pm}),\ &\ [A^\pm_s] = \gamma^\pm(s)\in L_{N^\pm}.
\end{align*}

\item Over $S_{[a,b]}\times \Sigma$, if we write ${\bm A} = d_{\bm C} + \phi ds + \psi dt + B(z)$, then for all $z \in S_{[a, b]}$
\begin{align*}
&\ B(z) \in {\mc A}_{\rm flat}(Q),\ &\ [B(z)] = u(z) \in R_\Sigma
\end{align*}
and for all $z \in S_{[a,b]}$, one has  
\beq\label{eqn218}
\partial_s B(z) - d_{B(z)} \phi,\ \partial_t B(z) -d_{B(z)} \psi \in {\rm ker} d_{B(z)} \cap {\rm ker} d_{B(z)}^*.
\eeq
\end{enumerate}
\end{defn}

\begin{prop}\label{prop221}
For a holomorphic map ${\bm u} = (u, \gamma^-, \gamma^+): S_{[a, b]} \to R_\Sigma$ as described above, there exists a piecewise smooth lift ${\bm A}$. Moreover, the energy of ${\bm u}$ is equal to $CS_{\bm P}(A_a, A_b)$ where $A_a, A_b \in {\mc A}^{\rm p.s.}(N)$ are the boundary restrictions of any lift ${\bm A}$. 
\end{prop}

\begin{proof}
Let ${\bm u} = (u, \gamma^-, \gamma^+)$ be a holomorphic map as described. First, one can find a family of smooth connections $B(z) \in {\mc A}_{\rm flat} (Q)$ parametrized by by $z \in S_{[a,b]}$ such that 
\beqn
[B(z)] = u(z).
\eeqn
We claim that the family $B(z)$ can be chosen such that it depends smoothly on $z$. Indeed, for each $z \in S_{[a, b]}$, one can choose a smooth $B(z)$ representing $u(z)$. Then there exists a small $\delta >0$ such that the local Coulomb slice
\beqn
V_{B(z)}(\delta):= \{ B\in {\mc A}_{\rm flat}(Q)\ |\ d_{B(z)}^* (B - B(z)) = 0,\ \| B - B(z)\|_{L^2(\Sigma)} < \delta \}
\eeqn
is diffeomorphic to a neighborhood of $u(z)$ in $R_\Sigma$ via the map $B \mapsto [B]$. Since $u$ is smooth, there exists $r>0$ such that for all $w \in B_r(z)\cap S_{[a, b]}$ (where $B_r(z) \subset {\bm C}$ is the open disk), one can choose $B(w) \in V_{B(z)}(\delta) \subset {\mc A}_{\rm flat}(Q)$. Hence $B$ depends smoothly on $w$. One can cover $S_{[a, b]}$ by finitely such neighborhoods of the form $B_r(z) \cap S_{[a, b]}$ such that over each such open set one has a lift of $u$ which varies smoothly. Then one can change these local lifts inductively by applying gauge transformations to construct a lift which is smooth globally over $S_{[a, b]}$.

Since $R_\Sigma$ is a locally free symplectic quotient, there are unique $\phi(z), \psi(z) \in \Omega^0( {\rm ad} Q)$ such that \eqref{eqn218} is satisfied. Then define
\beqn
\A_1 = d_S + \phi(z) ds + \psi(z) dt + B(z)
\eeqn
which is a smooth connection on ${\bm P}$ restricted to $S_{[a,b]} \times \Sigma$. %By applying suitable gauge transformations, we can assume that 
%\beqn
%\psi|_{\partial^\pm S_{[a,b]}\times \Sigma} \equiv 0.
%\eeqn

On the other hand, for each boundary point $(s, \pm 1) \in \partial^\pm S_{[a,b]}$, using the map $\gamma^\pm$ one can find a family of flat connections $A^\pm_s \in {\mc A}_{\rm flat}(P_{N^\pm})$ such that for all $s \in [a,b]$,
\begin{align*}
&\ [A^\pm_s] = \gamma^\pm (s),\ &\ A^\pm_s|_{\partial N^\pm} = B(s, \pm 1).
\end{align*}
One can also make $A^\pm_s$ depend smoothly on $s$. Moreover, by applying a further gauge transformation which is the identity over $[a,b]\times \Sigma \subset [a,b]\times N^\pm$, one may assume that along the boundary $A^\pm_s$ has vanishing $dt$ component. Then define 
\beqn
\A_\pm = d_s + \eta^\pm(s) ds + A_{s, \pm} \in {\mc A}( \partial^\pm S_{[a,b]} \times P_{N^\pm} )
\eeqn
where $\eta^\pm(s)\in \Omega^0(N^\pm, {\rm ad} P_{N^\pm})$ is an arbitrary smooth extension of the boundary restriction of $\phi(s)$. This is a smooth connection on ${\bm P}$ restricted to $\partial^\pm S_{[a,b]} \times N^\pm$. Then $\A_1$ and $\A_\pm$ together define a piecewise smooth connection $\A$ on ${\bm P}$. This finishes the proof of the existence of a piecewise smooth lift.  

Now assume that $\A$ is a piecewise smooth lift of ${\bm u}$. We calculate the energy of ${\bm u}$. Indeed, since $A^\pm_s$ is flat, one has 
\beq\label{eqn219}
{\rm tr}( F_\A \wedge F_\A)|_{\partial^\pm S_{[a,b]} \times N^\pm} = 0.
\eeq
On the other hand, over $S_{[a,b]} \times \Sigma$ one has
\beqn
F_\A = F_{B(z)} + ds \wedge ( \partial_s B(z) - d_B \phi) + dt\wedge ( \partial_t B(z) - d_B \psi) + ( \partial_s \psi - \partial_t \phi + [\phi, \psi]) ds\wedge dt.
\eeqn
Since $F_{B(z)} \equiv 0$, one has 
\beq\label{eqn220}
{\rm tr} (F_\A \wedge F_\A) =   {\rm tr} \Big( ( \partial_s B - d_B \phi ) \wedge (\partial_t B(z) - d_B \psi) \Big) ds\wedge dt.
\eeq
Moreover, since the linear map $D\ov{\it NS}_2: {\rm ker} d_{B(z)} \cap {\rm ker} d_{B(z)}^* \to T_{[B(z)]} R_\Sigma$ is an isometry, by \eqref{eqn218} one has 
\beq\label{eqn221}
|\partial_s u |^2  = |\partial_s B - d_{B } \phi(z)|^2  = - \int_\Sigma {\rm tr}  \Big( (\partial_s B - d_{B } \phi ) \wedge (\partial_t B - d_{B} \psi ) \Big).
\eeq
Then by \eqref{eqn219}, \eqref{eqn220}, and \eqref{eqn221}, one has
\beqn
{\it CS}_{\bm P}(A_a, A_b) =  - \frac{1}{2} \int_{\N_{[a,b]}} {\rm tr} (F_\A \wedge F_\A)  = \int_{S_{[a,b]}} |\partial_s u|^2 dsdt = E(u). \qedhere
\eeqn
\end{proof}

\subsubsection{A technical lemma}

The following lemma will be used a few times in the rest of this paper. 

\begin{lemma}\label{lemma222}
Let $p>2$. There exist $\delta>0$ and $C>0$ satisfying the following property. Suppose $B \in {\mc A}^{1,p}(Q)$ and $g \in {\mc G}^{1,p}(Q)$ satisfying 
\begin{align*}
&\ \| F_B \|_{L^p(\Sigma)} \leq \delta,\ &\ \| g^* B - B \|_{L^p(\Sigma)} \leq \delta.
\end{align*}
Then there exists $h \in W^{1,p}({\rm ad} Q)$ with $g = \pm e^h$ and 
\beqn
\| h \|_{W^{1,p}_B} \leq C \| g^* B - B \|_{L^p(\Sigma)}.
\eeqn
\end{lemma}
 
\begin{proof}
Suppose this is not the case. Then there exist a sequence of connections $B_i \in {\mc A}^{1,p}(Q)$ and a sequence of gauge transformations $g_i \in {\mc G}^{1,p}(Q)$ such that
\begin{align*}
&\ \| F_{B_i} \|_{L^p(\Sigma)} \to 0,\ &\ \| g_i^* B_i - B_i \|_{L^p(\Sigma)} \to 0
\end{align*}
and such that none of $g_i$ satisfies the required property. Then by the Uhlenbeck compactness a subsequence of $B_i$ (still indexed by $i$) converges up to gauge transformation to a flat connection $B_\infty$ weakly in $W^{1,p}$. Notice that this lemma is gauge invariant. Then using the local slice theorem and the elliptic estimate, one can assume that $B_i$ converges to $B_\infty$ in $W^{1,p}$. Then it follows that $g_i$ has uniformly bounded $W^{1,p}$-norm. By the Sobolev embedding $W^{1,p}\to C^0$ in dimension two, a subsequence of $g_i$ (still indexed by $i$) converges in $C^0$ to a limit continuous gauge transformation $g_\infty$. Then in the weak sense $g_\infty^* B_\infty = B_\infty$. Since $Q$ has no reducible flat connections, it follows that $g_\infty = \pm 1$. Without loss of generality, we may assume that $g_\infty = 1$. Then the $C^0$ convergence of $g_i \to 1$ implies that for $i$ sufficiently large, one can write $g_i = e^{h_i}$ with $h_i \in W^{1,p}({\rm ad} Q)$. Then by the definition of gauge transformation, pointwise on $\Sigma$ one has the bound
\beqn
| d_{B_i} h_i - ( g_i^* B_i - B_i)| \leq C |h_i| |d_{B_i} h_i|.
\eeqn
Since $h_i$ is small in the $C^0$ norm, it follows that 
\beqn
\| d_{B_i} h_i \|_{L^p(\Sigma)} \leq C \| g_i^* B_i - B_i \|_{L^p(\Sigma)} \to 0.
\eeqn
On the other hand, the operator $d_{B_\infty}: W^{1,p}_{B_\infty}({\rm ad} Q) \to L^p(\Lambda^1\otimes {\rm ad} Q)$ is injective as $B_\infty$ is not reducible. The weak $W^{1,p}$-convergence of $B_i \to B_\infty$ implies that $d_{B_i}$ is also injective when $i$ is sufficiently large. Hence it follows that 
\beqn
\| h_i \|_{W^{1,p}_{B_i}} \leq C \| d_{B_i} h_i \|_{L^p(\Sigma)} \leq  C \| g_i^* B_i - B_i \|_{L^p(\Sigma)}
\eeqn
which contradicts our assumption. Hence the lemma is proved. \end{proof}

\section{The rescaled equation and interior compactness}\label{section3}

In this section we consider the ASD equation over the product of two surfaces. In the adiabatic limit of the rescaled version, we recall an interior estimate by Dostoglou--Salamon \cite{Dostoglou_Salamon, Dostoglou_Salamon_corrigendum} which leads to the compactness modulo bubbling results (Theorem \ref{thm33}). We also give a refined version of the interior estimate near the boundary (Corollary \ref{cor36}) which will be useful in the next section.

\subsection{An interior estimate for the rescaled equation}\label{subsection31}

First we recall the notion of the rescaled ASD equations. Let $\rho$ be a positive number and $S \subset {\bm C}$ be open. Let 
\beqn
\varphi_\rho: {\bm C} \to {\bm C}
\eeqn
be the multiplication $z \mapsto \rho z$. Recall that a connection $\A$ on $\varphi_\rho(S) \times Q$ can be written in components as $\A = d_{\varphi_\rho(S)} + \phi ds + \psi dt + B$. Denote 
\beqn
\A' = \varphi_\rho^* \A = d_S + \phi' ds + \psi' dt + B'.
\eeqn
Recall the notations introduced in \eqref{eqn22} and \eqref{eqn23}. Then the ASD equation on $\A$ can be rewritten as the following equation on $\A'$ over $S \times \Sigma$
\begin{align}\label{eqn31}
&\ \A_s' + * \A_t' = 0,\ &\ * \kappa_{\A'} + \rho^2 \mu_{\A'} = 0.
\end{align}
We call this equation the {\bf $\rho$-ASD equation}. Define the rescaled energy density function
\beq\label{eqn32}
e_{\rho}(z):=  \| \A_s' \|_{L^2(\{z\}\times \Sigma)}^2 + \rho^2 \|\mu_{\A'} \|_{L^2(\{z\}\times \Sigma)}^2
\eeq
and the rescaled energy 
\beqn
E_\rho(\A'; S):= \int_{S} e_\rho(z)ds dt.
\eeqn
The basic relation between the rescaled energy density and the original energy density is
\beqn
e_\rho(z) = \rho^2 \| F_\A \|_{L^2( \{\rho z \}\times \Sigma)}^2,\ \forall z \in S.
\eeqn
Therefore we have
\beqn
E( \A; \varphi_\rho(S)) = \frac{1}{2} \| F_{\A}\|_{L^2( \varphi_{\rho}(S) \times \Sigma)}^2 = E_\rho(\A';S).
\eeqn
When the set $S$ is understood from the context, we abbreviate $E_\rho(\A'; S)$ by $E_\rho(\A')$. We also introduce another density in terms of fibrewise $L^\infty$-norm:
\beqn
e_\rho^\infty(z):= \| \A_s' \|_{L^\infty(\{z\}\times \Sigma)}^2 + \rho^2 \| \mu_{\A'} \|_{L^\infty( \{z\}\times \Sigma)}^2.
\eeqn

In \cite{Dostoglou_Salamon, Dostoglou_Salamon_corrigendum} Dostoglou--Salamon obtain several important estimates about the rescaled equation. These estimates are essential in proving the convergence of rescaled instantons towards holomorphic curves. We recall them as follows.

\begin{thm}\cite[Theorem 7.1]{Dostoglou_Salamon}\cite[Corollary 1.1]{Dostoglou_Salamon_corrigendum}\cite[Lemma 2.1]{Dostoglou_Salamon_corrigendum}\label{thm31} Let $S \subset {\bm C}$ be an open set, $K \subset S$ be a compact subset, and $c_0>0$. Then for $p \in [2, \infty]$, there exist constants $C_p > 0$ and $T_p>0$ such that the following holds. Let $\A' = (B', \phi', \psi')$ be a solution to \eqref{eqn31} over $S \times \Sigma$ with $\rho \geq T_p$ satisfying
\beq\label{eqn33}
\sup_{z \in S} e_\rho(z) = \sup_{z \in S} \Big( \| \A_s' \|_{L^2(\Sigma)}^2 + \rho^2 \| \mu_{\A'} \|_{L^2(\Sigma)}^2 \Big) \leq c_0.
\eeq
Then there holds
\beq\label{eqn34}
\| d_{B'} \A_s' \|_{L^p(K \times \Sigma)} + \| d_{B'} \A_t'\|_{L^p(K \times \Sigma)} \leq C_p \rho^{- \frac{2}{p}} \sqrt{E_\rho(\A'; S)};
\eeq
\beq\label{eqn35}
\| \A_s'\|_{L^\infty(K \times \Sigma)} + \rho \| \mu_{\A'} \|_{L^\infty(K \times \Sigma)} \leq C_\infty \sqrt{ E_\rho (\A'; S)}.
\eeq
\beq\label{eqn36}
\sup_{z \in K} \| \mu_{\A'}(z)\|_{L^2(\Sigma)} \leq C_p \rho^{- 2 + \frac{2}{p}}.
\eeq
\end{thm}

\begin{proof}
By \cite[Theorem 7.1]{Dostoglou_Salamon}, for $c_0'>0$ there exist $C_p>0$ and $T_p>0$ such that for a solution $\A'= (B', \phi', \psi')$ to the $\rho$-ASD equation for $\rho \geq T_p$ and satisfying 
\beq\label{eqn37}
\sup_{z\in S} e_\rho^\infty(z) \leq c_0'
\eeq
there holds
\beqn
\| d_{B'} \A_s' \|_{L^p(K \times \Sigma)} + \| d_{B'} \A_t' \|_{L^p( K \times \Sigma)} \leq C_p \rho^{-\frac{2}{p}} \sqrt{ E_\rho(\A'; S \times \Sigma)}.
\eeqn
By the discussion at the beginning of \cite[Section 1]{Dostoglou_Salamon_corrigendum}, the assumption \eqref{eqn37} can be weakened to \eqref{eqn33}. So \eqref{eqn34} holds. Second, the estimate \eqref{eqn35} follows from \cite[Corollary 1.1]{Dostoglou_Salamon_corrigendum}. Lastly we prove \eqref{eqn36}. Choose a pair of precompact open neighborhoods $S'' \subset S' \subset S$ of $K$ whose areas are finite such that $\ov{S''} \subset S'$ and $\ov{S'} \subset S$. Let $C_\infty'>0$ such that \eqref{eqn35} holds if we replace $K$ by $\ov{S''}$, $S$ by $S'$, and $C_\infty$ by $C_\infty'$. By \cite[Lemma 2.1]{Dostoglou_Salamon_corrigendum}, for any $c_1>0$, there exist $C_p >0$, $T_p >0$ (which also depend on $S''$ and $K$) and $\delta_0>0$ (which is independent of $p$, $c_1$, and regions $S'', K$) with the following property. Let $\A' = (B', \phi', \psi')$ be a solution to the $\rho$-ASD equation over $S'' \times \Sigma$ for $\rho \geq T_p$ satisfying
\begin{align}\label{eqn38}
&\ \| \A_s'\|_{L^\infty(S'' \times \Sigma)} \leq c_1,\ &\ \| \mu_{\A'}\|_{L^\infty(S'' \times \Sigma)} \leq \delta_0.
\end{align}
Then \eqref{eqn36} holds. On the other hand, by \eqref{eqn33} and \eqref{eqn35}, choosing $T_p$ sufficiently large, one has
\beqn
\| \A_s' \|_{L^\infty(S'' \times \Sigma)} \leq C_\infty' \sqrt{E_\rho(\A'; S')} \leq C_\infty' \sqrt{c_0 {\rm Area}(S')}
\eeqn
and
\beqn
\| \mu_{\A'}\|_{L^\infty(S'' \times \Sigma)} \leq \frac{C_\infty' \sqrt{E_\rho(\A'; S')}}{T_p} \leq \frac{ C_\infty' \sqrt{ c_0 {\rm Area} (S')}}{T_p} \leq \delta_0.
\eeqn
Then by applying \cite[Lemma 2.1]{Dostoglou_Salamon_corrigendum} for $c_1 = C_\infty' \sqrt{c_0 {\rm Area} (S')}$, \eqref{eqn36} is proved.
\end{proof}

\begin{rem}
Dostoglou--Salamon's estimates are valid for the case that the $\rho$-ASD equation has a type of holonomic perturbation term. Such a perturbation induces a family of Hamiltonian function on $R_\Sigma$. See \cite[Section 7]{Dostoglou_Salamon} for more details. 
\end{rem}

\subsection{Adiabatic limit}

Now we consider the adiabatic limit of the rescaled ASD equation towards holomorphic maps in the representation variety. The following theorem, which is a reformulation of results used in the proof of \cite[Theorem 9.1]{Dostoglou_Salamon}, is fundamental to the main result of this paper. For completeness we provide its proof in the appendix.

\begin{thm}\label{thm33}(cf. \cite[Theorem 9.1]{Dostoglou_Salamon} \cite[Lemma 3.5]{Duncan_2012} \cite[Theorem 1.2]{Nishinou_2010})
Let $\rho_i \to  + \infty$ be a sequence of positive numbers diverging to infinity. Let $S \subset  \C$ be an open subset and $S_i \subset S$ be an exhausting sequence of open subsets. Let $\A_i' = d_S + \phi_i' ds + \psi_i' dt + B_i'(z)$ be a sequence of solutions to the $\rho_i$-ASD equation over $S_i \times \Sigma$. Let $e_{\rho_i}: S_i \to [0, +\infty)$ the rescaled energy density function of $\A_i'$. Suppose 
\beq\label{energydensity}
\limsup_{i \to \infty} E_{\rho_i}(\A_i') < \infty.
\eeq
Then there exists a subsequence (still indexed by $i$) and a holomorphic map with mass
\beqn
\tilde {\bm u}_\infty = (u_\infty, W_\infty, m_\infty): S \to R_\Sigma
\eeqn
satisfying the following condition. 

\begin{enumerate}

\item For every point $z \in S$, for sufficiently small $r$, the limit 
\beqn
\lim_{i \to \infty} E_{\rho_i}(\A_i'; B_r(z)\times \Sigma)
\eeqn
exists. Moreover, there holds
\beqn
\lim_{r \to 0} \lim_{i\to \infty} E_{\rho_i}(\A_i'; B_r(z) \times \Sigma) = m_\infty(z).
\eeqn

\item For any precompact open subset $K \subset S \setminus W_\infty$, for $i$ sufficiently large, $B_i'(z)$ is contained in the domain of the Narasimhan--Seshadri map $\ov{\it NS}_2$ for all $z \in K$, hence $\A_i'$ induces a holomorphic map 
\beqn
u_i': K \to R_\Sigma,\ u_i' (z):= \ov{\it NS}_2( B_i'(z)).
\eeqn
Moreover, the sequence of maps $u_i'$ converges to $u_\infty$ in $C^\infty ( K)$ and the energy density $e_{\rho_i}$ converges to the energy density of $u_\infty$ in $C^0(K)$. 

\item For any compact subset $K \subset S$ containing $W_\infty$, there holds
\beqn
\lim_{i \to \infty} E_{\rho_i}(\A_i'; K \times \Sigma)  = E(u_\infty; K) + \int_K m_\infty.
\eeqn
\end{enumerate}
Moreover, $W_\infty = \emptyset$ if and only if for all compact $K \subset S$ there holds
\beqn
\limsup_{i \to \infty} \| e_{\rho_i}\|_{L^\infty(K)} < \infty.
\eeqn
\end{thm}

\begin{proof}
See Appendix \ref{appendixa}.
\end{proof}

\begin{defn}\label{defn34}
Let $S \subset {\bm C}$ be an open subset and $S_i \subset S$ be an exhaustive sequence of open subsets. Let $\rho_i \to +\infty$ be a sequence of positive numbers. Let $\A_i$ be a sequence of solutions to the ASD equation over $\varphi_{\rho_i}(S_i) \times \Sigma$. Let $\tilde {\bm u}_\infty = (u_\infty, W_\infty, m_\infty)$ be a holomorphic map from $S$ to $R_\Sigma$ with mass. We say that $\A_i$ converges to $\tilde {\bm u}_\infty$ along with $\{\rho_i\}$ if for the corresponding sequence of solutions $\A_i'$ of the $\rho_i$-ASD equation over $S_i \times \Sigma$, conditions (a), (b), and (c) of Theorem \ref{thm33} are satisfied. 
\end{defn}

\subsection{Estimates near the boundary}

In this section we extend certain estimates of Dostoglou--Salamon near the boundary of the domain. Such an extension will be useful in the next section. Let $B_R= B_R(0) \subset {\bm C}$ be the radius $R$ open disk centered at the origin. 

\begin{lemma}\label{lemma35}
There exist $\epsilon>0$, $T>0$, and $C>0$ satisfying the following conditions. Let $\A = d_{\bm C} + \phi ds + \psi dt + B(z)$ be a solution to the $\rho$-ASD equation over $B_1 \times \Sigma$ satisfying $\rho \geq T$, $E_\rho(A) \leq \epsilon$ and 
\beqn
\sup_{B_1} e_{\rho}^\infty \leq \rho^2.
\eeqn
Then for all $r \in (0,1)$ there holds
\beq\label{eqn311}
\sup_{B_r} e_\rho^\infty \leq \frac{C}{(1-r)^2}.
\eeq
\end{lemma}

\begin{proof}
Suppose this is not the case. Then there exist a sequence $\rho_i \to \infty$, a sequence of solutions $\A_i$ to the $\rho_i$-ASD equation over $B_1 \times \Sigma$, and a sequence of points $z_i \in B_1$ such that $E_{\rho_i} (\A_i) \to 0$, $\sup e_{\rho_i}^\infty \leq \rho_i^2$, and 
\beq\label{eqn312}
\lim_{i \to \infty} ( 1 - r_i)^2 e_{\rho_i}^\infty (z_i)  = \infty,\ {\rm where}\ r_i = |z_i|.
\eeq
Then it follows that 
\beqn
\rho_i^2 \geq e_{\rho_i}^\infty(z_i) \gg \frac{1}{(1-r_i)^2}.
\eeqn
Here $a_i \gg b_i$ (equivalently $b_i \ll a_i$) means $a_i / b_i \to +\infty$. Denote $\tilde \rho_i:= \rho_i (1-r_i)$ which diverges to infinity. Then define the rescaling 
\beqn
\varphi_i: B_1 \to B_{1-r_i}(z_i),\ w \mapsto z_i + (1-r_i) w.
\eeqn
The sequence $\A_i$ is pulled back to a sequence of solutions $\tilde \A_i$ to the $\tilde \rho_i$-ASD equation over $B_1\times \Sigma$ whose energy converges to zero. On the other hand, \eqref{eqn312} implies that 
\beq\label{eqn313}
e_{\tilde \rho_i}^\infty(0) \to +\infty.
\eeq
By Theorem \ref{thm33}, this sequence converges to the constant holomorphic map, implying that $e_{\tilde \rho_i}$ converges to zero uniformly over compact subsets of $B_1$. This contradicts the interior estimate of Theorem \ref{thm31} applied to $\tilde \A_i$ and \eqref{eqn313}. 
\end{proof}

There is also a more refined case of the inequality \eqref{eqn35} of Theorem \ref{thm31} which will be useful in the next section. Recall that for $z \in {\bm H}$,
\beqn
B_r^+(z) = \{ w \in {\bm C}\ |\ {\rm Im} w \geq 0,\ |z-w| < r\}.
\eeqn
Define $B_r^+:= B_r^+(0)$ which contains the interval $(-r, r) \subset \partial {\bm H}$ but not half circle.

\begin{cor}\label{cor36}
There exist $\epsilon>0$, $C >0$, $T>0$ satisfying the following conditions. Suppose $\A = d_{\bm C} + \phi ds + \psi dt + B(z)$ is a solution to the $\rho$-ASD equation over ${\rm Int} B_2^+ \times \Sigma$ for $\rho \geq T$ satisfying 
\beqn
E_\rho(\A) \leq \epsilon,\ \hspace{2cm}\ \sup_{z \in {\rm Int} B_2^+} \| F_{B(z)}\|_{L^3 (\Sigma)} \leq \epsilon,\footnote{In fact what we need is that $F_{B(z)}$ is uniformly small in $L^p$ for some $p>2$.}\ \hspace{2cm}\  \sup_{B_2^+} e_\rho^\infty \leq \rho^2.
\eeqn
Then for all $z = s + {\bm i} t \in {\rm Int} B_1^+$ there holds
\beq\label{eqn314}
\| \A_s \|_{L^2 (\{z\}\times \Sigma)} + \rho \| \mu_\A\|_{L^2 (\{z\}\times \Sigma)} \leq \frac{C}{t} \sqrt{ E_\rho (\A; B_t (z) \times \Sigma)}.
\eeq
\end{cor}

\begin{proof}
The proof is essentially from the proof of \cite[Theorem 7.1]{Dostoglou_Salamon}. Define two functions $h_0, h_1: {\rm Int} B_2^+ \to [0, +\infty)$ by 
\beqn
\begin{split}
h_0 (z) = &\ \frac{1}{2} \Big( \| \A_s \|_{L^2(\{z\}\times \Sigma)}^2 + \rho^2 \| F_{B(z)}\|_{L^2(\Sigma)}^2 \Big) = \frac{1}{2} e_\rho(z), \\
h_1(z) = &\ \frac{1}{2} \Big( \rho^2 \| d_{B(z)} \A_s \|_{L^2(\Sigma)}^2 + \rho^2 \| d_{B(z)} * \A_s \|_{L^2(\Sigma)}^2 + \| \nabla_s \A_s \|_{L^2(\Sigma)}^2 + \| \nabla_t \A_s \|_{L^2(\Sigma)}^2 \\
   &\ + \| d_{B(z)} \kappa_{\A} \|_{L^2(\Sigma)}^2 + \rho^{-2} \| \nabla_s \kappa_{\A} \|_{L^2(\Sigma)}^2 + \rho^{-2} \| \nabla_t \kappa_{\A} \|_{L^2(\Sigma)}^2 \Big).
         \end{split}
\eeqn
Then by the explicit calculation in \cite[p. 616]{Dostoglou_Salamon}, one obtains 
\beqn
\Delta h_0 = 2 h_1 + 5 \langle \A_s, * [\A_s \wedge \kappa_{\A} ] \rangle \geq \| d_{B(z)} \kappa_{\A} \|_{L^2(\Sigma)}^2 + 5 \langle \A_s, * [\A_s \wedge \kappa_{\A} ] \rangle.
\eeqn
When $\epsilon$ is small enough, the condition $\| F_{B(z)}\|_{L^3 (\Sigma)}\leq \epsilon$, the weak Uhlenbeck compactness in dimension two (see \cite[Theorem A]{Wehrheim_Uhlenbeck}), and the Sobolev embedding $W^{1,3} \to C^0$ imply that $B(z)$ is sufficiently close to a flat connection on $Q \to \Sigma$ in the $C^0$-norm. Since there is no reducible flat connections on $Q$ and $R_\Sigma$ is compact, there is a constant $a > 0$ such that
\beqn
\| d_{B(z)} \kappa_{\A} \|_{L^2(\Sigma)} \geq \frac{1}{a }  \| \kappa_{\A} \|_{L^2(\Sigma)}
\eeqn
(cf. \cite[Lemma 7.6]{Dostoglou_Salamon}). On the other hand, by Lemma \ref{lemma35}, when $\rho$ is sufficiently large and the total energy of $\A$ is sufficiently small, for some $C>0$ there holds
\beqn
\| \A_s \|_{L^\infty( \{z\}\times \Sigma)} \leq C t^{-1},\ \forall z = s + {\bm i} t \in {\rm Int} B_1^+.
\eeqn
Hence
\beqn
\big| \langle \A_s, * [\A_s\wedge \kappa_{\A} ]\rangle \big| \leq \frac{C a }{t}  \| d_{B(z)} \kappa_{\A} \|_{L^2(\Sigma)} \| \A_s \|_{L^2(\Sigma)}.
\eeqn
Therefore, for a suitably modified value of $C$ there holds
\beqn
\Delta h_0 \geq  - \frac{ C }{2 t^2} \| \A_s \|_{L^2(\Sigma)}^2 \geq - \frac{C}{t^2} h_0.
\eeqn
Then apply Lemma \ref{meanvalue} below to the disk $B_{\frac{t}{2}}(z)$, one has
\beqn
h_0 (z) \leq \frac{C}{t^2} \int_{B_{\frac{t}{2}}(z)} h_0 \leq \frac{C}{t^2} E_\rho (\A; B_t (z)\times \Sigma).
\eeqn
By the definition of $h_0$, \eqref{eqn314} follows.
\end{proof}

The following mean value estimate was used in the above proof.

\begin{lemma}\cite[Lemma 7.3]{Dostoglou_Salamon}\label{meanvalue}
Let $u: \ov{B_r}  \to [0, +\infty)$ be a $C^2$-function satisfying 
\beqn
\Delta u \geq - a u,\ {\rm where}\  a \geq 0.
\eeqn
Then there holds
\beqn
u(0) \leq \frac{2}{\pi} \Big( a + \frac{4}{r^2} \Big) \int_{B_r} u.
\eeqn
\end{lemma}

%\begin{proof}
%By rescaling one can reduce the problem to the case $r = 1$. Define 
%\beqn
%F(\rho) = ( 1- \rho) \max_{|z|\leq \rho} u(z),\ 0 \leq \rho \leq 1.
%\eeqn
%Then $F(\rho)$ attains its maximum at a certain $\rho^* \in [0, 1)$. Assume and denote
%\beqn
%c:= u(z^*) = \max_{ |z|\leq \rho^*} u,\ {\rm where}\ |z^*|\leq \rho^*.
%\eeqn
%Denote
%\beqn
%\epsilon = \frac{1}{2} ( 1 -\rho^*).
%\eeqn
%Then within the closed disk $\ov{B_\epsilon(z^*)}$ one has
%\beqn
%u(z) \leq 4 c.
%\eeqn
%Then one has
%\beqn
%u''(t) \geq - 4ac^2,\ |z-z^*| \leq \epsilon
%\eeqn
%which means the function $u(z) + ac |z-z^*|^2$ is a subharmonic function. Then by the mean value inequality for subharmonic %functions one obtains
%\beqn
%c^* = u(z^*) \leq \frac{1}{\pi \rho^2 } \int_{B_\rho(z^*)} ( u + ac|z-z^*|^2) \leq \frac{1}{\pi \rho^2} \int_{-1}^1 u(t) dt %+ \frac{ac^*\rho^2}{2},\ \forall 0 < \rho \leq \epsilon.
%\eeqn
%Now suppose first $a \epsilon^2 \leq 1$. Then take $\rho = \epsilon$ one obtains 
%\beqn
%c^* \leq \frac{1}{\pi \epsilon^2 } \int_{B_1} u + \frac{a\epsilon^2 c^*}{2} \Longrightarrow c^* \leq \frac{2}{\pi %\epsilon^2}\int_{B_1} u.%
%\eeqn
%Hence
%\beqn
%u(0) = F(0) \leq F(\rho^*) = 4 \epsilon^2 c^* \leq \frac{8}{\pi} \int_{B_1} u.
%\eeqn
%On the contrary, if $a \epsilon^2 \geq 1$, then choose $\rho$ such that $a\rho^2 = 1$, one obtains
%\beqn
%u(0) \leq c^* \leq \frac{2 a}{\pi} \int_{B_1} u.
%\eeqn
%The case for general $r>0$ follows from the $r=1$ case by rescaling. 
%\end{proof}

\section{The isoperimetric inequality}\label{section4}

It is not too far away from results of the previous section to the compactification of the moduli space of instantons over the product of the complex plane and the compact surface, namely Theorem \ref{thm15}. For the other type of noncompact four-manifold, namely the product ${\bm R} \times M$ where $M$ is the three-manifold with cylindrical end, there is another difficulty. As approaching to the infinity of the ${\bm R}$-direction, the ASD equation over ${\bm R}\times M_0$ is almost like a Lagrangian boundary condition imposed for the ASD equation over ${\bm H}\times \Sigma$. In the theory of holomorphic curves, the Lagrangian boundary condition allows one to extend interior elliptic estimates to the boundary which leads to compactness near the boundary. However, in our situation, the failure of being an actual Lagrangian boundary condition prevents one to have elliptic estimates near the boundary, at least not in a straightforward way. Indeed, one should view the ``seam'' ${\bm R} \times \Sigma$ between ${\bm R} \times M_0$ and ${\bm H} \times \Sigma$ as giving a Lagrangian seam condition given by an infinite dimensional Lagrangian correspondence (see \cite{Wehrheim_Woodward_2015}). One possible approach of our compactness problem would be based on certain hard estimates as did in \cite{Bottman_Wehrheim_2018} for finite dimensional holomorphic quilts (see another approach in \cite[Section 4]{Duncan_2012}). 

In this paper, instead, we take a less analytic approach. The main idea is to view the ASD equation over ${\bm R}\times M$ as a gradient line of the Chern--Simons functional on the closed three-manifold $M^{\rm double}$ with respect to a time-dependent metric. The asymptotic behavior as well as the compactness problem over the part ${\bm R} \times M_0$ can both be treated via an isoperimetric inequality. Roughly speaking, for a closed Riemannian three-manifold $N$, if the Chern--Simons functional for connections on an $SO(3)$-bundle $P \to N$ is Morse or Morse--Bott, then for an almost flat connection $A \in {\mc A}(P)$, one can define a ``local Chern--Simons action,'' which is the Chern--Simons action of $A$ relative to a certain nearby flat connection. The isoperimetric inequality says that the local action can be controlled by $\| F_A \|_{L^2(N)}^2$. This is analogous to the isoperimetric inequality in symplectic geometry (see \cite[Section 4.4]{McDuff_Salamon_2004} \cite[Chapter 3]{Pozniak}). 

In this section we prove the isoperimetric inequality and certain monotonicity properties of solutions to the ASD equation (which we call the annulus lemma). We will also derive a diameter estimate which will be needed for the compactness theorem as well as the asymptotic behavior and energy quantization property of instantons over ${\bm R}\times M$.

\subsection{The isoperimetric inequality}

In this subsection we derive the isoperimetric inequality. Let $N^-, N^+$ be a pair of three-manifolds with diffeomorphic boundary $\Sigma$ and let $P_{N^\pm} \to N^\pm$ be $SO(3)$-bundles with isomorphic boundary restrictions. Assume these object satisfy Hypothesis \ref{hyp210} and Hypothesis \ref{hyp213}. Let $\check N$ be the closed three-manifold obtained by gluing $N^-$ with $(N^+)^{\rm op}$ and let $P_{\check N} \to \check N$ be the glued $SO(3)$-bundle (see the detailed description in Subsection \ref{subsection24}).

\begin{lemma}\label{lemma41}
There exist $\epsilon>0$ and $C>0$ satisfying the following condition. Let $A$ be a piecewise smooth connection on $P_{\check N}$. % whose restriction to $N^{\rm neck} = [-1, 1]\times \Sigma$ is 
%\beqn
%A|_{[-1, 1] \times \Sigma} = d_t + \eta(t ) d t  + B(t ).
%\eeqn
 Suppose 
\beq\label{eqn41}
\| F_A \|_{L^2 (\check N)} \leq \epsilon.
\eeq
Then there exists a piecewise smooth flat connection $A_0 \in {\mc A}_{\rm flat}^{\rm p.s.}(P_{\check N})$ satisfying  
\beq\label{eqn42}
\| A - A_0 \|_{L^4 (\check N)} \leq C \| F_A \|_{L^2(\check N)}.
\eeq
%\end{enumerate}
\end{lemma}

\begin{proof}
%First, notice the statement of this lemma is gauge invariant. Hence we may assume that $\eta(t) \equiv 0$. For any pair $a, b\in [-1, 1]$ there holds
%\beqn
% \| B(b) - B(a)\|_{L^2 (\Sigma)} \leq \int_{a}^b \| B'(t) \|_{L^2 (\Sigma)} dt \leq l_2 (A).
%\eeqn
Let $C>0$ be a constant which is independent of $A$ but may vary in the context. Let $\epsilon$ be smaller than the constant $\epsilon_p$ of Lemma \ref{lemma215} for $p = 2$ and $P_0 \to M_0$ replaced by $P_{N^\pm} \to N^\pm$. %Denote by 
%\begin{align*}
%&\ \check{N}^- := N^- \cup ( [-1, 0]\times \Sigma),\ &\ \check{N}^+: = ( [0, 1]\times \Sigma )\cup N^+
%\end{align*}
%where we glue $\partial N^-$ with $\{-1\}\times \Sigma$ and glue $\partial N^+$ with $\{1\}\times \Sigma$. Then $N = \check{N}^- \cup \check{N}^+$ where the two parts share the common boundary $\partial \check{N}^- \cong \partial \check{N}^+ \cong \Sigma$. 
Then by Lemma \ref{lemma215} there exist flat $W^{1,2}$-connections $A^\pm$ on ${N}^\pm$ such that
\beq\label{eqn43}
\| A -  A^\pm \|_{W_A^{1,2}({N}^\pm)} \leq C \| F_A \|_{L^2({N}^\pm)}.
\eeq
We can replace ${A}^\pm$ by nearby smooth flat connections such that this bound is still valid for a slightly larger $C$. Let $B^\pm$ be the boundary restrictions of $A^\pm$ over $\partial {N}^\pm \cong \Sigma$. By the inequality 
\beqn
\big| \nabla | A - {A}^\pm| \big| \leq |\nabla^A (A - {A}^\pm)|
\eeqn
we see that $|A - {A}^\pm|\in W^{1,2}({N}^\pm)$ where $W^{1,2}( N^\pm)$ is the Sobolev space of functions defined using the Riemannian metric on ${N}^\pm$. Recall that the boundary restriction defines a trace operator
\beqn
W^{1,2}( N^\pm ) \cong H^1( N^\pm ) \to H^{\frac{1}{2}}( \partial N^\pm ) \cong   W^{\frac{1}{2}, 2}(\Sigma).
\eeqn
The fractional Sobolev embedding in dimension two (see \cite[Theorem 6.5]{Fractional_Sobolev}) says that there is a bounded embedding $W^{\frac{1}{2}, 2}(\Sigma) \hookrightarrow L^4(\Sigma)$. Therefore by \eqref{eqn43} one has
\beq\label{eqn44}
\| B^+ - B^-\|_{L^4(\Sigma)} \leq \| B^+ - A|_{\partial N^-} \|_{L^4(\Sigma)} + \| B^- - A|_{\partial N^+} \|_{L^4(\Sigma)} \leq C \| F_A \|_{L^2(  \check N )}.
\eeq

%Then
%\beqn
%\| B_\pm - B(\pm 1)\|_{L^2 (\Sigma)} \leq C \| A - A_\pm \|_{W_A^{1,p} (N^\pm)} \leq C \| F_A \|_{L^p (N^\pm)}.
%\eeqn
%On the other hand, one has
%\begin{multline}\label{eqn44}
%\| B_+ - B_-\|_{L^2 (\Sigma)} \leq \| B_+ - B(1) \|_{L^2 (\Sigma)} + \| B( 1) - B( -1)\|_{L^2 %(\Sigma)} + \| B(-1) - B_- \|_{L^2 (\Sigma)} \\
% \leq C \left( \| F_A \|_{L^p ( N^- \cup N^+ )} + l_2 (A) \right).
% \end{multline}

Denote $b^\pm = [B^\pm]\in R_{\Sigma}$ and $a^\pm = [A^\pm] \in L_{N^\pm}$. Let $d_{R_{\Sigma}}$ be the distance function on $R_{\Sigma}$ induced from the $L^2$-metric and let $d_{L_{N^\pm}}$ be the metric on $L_{N^\pm}$ pulled back from the immersion $\iota^\pm: L_{N^\pm} \to R_{\Sigma}$. 

\vspace{0.2cm}

\noindent {\it Claim. When $\epsilon>0$ is sufficiently small, for some $C>0$ independent of $A$, there exists $a_*^\pm \in L_{N^\pm}$ with $\iota^-(a_*^-) = \iota^+(a_*^+)$ such that 
\beqn
d_{L_{N^-}}(a^-, a_*^-) + d_{L_{N^+}}( a^+, a_*^+) \leq C d_{R_{\Sigma}}( b^-, b^+).
\eeqn
}

\vspace{0.1cm}

\noindent {\it Proof of the claim.} When $\epsilon>0$ is very small, $(a^-, a^+)$ is close to a point $(p^-, p^+)$ with $\iota^-(p^-) = \iota^+(p^+) = b \in R_{\Sigma}$. Because $\iota^-$ and $\iota^+$ intersect cleanly, there exists a local coordinate system of $R_{\Sigma}$ around $b$, denoted by $\varphi: U \to {\mb R}^{2n}$ such that $\varphi(b) = 0$ and the piece of $(\iota^\pm)^{-1}(U)$ containing $p^\pm$ is the subspace ${\mb R}^{I_\pm}$, where $I_\pm \subset \{1, \ldots, 2n\}$ is a subset of $n$ elements. Then if we denote 
\beqn
\varphi(b^\pm) = x^\pm = (x_1^\pm, \ldots, x_{2n}^\pm)
\eeqn
then $x_i^\pm = 0$ when $i \notin I_\pm$. Define $y_*$ by 
\begin{align*}
&\ y_* = (y_{*, 1}, \ldots, y_{*, 2n} ),\ &\ y_{*, i} = \left\{ \begin{array}{cc} \frac{1}{2} ( x_i^- + x_i^+),\ &\ i \in I_- \cap I_+,\\
                                                0,\ &\ i \notin I_- \cap I_+. \end{array}\right.
                                                \end{align*}
Then $\varphi^{-1}(y_*) \in \iota^-(L_{N^-}) \cap \iota^+(L_{N^+}) $ and corresponds to $a_*^\pm \in L_{N^\pm}$ which is close to $a^\pm$. Because locally the $L^2$-metric on $R_{\Sigma}$ is comparable to the Euclidean metric $d_{{\mb R}^{2n}}$, one has
\begin{multline*}
d_{L_{N^-}} (a^-, a_*^-) + d_{L_{N^+}} (a^+, a_*^+) \leq C \big( d_{{\mb R}^{2n}} (x^-, y_+) + d_{{\mb R}^{2n}} (x^+, y_*)\big) \\
\leq C  \left( \sum_{i\in I_- \cap I_+} |x_i^- - x_i^+| + \sum_{i \in I_- \setminus I_+} |x_i^-| + \sum_{i \in I_+ \setminus I_-} |x_i^+| \right) \leq C |x^- - x^+|\leq C d_{R_{\Sigma}}( b^-, b^+).
\end{multline*}

\hfill {\it End of the proof of the claim.}

\vspace{0.2cm}

By the above claim and \eqref{eqn44}, there are representatives $A_*^\pm\in {\mc A}_{\rm flat}^{1,2} (P_{ {N}^\pm })$ of $a_*^\pm$ such that 
\beqn
\| A^\pm - A_*^\pm \|_{W^{1,2}_{A^\pm}( {N}^\pm )} \leq C d_{L_{ N^\pm}} (a^\pm, a_*^\pm) \leq C d_{R_{ \Sigma}}(b^-,b^+) \leq C \| F_A \|_{L^2( \check N )}.
\eeqn
Then by \eqref{eqn43}, one also has
\beq\label{eqn45}
\| A - A_*^\pm \|_{W^{1,2}_A( {N}^\pm )} \leq C \| F_A \|_{L^2( \check N )}.
\eeq
One can replace $A_*^\pm$ by nearby gauge equivalent smooth connections such that the above estimate is still true for slightly large $C$. 

Now by the above claim, $A_*^+$ and $A_*^-$ agree on the boundary up to gauge equivalence. We would like to gauge transform $A_*^+$ such that they glue together to a piecewise smooth flat connection on $P_{\check N}$ satisfying \eqref{eqn42}. Indeed, let us denote $B_*^\pm:= A_*^\pm |_{\partial N^\pm}$ which are a pair of gauge equivalent flat connections on $\Sigma$. Then via the bounded restriction map $W^{1,2}({N}^\pm ) \to L^4(\partial  N^\pm )$, from \eqref{eqn45} one has
\beqn
\| B_*^- - B_*^+\|_{L^4(\Sigma )} \leq C \| F_A \|_{L^2( \check N)}.
\eeqn
Then by $p=4$ case of Lemma \ref{lemma222}, when $\epsilon$ is sufficiently small, there exists a gauge transformation $g = e^h$ such that $h \in W^{1,4}$, $B_*^- = g^* B_*^+$, and 
\beqn
\| h \|_{W^{1,4}_{B_*^+}} \leq  C \| B_*^+ - B_*^- \|_{L^4( \Sigma )} \leq C \| F_A \|_{L^2( \check N)}.
\eeqn
Then using a smooth cut-off function on $N^+$ which is supported in a small neighborhood of $\partial N^+$ and equal to $1$ near $\partial N^+$, one can extend $h$ to a section $h_+\in W^{1,4}_{A_*^+}({N}^+; {\rm ad} P_{ N^+})$ such that 
\beqn
\| h_+ \|_{W_{A_*^+}^{1,4}( N^+ )} \leq C \| h \|_{W^{1,4}_{B_*^+ }(\Sigma) } \leq C \| F_A \|_{L^2( \check N )}.
\eeqn
In particular, $h_+$ is uniformly small as one has the Sobolev embedding $W^{1,4} \to C^0$ in dimension three. Then we can define a piecewise smooth flat connection 
\beqn
A_0:= \left\{ \begin{array}{cc} A_*^-,\ &   {N}^- ,\\
                                (e^{h_+} )^* {A}_*^+,\ &  {N}^+  \end{array} \right.
\eeqn
Then one has
\beqn
\| A - A_0 \|_{L^4 ({N}^+ )} \leq  C \left(  \| A - {A}_*^+ \|_{L^4 (  {N}^+ )} + \| h_+ \|_{W^{1,4}_{A_*^+ }( {N}^+ )} \right) \leq C \| F_A \|_{L^2( \check N  )}.
\eeqn
Then together with \eqref{eqn45} and the Sobolev embedding $W^{1,2}\to L^4$ in dimension 3, one obtains that 
\beqn
\| A - A_0 \|_{L^4 ( \check N )} \leq C \| F_A \|_{L^2(  \check N  )}. \qedhere
\eeqn
\end{proof}

The above lemma allows one to define a local Chern--Simons action. 

\begin{defn}\label{defn42}
Let $A$ be a piecewise smooth connection on $P_{\check N}$ satisfying \eqref{eqn41}. Then there exists a nearby flat connection $A_0$ on $P_{\check N}$ satisfying properties of Lemma \ref{lemma41}. Then we define 
\beqn
F_{\rm loc}(A) = {\it CS}^{\rm rel}(A_0, A)
\eeqn
where ${\it CS}^{\rm rel}(A_0, A)$ is the Chern--Simons action of $A$ relative to $A_0$.
\end{defn}

From now on we assume that the closed three-manifold $\check N$ is the manifold $N$ constructed by gluing together $N^-$, $N^+$ and a neck region $N^{\rm neck} = [0, \pi]\times \Sigma$ (see Subsection \ref{subsection24}). The value of the local Chern--Simons functional is independent of the choice of the nearby reference flat connection as the relative Chern--Simons functional between two nearby flat connections is zero. Indeed, denote 
\beqn
\tilde L_N:= {\mc A}_{\rm flat}(P_{N})/ {\mc G}_0(P_{N})
\eeqn
where ${\mc G}_0(P_N )$ is the identity component of ${\mc G}(P_{ N})$, then the relative action ${\it CS}_{A_0}^{\rm rel}(A)$ only depends on the component of $\tilde L_{N}$ containing the ${\mc G}_0(P_{N})$-orbit of $A_0$. For any connected component $Z \in \pi_0(\tilde L_{N})$ and $A \in {\mc A}^{\rm p.s.}(P_{N})$, define
\beq\label{eqn46}
F_Z(A):= {\it CS}^{\rm rel}(A_0, A)
\eeq
where $A_0$ is any flat connection whose ${\mc G}_0(P_N)$-orbit is contained in $Z$.

Lemma \ref{lemma41} allows us to derive the following isoperimetric inequality. 

\begin{thm}[Isoperimetric Inequality]\label{isoperimetric}
There exist $\epsilon>0$ and $c_P >0$ such for all piecewise smooth connection $A$ on $P_N$ satisfying $\| F_A \|_{L^2(N)} \leq \epsilon$, there holds 
\beq\label{eqn47}
|F_{\rm loc}(A)| \leq c_P \| F_A \|_{L^2(N)}^2.
\eeq
\end{thm}

\begin{proof}
By Lemma \ref{lemma41}, there is a nearby flat connection $A_0$ and the local action is defined by Definition \ref{defn42}. Then by Lemma \ref{lemma41}, \eqref{eqn42}, and formula \eqref{eqn21}, one has 
\begin{multline*}
\int_N {\rm tr}\left[ \frac{1}{2} d_{A_0} (A-A_0) \wedge (A-A_0) + \frac{1}{3} (A-A_0) \wedge (A-A_0) \wedge (A-A_0) \right]\\
\leq C \| F_A\|_{L^2(N)} \| A-A_0 \|_{L^2(N)} + C \| A - A_0\|_{L^3(N)}^3 \leq C \| F_A \|_{L^2(N)}^2.
\end{multline*}  
\end{proof}

\begin{rem}
In (gauged) Lagrangian Floer theory with clean intersections one needs a similar isoperimetric inequality (see \cite[Lemma 3.4.5]{Pozniak}\cite[Proposition 4.3.1]{Schmaschke_thesis}\cite[Lemma 3.17]{Frauenfelder_thesis}) which can be viewed as certain finite-dimensional versions of Theorem \ref{isoperimetric}. In those cases, the proofs essentially depend on the normal form of the clean intersection or a Darboux chart. The current situation is infinite-dimensional and the estimates are sensitive to the choices of norms. Moreover, our proof is more involved as one does not have a true Lagrangian boundary condition. %The methodology of the proof via Lemma \ref{lemma41} also uses the idea of Gaio--Salamon's proof of the isoperimetric inequality in vortex theory for loops (\cite[Lemma 11.3]{Gaio_Salamon_2005}).
\end{rem}

%In practice it is more convenient to use the following version of the isoperimetric inequality. Notice that the $L^2$-length $l_2(A)$ can be bounded by $\| F_A \|_{L^2(N^{\rm neck})}$. So the $p=2$ case of Lemma \ref{lemma41} and Theorem \ref{thm43} imply the following more standard form of the isoperimetric inequality. %Indeed, by using the fact that the moduli space $L_N$ is a clean intersection (which is equivalent to the Chern--Simons functional being Morse--Bott), the following version follows from an estimate of the Hessian of the Chern--Simons functional (see\cite{Fukaya_1996}).

%\begin{thm}[Isoperimetric Inequality]
%There exist $\epsilon>0$ and $c_P>0$ satisfying the following condition. Let $A$ be a piecewise smooth connection on $P_N \to N$ satisfying 
%\beqn
%\| F_A \|_{L^2(N)}^2 \leq \epsilon.
%\eeqn
%Then the local Chern--Simons action $F_{\rm loc}(A)$ is well-defined and there holds
%\beqn
%|F_{\rm loc}(A)| \leq c_P \| F_A \|_{L^2(N)}^2.
%\eeqn
%\end{thm}

\subsection{The annulus lemma}

Now we turn to the decay property of the energy for ASD instantons. For $0 < r < R <\infty$ define the open annulus and the half annulus by 
\begin{align*}
&\ {\it Ann} (r, R):= \{ z \in {\bm C}\ |\ r < |z| < R \},\ &\ {\it Ann}^+(r, R) = {\it Ann}(r, R) \cap {\bm H}.
\end{align*}
The half annulus has two boundary components
\beqn
\partial^\pm {\it Ann}^+(r, R):= \{ s \in \partial {\bm H} \cong {\bm R}\ |\ r < \pm s < R \}.
\eeqn
Define a noncompact four-manifold $\N_{r, R}$ as 
\beqn
\N_{r, R} = ({\it Ann}^+(r, R) \times \Sigma) \cup (\partial^+ {\it Ann}^+(r, R) \times N^+) \cup (\partial^- {\it Ann}^+(r, R) \times N^-).
\eeqn
Equip $\N_{r, R}$ with the product metric. The bundle $P_N \to N$ also extends to a bundle ${\bm P}\to \N_{r, R}$. Moreover, for each $\rho \in (r, R)$, the slice $N_\rho \subset {\bm N}_{r, R}$ is the three-manifold
\beqn
N_\rho:= N^- \cup ( \{ z \in {\bm H}\ |\ |z| = \rho \} \times \Sigma ) \cup N^+.
\eeqn
The ``neck'' of $N_\rho$ is isometric to the product $N_\rho^{\rm neck}\cong [0, \rho \pi] \times \Sigma$. Then each smooth connection $\A \in {\mc A}({\bm P})$ restricts to a piecewise smooth connection 
\beqn
A_\rho \in {\mc A}^{\rm p.s.}( P_\rho),\ {\rm where}\ P_\rho:= {\bm P}|_{N_\rho}.
\eeqn
Moreover, using the polar coordinates, we can identify $P_\rho\to N_\rho$ with $P_N \to N$. By abuse of notation, the corresponding family of connections on $P_N$ parametrized by $\rho \in (r, R)$ is still denoted by $A_\rho\in {\mc A}^{\rm p.s.}(P_N)$. 

\begin{lemma}\label{lemma45}
There exists $\epsilon>0$ satisfying the following condition. Let $\A$ be a solution to the ASD equation over $\N_{r, R}$ with $r \geq 1$ and $\log R - \log r \geq 4$. Assume that 
\beqn
E(\A) < \epsilon.
\eeqn
Then for all $a\in {\mb R}$ such that $(a, a+1) \subset (\log r, \log R)$ there exists at least one $ \tau \in (a, a+1)$ for which $A_{e^\tau} \in {\mc A}^{\rm p.s.}(P)$ satisfies \eqref{eqn41}, hence the local action $F_{\rm loc}(A_{e^\tau})$ is defined.  
\end{lemma}

\begin{proof}
For all $a$ such that $(a,a+1) \subset (\log r +1, \log R-1)$, by the ASD equation, one has
\begin{multline}\label{eqn48}
E(\A; \N_{e^a, e^{a+1}}) = \int_{e^a}^{e^{a+1}} \| F_{A_\rho} \|_{L^2(N_\rho)}^2 d\rho\\
= \int_a^{a+1} \left( \| F_{A_{e^\tau}}\|_{L^2(N^- \cup N^+)}^2  + \| d \theta \wedge ( \partial_\theta B_{e^\tau} - d_{B_{e^\tau}} \eta_{e^{\tau}}) \|_{L^2(N_{e^\tau}^{\rm neck})}^2 + \| F_{B_{e^\tau}} \|_{L^2(N_{e^\tau}^{\rm neck})}^2 \right) e^\tau d\tau \\
= \int_a^{a+1} \left( e^\tau \| F_{A_{e^\tau}} \|_{L^2(N^- \cup N^+)}^2 + \| \partial_\theta B_{e^\tau} - d_{B_{e^\tau}} \eta_{e^\tau} \|_{L^2(N^{\rm neck})}^2 + e^{2\tau} \| F_{B_{e^\tau}} \|_{L^2(N^{\rm neck})}^2 \right) d\tau\\
\geq \int_a^{a+1} \left( \| F_{A_{e^\tau}}\|_{L^2(N^- \cup N^+)}^2 +  \| \partial_\theta B_{e^\tau} - d_{B_{e^\tau}} \eta_{e^{\tau}} \|_{L^2(N^{\rm neck})}^2 + \| F_{B_{e^\tau}} \|_{L^2(N^{\rm neck})}^2 \right) d\tau \\
= \int_a^{a+1} \| F_{A_{e^\tau}}\|_{L^2(N)}^2 d\tau.
\end{multline}
Then there exists some $\tau \in (a, a+1)$ such that 
\beqn
\| F_{A_{e^\tau}} \|_{L^2(N)}^2\leq \epsilon.
\eeqn
Therefore, when $\epsilon$ is sufficiently small, $A_{e^\tau}$ satisfies the hypothesis of Theorem \ref{isoperimetric} and hence its local action can be defined. 
\end{proof}

Now we state and prove the annulus lemma. Define the exponential factor 
\beqn
\delta_P:= c_P^{-1}
\eeqn
where $c_P$ is the isoperimetric constant of Theorem \ref{isoperimetric}. Without loss of generality, assume that $\delta_P \leq 1$. The constant depends on the bundle $P_N \to N$ and the Riemannian metric on $N$. But by abuse of notation we do not distinguish them since we only consider finitely many examples of such a bundle $P_N \to N$. 

\begin{prop}[Annulus Lemma] \label{annulus}
There exist $\epsilon>0$, $C>0$ satisfying the following conditions. Let $\A$ be a solution to the ASD equation over $\N_{r, R}$ with $r \geq 1$ and $\log R - \log r \geq 4$. Assume 
\beqn
E(\A) < \epsilon.
\eeqn
Then for $1 \leq s \leq \frac{1}{2} (\log R - \log r)$ there holds
\beq\label{eqn49}
E(\A; \N_{e^s r, e^{-s} R} ) \leq C e^{-\delta_P s} E(\A; \N_{r, R}).
\eeq
\end{prop}

\begin{proof}
One can identify $\A$ with a piecewise smooth connection over $( r, R ) \times N$ as explained above. In particular, one obtains a family of connections $A_\rho$ on $P_N \to N$ parametrized by $\rho \in (r, R)$. Define 
\beqn
I_\epsilon(r, R):= \{ \rho \in (r, R)\ |\ \| F_{A_\rho}\|_{L^2(N)} \leq \epsilon \}.
\eeqn
Then by Lemma \ref{lemma45}, $(e^a, e^{a+1}) \cap I_\epsilon(r, R) \neq \emptyset$ for all $(a, a+1) \subset (\log r, \log R)$ so that the local action of $A_{e^\tau}$, denoted temporarily by $F(A_{e^\tau})$, can be defined for $e^\tau \in I_\epsilon (r, R)$. 

We would like to show that these local actions, which {\it a priori} are not defined for all $\tau$, extend to a smooth function in $\tau$. In fact there is a map 
\beqn
Z: I_\epsilon( r,  R) \to \pi_0( \tilde L_N)
\eeqn
such that 
\beqn
F(A_\rho) = F_{Z(\rho)} (A_\rho),\ \forall \rho \in I_\epsilon(r, R)
\eeqn
where $F_{Z(\rho)} (A_\rho)$ is the Chern--Simons action of $A_\rho$ relative to the connected component $Z(\rho)$ (see \eqref{eqn46}). We claim that 
\beq\label{eqn410}
F_{Z(\rho') } (A_\rho) = F_{Z(\rho)} (A_\rho),\ \forall \rho, \rho' \in I_\epsilon(r, R).
\eeq
Indeed, by Lemma \ref{lemma21} we know that (suppose $\rho < \rho'$) 
\beqn
|F_{Z(\rho')}( A_\rho) - F_{Z(\rho')} ( A_{\rho'})| = E(\A; \N_{\rho, \rho'}) < \epsilon.
\eeqn
Moreover, by the isoperimetric inequality (Theorem \ref{isoperimetric}) we know that 
\beqn
| F_{Z(\rho') }( A_{\rho'}) - F_{Z(\rho)} ( A_\rho) | \leq c_P \left( \| F_{A_\rho}\|_{L^2(N)}^2 + \| F_{A_{\rho'}}\|_{L^2(N)}^2 \right) \leq C \epsilon^2.
\eeqn
Therefore, when $\epsilon$ is small enough, the difference between $F_{Z(\rho)} (A_\rho)$ and $F_{Z(\rho')} (A_\rho)$ is smaller than the minimal difference between critical values of the Chern--Simons functional on $N$. Hence \eqref{eqn410} is true. Then we can fix a connected component $Z_0 = Z(\rho_0)\in \pi_0(\tilde L_N)$ for some $\rho_0 \in I_\epsilon(r, R)$ and define 
\beqn
F(\rho):= F_{Z_0} ( A_\rho ),\ \forall \rho \in (r, R).
\eeqn
This is a smooth non-increasing function and $F(\rho)$ agrees with the local action of $A_\rho$ when $\rho \in I_\epsilon( r, R)$. Then for $s \in [1, \frac{1}{2}(\log R - \log r) ]$, define
\beqn
J(s):= F(r e^s) - F( R e^{-s}) = \int_{r e^s}^{Re^{-s}} \| F_\A \|_{L^2(N_\rho)}^2 d\rho.
\eeqn

We would like to derive a certain differential inequality of $J(s)$. First, by comparing the stretched metric on $N_\rho$ and the fixed metric on $N$, one has
\beq\label{eqn411}
\| F_{A_\rho}\|_{L^2(N)}^2 \leq \rho \| F_{A_\rho}\|_{L^2(N_\rho)}^2,\ \ \forall \rho \geq 1.
\eeq
By \eqref{eqn411} and the fact that $Re^{-s} \geq re^s \geq 1$ one obtains
\beqn
J'(s) = - r e^s  \| F_\A\|_{L^2(N_{re^s})}^2  -  R e^{-s} \| F_{\A}\|_{L^2(N_{R e^{-s}})}^2 \leq -  \| F_{A_{r e^s}}\|_{L^2(N)}^2 - \| F_{A_{R e^{-s}}}\|_{L^2(N)}^2.
\eeqn
If both $r e^s$ and $R e^{-s}$ are in $I_\epsilon(r, R)$, then by the isoperimetric inequality, one has
\beqn
J'(s) \leq -  \| F_{A_{r e^s}}\|_{L^2(N)}^2 - \| F_{A_{R e^{-s}}}\|_{L^2(N)}^2 \leq - \delta_P \left( F(r e^s) - F( R e^{-s}) \right) = - \delta_P J(s).
\eeqn
If $r e^s$ and/or $R e^{-s}$ are not in $I_\epsilon(r, R)$, i.e., when 
\beqn
\| F_{A_{re^s} }\|_{L^2(N)}^2 > \epsilon\ \ \ {\rm and/or}\ \ \ \| F_{A_{Re^{-s}}} \|_{L^2(N)}^2 > \epsilon,
\eeqn
there exists $s'$ and/or $s'' \in [s-1, s] \subset [0, \frac{1}{2} (\log R - \log r)]$ such that 
\beqn
\| F_{A_{re^{s'}}} \|_{L^2(N)}^2 \leq \epsilon\ \ \ {\rm and/or}\ \ \ \| F_{A_{R e^{-s''}}} \|_{L^2(N)}^2 \leq \epsilon.
\eeqn
Then by the isoperimetric inequality and the monotonicity of $F$, one obtains
\begin{multline*}
J'(s) \leq - \| F_{A_{re^s}} \|_{L^2(N)}^2 - \| F_{A_{R e^{-s}}}\|_{L^2(N)}^2 \leq - \| F_{A_{re^{s'}}} \|_{L^2(N)}^2  -  \| F_{A_{R e^{-s''}}} \|_{L^2(N)}^2 \\
\leq - \delta_P ( F( re^{s'}) - F( R e^{-s''}) ) \leq -\delta_P ( F( r e^s) - F(R e^{-s})) = -\delta_P J(s).
\end{multline*}
It follows that $J(s)$ decays exponentially as $s$ increases and hence \eqref{eqn49} is proved.
\end{proof}

The above annulus lemma contains two special cases corresponding to the two special cases of Remark \ref{rem214}. In the first case when $N^- \cong (N^+)^{\rm op} \cong M_0$, the four-manifold $\N_{r, R}$ is an open subset of ${\bm R} \times M$. In this case we denote 
\beqn
\M_{r, R}:= \N_{r, R},\ {\rm where}\ N \cong M^{\rm double}.
\eeqn
In the second case when $N^- \cong (N^+)^{\rm op} \cong [0, \pi]\times \Sigma$ (whose boundary is two copies of $\Sigma$ instead of one), one has a diffeomorphism $N \cong S^1 \times \Sigma$. Although the four-manifold $\N_{r, R}$ is not isometric to ${\it Ann}(r, R) \times \Sigma$ there is a constant $a>0$ independent of large $r$ and $R$ such that $\N_{r, R}$ is an open subset of ${\it Ann}(r-a, R + a) \times \Sigma$. Then one obtains a corresponding version of annulus lemma for instantons over ${\it Ann}(r, R) \times \Sigma$. 

\begin{cor}\label{cor47} There exist $\epsilon>0$, $C>0$, and $r_0>0$ satisfying the following conditions. Let $\A$ be a solution to the ASD equation over ${\it Ann}(r, R) \times \Sigma$ with $r \geq r_0$ and $\log R - \log r \geq 4$ and assume $E(\A) < \epsilon$. Then for $1 \leq s \leq \frac{1}{2} (\log R - \log r)$ there holds
\beqn
E(\A; {\it Ann}(e^s r, e^{-s} R) \times \Sigma) \leq C e^{-\delta_P s} E(\A; {\it Ann}(r, R) \times \Sigma).
\eeqn
\end{cor}

In both of the two special cases, we would like to extend the annulus lemma to allow $r = 0$. For $R>0$, define
\beqn
\M_R:= (B_R^+ \times \Sigma) \cup ( (-R, R) \times M_0)
\eeqn
where we glue the common boundary $(-R, R) \times \Sigma$ in the obvious manner. The four-manifold can be formally viewed as the annulus $\M_{0, R}$ considered above. 

\begin{prop}\label{prop48}
There exist $\epsilon>0$, $C>0$ satisfying the following conditions. Let $\A$ be a solution to the ASD equation over $\M_R$ with $\log R \geq 4$ and $E(\A) < \epsilon$. Then for $s \geq 0$ there holds
\beq\label{exponential}
E(\A; \M_{R e^{-s}}) \leq C e^{-\delta_P s} E(\A; \M_R).
\eeq
\end{prop}

\begin{proof}
Similar to the proof of Proposition \ref{annulus}, for all $\rho \in [1, R]$, one can define a relative Chern--Simons action $F(A_\rho)$ which agrees with the local action whenever $\| F_{A_\rho}\|_{L^2(N)}^2 \leq \epsilon$. We claim that when $R e^{-s} \geq 1$ there holds
\beq\label{eqn413}
E(\A; \M_{R e^{-s}}) = - F(A_{R e^{-s}}).
\eeq
Notice that the connection $A_0:= \A|_{\{0\}\times M_0}$ defines a not necessarily flat connection $A_0^{\rm double}$ over $M^{\rm double}$ by doubling $A_0$ and by Lemma \ref{lemma21} one has
\beqn
E(\A; \M_{R e^{-s}}) = - {\it CS}^{\rm rel}({A_0^{\rm double}}, A_{Re^{-s}}).
\eeqn
Next, by Uhlenbeck compactness, if $\epsilon$ is sufficiently small, the restriction $\A|_{\M_2}$ is sufficiently close to a flat connection $\A'$ on $\M_2$. Let $A_0' \in {\mc A}(P_0)$ be the restriction of $\A'$ to $M_0 \times \{0\}\subset \M_2$ and let $B_0'\in {\mc A}(Q)$ be the restriction of $A_0'$ to $\partial M_0$. By applying a gauge transformation ${\bm g}$ to $\A'$, one can assume that the restriction of $\A'$ to $(-2, 2) \times M_0$ is $d_s + A_0'$ and the restriction of $\A'$ to $B_2^+ \times \Sigma$ is $d_{\bm H} + B_0'$. Moreover the restriction of ${\bm g}$ to each $N_\rho \cong N$ for $\rho \in (0, 2)$ is contained in the identity component ${\mc G}_0(P_N)$. Let $A_\rho'$ be the restriction of $\A'$ to $N_\rho \subset \M_2$ for $\rho \in [1, 2]$. As $\| A_\rho - A_\rho'\|_{L^4(N)}$ is small for $\rho \in [1, 2]$, by the definition of the local action (Definition \ref{defn42}) and the definition of $F(A_\rho)$, one has 
\beqn
F(A_\rho) = CS^{\rm rel}(A_1', A_\rho),\ \forall \rho \in [1, R].
\eeqn
Notice that $A_1'$ is the double of a flat connection on $M_0$. Hence
\beqn
{\it CS}^{\rm rel}({A_0^{\rm double}}, A_1') = 0
\eeqn
as it is the sum of an integral of the same differential form over two copies of the same three-manifold with boundary having the opposite orientations. Therefore
\beqn
E(\A; \M_{Re^{-s}}) = - {\it CS}^{\rm rel}(A_0^{\rm double}, A_{R e^{-s}}) = - {\it CS}^{\rm rel} (A_1', A_{Re^{-s}}) = - F(R e^{-s}).
\eeqn
Abbreviate $J(s) = - F( Re^{-s})$. By \eqref{eqn413}, \eqref{eqn411}, and the condition $R e^{-s} \geq 1$ one has
\beqn
J'(s) = - R e^{-s} \| F_{\A}\|_{L^2(N_{R e^{-s}})}^2 \leq - \| F_{A_{R e^{-s}}} \|_{L^2(N)}^2.
\eeqn
Assume $s \geq 1$. If $\| F_{A_{Re^{-s}}}\|_{L^2(N)}^2\leq \epsilon$, then by the isoperimetric inequality one has 
\beqn
J'(s) \leq - \delta_P J(s). 
\eeqn
If $\| F_{A_{R e^{-s}}} \|_{L^2(N)}^2 > \epsilon$, then one can find $s' \in (s-1, s) \subset [0, \log R]$ such that $\| F_{A_{R e^{-s'}}}\|_{L^2(N)}^2 < \epsilon$. Then 
\beqn
J'(s) \leq - \| F_{A_{R e^{-s}}}\|_{L^2(N)}^2 \leq - \| F_{A_{R e^{-s'}}}\|_{L^2(N)}^2 \leq - \delta_P J(s') \leq - \delta_P J(s).
\eeqn
This shows that $J(s)$ decays exponentially for $s \in [1, \log R]$. So for some $C>0$,
\beq\label{eqn414}
E(\A; \M_{R e^{-s}} ) \leq C e^{-\delta_P s} E(\A; \M_R), \ \forall s \in [0, \log R].
\eeq
In particular, 
\beqn
E(\A; \M_2) \leq C R^{-\delta_P} E(\A; \M_R).
\eeqn 
Then by the standard interior estimate for the ASD equation (see \cite[Theorem 2.3.7 {\&} 2.3.8]{Donaldson_Kronheimer}), one has 
\beqn
\| F_\A \|_{L^\infty(\M_1)} \leq C \sqrt{ E(\A; \M_2)} \leq C R^{-\frac{\delta_P}{2}} \sqrt{ E(\A; \M_R)}.
\eeqn
On the other hand, for $r \leq 1$, one has ${\rm Volume}(\M_r) \leq Cr$. Hence for $s \geq \log R$, $r= R e^{-s} \leq 1$, using the condition that $\delta_P \leq 1$, one obtains
\beq\label{eqn415}
E(\A; \M_r) \leq C r R^{-\delta_P} E(\A; \M_R)  \leq C e^{-s} R^{1- \delta_P} E(\A; \M_R) \leq C e^{-\delta_P s} E(\A; \M_R).
\eeq
Combining \eqref{eqn414} and \eqref{eqn415} we obtain \eqref{exponential}.
\end{proof}

Similarly one has the following monotonicity property of ASD equation over the product of a disk with the compact surface, although it can be proved by using the mean value estimate instead of the isoperimetric inequality. 

\begin{prop}\label{prop49}
There exist $\epsilon>0$, $C>0$ satisfying the following conditions. Let $\A$ be a solution to the ASD equation over $B_R \times \Sigma$ with $\log R \geq 4$ and $E(\A) < \epsilon$. Then for $s \geq 0$ there holds
\beqn
E(\A; B_{R e^{-s}} \times \Sigma) \leq C e^{-\delta_P s} E(\A; B_R\times \Sigma).
\eeqn
\end{prop}

%The exponential decay property for instantons over $\M$ is a simple consequence of Proposition \ref{annulus}. Moreover, we can prove the energy quantization property for instantons over $\M$.

%%\begin{proof}[Proof of Theorem \ref{quantization}]
%We claim that this statement holds for $\hbar$ being the $\epsilon$ of Proposition \ref{annulus}. Suppose not, then there exists a nonconstant ASD instanton $\A$ over $\M$ with $E(\A) < \epsilon$. Choose $R>0$ such that 
%\beqn
%E(\A; \M_R ) = \frac{1}{2} E(\A) < \frac{\epsilon}{2}. 
%\eeqn
%For all large positive number $T$, let $\vartheta_T: \M \to \M$ be the translation in the ${\bm R}$-direction by $T$ and %denote $\A_T:= \vartheta_T^* \A$. Then $\A_T$ is also an ASD instanton. Then by Annulus Lemma II, there is a constant $C>0$ %(independent of $T$) such that 
%\beqn
%E( \A_T; \M_{R, \infty}) \leq C r^{-\delta_P}. 
%\eeqn
%However, one has 
%\beqn
%E(\A; \M_R ) \leq E(\A_T; \M_{T - r, \infty}) \leq C (T - r)^{-\delta_P}.
%\eeqn
%This is a contradiction when $T$ is sufficiently large. 
%\end{proof}

\subsection{Diameter bound}\label{subsection43}

To ensure the convergence towards holomorphic curves on the boundary, we need a further diameter control. We first define the notion of diameter. Let $S\subset {\bm H}$ be an open subset and let $\A$ be a solution to the ASD equation over $S \times \Sigma$ identified with a triple $(B, \phi, \psi)$. Suppose for each $z \in S$ the connection $B(z) \in {\mc A}(Q)$ is contained in the domain of the Narasimhan--Seshadri map ${\it NS}_2$. Then $\A$ projects to a continuous map $u: S \to R_\Sigma$. We define 
\beqn
{\rm diam}(\A; S \times \Sigma):= {\rm diam}(u(S)):= \sup_{p, q \in S} d_{R_\Sigma} ( u(p), u(q)).
\eeqn
Clearly this notion is gauge invariant. Moreover, denote
\beqn
\M_S:= (S \times \Sigma) \cup ( (S \cap \partial {\bm H}) \times M_0))
\eeqn
where the two parts are glued along the common boundary $(S \cap \partial {\bm H}) \times \Sigma$. If $\A$ is a solution on $\M_S$ whose restriction to $S \times \Sigma$ is $d_S + \phi ds + \psi dt + B(z)$ such that all $B(z)$ is contained in the domain of the Narasimhan--Seshadri map ${\it NS}_2$, then we define 
\beqn
{\rm diam}(\A; \M_S):= {\rm diam} ( \A; S \times \Sigma). 
\eeqn

By using Theorem \ref{thm31} and \ref{thm33}, one has the following interior diameter estimate.

\begin{lemma}[Interior diameter bound]\label{diameter1} There exist $C>0$, $T>0$, and $\epsilon>0$ such that for any solution to the ASD equation over $B_{2R} \times \Sigma$ with $R\geq T$ and
\begin{align*}
&\ E(\A) \leq \epsilon,\ &\ \sup_{B_{2R}} \| F_{B(z)} \|_{L^2(\Sigma)} \leq \epsilon
\end{align*}
there holds
\beqn
{\rm diam}(\A; B_R \times\Sigma) \leq C \sqrt{ E(\A; B_{2R}\times \Sigma)}.
\eeqn
\end{lemma}

Next, we prove the following diameter estimate near the boundary. 

\begin{lemma}[Boundary diameter bound] \label{diameter2}
There exist $C>0$, $T>0$, and $\epsilon>0$ such that, for any solution to the ASD equation $\A$ over $\M_{3R} = \M_{B_{3R}}$ with $E(\A) \leq \epsilon$ and $R \geq T$ satisfying
\begin{align}\label{eqn416}
&\ \| F_\A \|_{L^\infty(\M_{3R})} \leq 1,\ &\ \sup_{B_{3R}^+} \| F_{B(z)} \|_{L^3(\Sigma)} \leq \epsilon
\end{align}
there holds 
\beqn
{\rm diam}(\A; \M_R ) \leq C \sqrt{ E(A; \M_{3R})}.
\eeqn
\end{lemma}

\begin{proof}
Let $R$ and $\A$ satisfy the assumption of this lemma with $\epsilon$ and $T$ undetermined. Suppose the restriction of $\A$ to $B_{3R}^+ \times \Sigma$ by $d_{\bm C} + \phi ds + \psi dt + B(z)$. Denote by $u: B_{3R}^+ \to R_\Sigma$ the holomorphic map defined by $z \mapsto \ov{\it NS}_2 (B(z))$. Define the segment
\beqn
Z_R:= \big\{ z = s + {\bm i} t\ |\ -R \leq s \leq R,\ t = R \big\}. 
\eeqn
Since $Z_R$ can be covered by a fixed number of radius $R$ disks contained in $B_{3R}^+$, by Lemma \ref{diameter1}, for $T$ large and $\epsilon$ small, one has
\beq\label{eqn417}
{\rm diam} ( \A; Z_R \times \Sigma) \leq C \sqrt{ E(\A; \M_{3R})}.
\eeq
Hence it remains to show that 
\beq\label{eqn418}
\sup_{-R \leq s \leq R} \sup_{0 \leq t_1, t_2 \leq R} d_{R_\Sigma} ( u(s + {\bm i} t_1), u(s + {\bm i} t_2)) \leq C \sqrt{ E(\A; \M_{3R})}.
\eeq

To estimate the distance between $u(s + {\bm i} t_1)$ and $u(s + {\bm i} t_2)$ for $0 \leq t_1, t_2 \leq R$, we consider the rescaled equation and use the estimate of Corollary \ref{cor36}. The restriction of $\A$ to $B_{3R}^+\times \Sigma$ can be pulled back to a solution $\A' = d_{\bm C} + \phi' ds + \psi' dt + B'(z)$ to the $R$-ASD equation over $B_3^+ \times \Sigma$ satisfying $E_R(\A') \leq \epsilon$ and
\begin{align*}
&\ \sup_{B_3^+} e_R^\infty \leq R^2,\ &\ \sup_{B_3^+} \| F_{B'(z)} \|_{L^3 (\Sigma)} \leq \epsilon. 
\end{align*}
Then by Corollary \ref{cor36}, when $T$ is sufficiently large and $\epsilon$ is sufficiently small one has
\beqn
\| \A_s' \|_{L^2(\{s + {\bm i} t \}\times \Sigma)} \leq \frac{C}{t} \sqrt{ E_R ( \A'; B_t(z)\times \Sigma)},\ \forall |s|\leq 1,\ 0 < t \leq 1.
\eeqn
Moreover, by Proposition \ref{prop48} for the case $N = M^{\rm double}$, for some constant $C>0$ independent of $\A$ and $R$ one has
\beqn
E_R (\A'; B_t(z) \times \Sigma) \leq E_R ( \A'; B_{2t}^+(s) \times \Sigma) \leq E(\A; \M_{B_{2Rt}^+(Rs)}) \leq C t^{\delta_P} \cdot E(\A; \M_{3R}),\  \forall t \in (0, 1].
\eeqn
Therefore one has
\beqn
\| \A_s' \|_{L^2(\{ s + {\bm i} t \}\times \Sigma)} \leq C t^{-1 + \frac{\delta_P}{2}} \sqrt{ E(\A; \M_{3R})},\ \forall t\in (0, 1].
\eeqn
Now for $z_1 = s + {\bm i} t_1$ and $z_2 = s + {\bm i} t_2$ with $|s|\leq 1$ and $0 < t_1\leq t_2 \leq 1$ we would like to estimate the distance between $u(z_1)$ and $u(z_2)$. Using a suitable gauge transformation we can assume $\psi'(s + {\bm i} t) = 0$ for all $t \in [0, 1]$. Then 
\begin{multline*}
\| B' (z_1) - B' (z_2) \|_{L^2(\Sigma)} \leq \int_0^1 \|\partial_\tau B' (s + {\bm i} \tau)  \|_{L^2(\Sigma)} d\tau = \int_0^1 \| \A_t' \|_{L^2(\Sigma)} d\tau \\
= \int_0^1 \| \A_s'\|_{L^2(\Sigma)} d\tau \leq C \sqrt{ E(\A; \M_{3R})} \int_0^1 \tau^{-1 + \frac{\delta_P}{2}} d\tau  \leq C \sqrt{E(\A; \M_{3R})}.
\end{multline*}
Then by Lemma \ref{lemma217}, for $|s|\leq 1$ and $t_1, t_2 \in (0, 1]$, one has 
\begin{multline*}
{\rm dist}(u(R s + {\bm i} R t_1), u(R s + {\bm i} Rt_2)) \leq \| {\it NS}_2(B(z_1)) - {\it NS}_2(B(z_2)) \|_{L^2(\Sigma)} \\
\leq \| {\it NS}_2 (B(z_1)) - B(z_1) \|_{L^2(\Sigma)} + \| B(z_1) - B(z_2) \|_{L^2(\Sigma)} + \| B(z_2) - {\it NS}_2 (B(z_2)) \|_{L^2(\Sigma)}\\
\leq C \big( \| F_{B(z_1)} \|_{L^2(\Sigma)} + \| F_{B(z_2)}\|_{L^2(\Sigma)} + \sqrt{E(\A; \M_{3R})}\big).
\end{multline*}
It is standard to use the energy to bound the norm $\| F_{B(z)}\|_{L^2(\Sigma)}$. Hence one has 
\beqn
\sup_{-R \leq s \leq R} \sup_{0 \leq t_1, t_2 \leq R} {\rm dist}(u(s + {\bm i} t_1), u( s + {\bm i} t_2 )) \leq C \sqrt{ E(\A; \M_{3R})}.
\eeqn
This proves \eqref{eqn418}. Then together with \eqref{eqn417} the lemma is proved.
\end{proof}

\subsection{Consequences of the annulus lemma and the diameter estimate}

We first prove the asymptotic behavior of solutions to the ASD equation. 

\begin{thm}[Asymptotic Behavior] \label{asymptotic} Let $\A$ be a solution to the ASD equation over $\M = {\bm R}\times M$ with $E(\A)< \infty$. Then there exists a point $(a_-, a_+) \in (\iota \times \iota)^{-1}(\Delta_{R_\Sigma}) \cong L_{M^{\rm double}}$ whose image in $R_\Sigma$ is denoted by $b_\infty$ such that the following conditions hold.
\begin{enumerate}

\item In the quotient topology of the configuration spaces ${\mc A}(P_0)/ {\mc G}(P_0)$ one has 
\beqn
\lim_{s \to + \infty}  [\A|_{\{\pm s\}\times M_0} ] = a_\pm \in L_M \subset {\mc A}(P_0) / {\mc G}(P_0).
\eeqn

\item In the quotient topology of the configuration space ${\mc A}(Q)/ {\mc G}(Q)$ one has 
\beqn
\lim_{z\to \infty} [ \A|_{\{z\}\times \Sigma} ] = b_\infty \in R_\Sigma \subset {\mc A}(Q) / {\mc G}(Q).
\eeqn
\end{enumerate}
\end{thm}

\begin{proof}
Denote the restriction of $\A$ to ${\it Ann}^+(r, \infty) \times \Sigma$ by $d_{\bm H} + \phi ds + \psi dt + B$. We first prove the convergence of the gauge equivalence class $[B(z)]$. By the strong Uhlenbeck compactness for Yang--Mills connections, we know that for each sequence $z_i \to \infty$, there is a subsequence (still indexed by $i$) for which $[B(z_i)]$ converges in ${\mc A}(Q) / {\mc G}(Q)$ to a limit in $R_\Sigma$. We would like to show that the subsequential limit is unique. By the annulus lemma (Proposition \ref{annulus}), there exist $C>0$ and $\delta_P>0$ such that for $r\geq 1$ there holds
\beq\label{eqn419}
E(\A; \M \setminus \M_r ) \leq C r^{-\delta_P}.
\eeq
Moreover, we claim that 
\beqn
\lim_{r \to \infty} \| F_\A \|_{L^\infty( \M \setminus \M_r)} = 0.
\eeqn
Indeed, if this is not the case, then a nontrivial ASD instanton over ${\bm C}\times \Sigma$, a nontrivial ASD instanton over ${\bm R}\times M$, or a nontrivial ${\bm R}^4$-instanton bubbles off at infinity, contradicting \eqref{eqn419}. Then by the diameter estimates (Lemma \ref{diameter1} and Lemma \ref{diameter2}), there exist a constant $C>0$ and a sufficiently large $r_0$ such that for all $\tau \geq \log r_0$, there holds
\beqn
{\rm diam} \big( \A; \M_{e^{\tau -1}, e^{\tau+1}} \big) \leq C \sqrt{ E(\A; \M_{e^{\tau-3}, e^{\tau + 3}} ) } \leq C e^{-\frac{\delta_P}{2} \tau}
\eeqn
(this is because the half annulus ${\it Ann}^+(e^{\tau-2}, e^{\tau+2})$ can be covered by a fixed number of half disks and a fixed number of disks with radii being comparable to $e^\tau$ which are contained in the half annulus ${\it Ann}^+(e^{\tau-3}, e^{\tau+3})$). Hence 
\beqn
\lim_{r \to \infty} {\rm diam}(\A; \M \setminus \M_r) = 0.
\eeqn
Therefore the subsequential limit of $[B(z)]$ in $R_\Sigma$ is unique. Denote the limit by $b_\infty$. 

Now we prove the convergence of $[ \A|_{\{s\}\times M_0}]$ as $s \to \pm \infty$. Abbreviate the restriction $\A_{\{s\}\times M_0}$ by $A(s)$. First, the finiteness of energy implies that
\beqn
\lim_{s \to \pm \infty} \| F_\A \|_{L^2([s-1, s+1]\times M_0)} = 0.
\eeqn
Hence by Uhlenbeck compactness, for any sequence $s_i \to \pm \infty$ there is a subsequence for which the sequence $[A(s_i)]$ converges to a limit in $L_M \subset {\mc A}(P_0)/ {\mc G}(P_0)$. Then any subsequential limit must be in $\iota^{-1}(b_\infty)$ where $\iota: L_M \immerse R_\Sigma$ is the Lagrangian immersion. We would like to show that as $s \to +\infty$ or $s\to -\infty$ the subsequential limit of $[A(s)]$ is unique. Indeed, if there are two sequences $s_i', s_i'' \to +\infty$ such that $[A(s_i)]$ and $[A(s_i')]$ converges to two different preimages of $\iota^{-1}(b_\infty)$, denoted by $a_+, a_+'$, then since the configuration space ${\mc A}(P_0) / {\mc G}(P_0 )$ is Hausdorff (see Lemma \ref{hausdorff} below), one can choose two disjoint neighborhoods $U, U'$ of $a_+$ and $a_+'$ and a sequence $s_i'' \to +\infty$ with $[A(s_i'')] \notin U' \cup U''$. Then a subsequence of $[A(s_i'')]$ converges to a limiting flat connection different from $a_+$ and $a_+'$. However since $b_\infty$ has at most two preimages, this cannot happen. Hence the subsequential limit of $[A(s)]$ is unique and hence $[A(s)]$ converges to a limit $a_\pm$ as $s \to \pm \infty$. 
\end{proof}

The following lemma is used in the previous proof.

\begin{lemma}\label{hausdorff}
Let $M_0$ be a three-manifold with boundary and $P_0 \to M_0$ be an $SO(3)$-bundle. Then the configuration space ${\mc A}(P_0)/ {\mc G}(P_0)$ is Hausdorff with respect to the quotient topology induced from the $C^0$-topology of ${\mc A}(P_0)$.
\end{lemma}

\begin{proof}
For $A_1, A_2 \in {\mc A}(P_0)$, define 
\beq\label{c0distance}
{\rm dist}_{C^0}(A_1, A_2):= \inf_{g \in {\mc G}(P_0)} \| g^* A_1 - A_2\|_{C^0(M_0)}.
\eeq
This clearly descends to a symmetric function on the quotient ${\mc A}(P_0)/ {\mc G}(P_0)$ satisfying the triangle inequality. We claim that this is a metric, namely, if $A_1$ and $A_2$ are not gauge equivalent, then ${\rm dist}_{C^0}(A_1, A_2) > 0$. Suppose this is not the case, then there exists a sequence of smooth gauge transformations $g_i$ such that $\| g_i^* A_1 - A_2 \|_{C^0 (M_0)} \to 0$. Since $g_i$ takes value in a compact group $SU(2)$, this implies that $\| g_i\|_{W_{A_1}^{1,\infty}}$ is uniformly bounded and so a subsequence of $g_i$ converges to a continuous gauge transformation $g$ on $P_0$. Moreover, in the weak sense, $g^* A_1 = A_2$. Since both $A_1$ and $A_2$ are smooth, $g$ is also smooth and hence $A_1$ is gauge equivalent to $A_2$, which contradicts our assumption. Hence ${\rm dist}_{C^0}$ is a metric on the configuration space and the lemma is proved. 
\end{proof}

Theorem \ref{asymptotic} is the analogue of the following results about the asymptotic behavior of instantons over ${\bm C} \times \Sigma$.

\begin{thm}\cite[Proof of Theorem 9.1]{Dostoglou_Salamon} \label{thm414} Let $\A$ be an ASD instanton over $({\bm C} \setminus B_R)\times \Sigma$ for some $R>0$ which can be written as $d_{\bm C} + \phi ds + \psi dt + B(z)$. Then there is a gauge equivalence class of flat connections $b_\infty\in R_\Sigma \subset {\mc A}(Q) / {\mc G}(Q)$ such that in the quotient topology one has
\beqn
\lim_{|z|\to \infty} [B(z)] = b_\infty.
\eeqn
\end{thm}

\begin{defn}\label{evaluation}
In the situation of Theorem \ref{asymptotic} resp. Theorem \ref{thm414}, the point $(a_+, a_-) \in L_N$\footnote{We choose $(a_+, a_-)$ rather than $(a_-, a_+)$ because we want to match the convention of evaluations of holomorphic disks with immersed Lagrangian boundary condition. See Definition \ref{defn28}.} resp. $b_\infty \in R_\Sigma$ is called the {\bf evaluation at infinity} of the solution $\A$ to the ASD instanton over $\M$ resp. $({\bm C} \setminus B_R)\times \Sigma$.  
\end{defn}

Finally we have the energy quantization property for instantons over $\M = {\bm R}\times M$.

\begin{thm}\label{quantization}
There exists $\hbar>0$ which only depends on the bundle $P \to M$ and which satisfies the following property. For any ASD instanton $\A$ over $\M$ there holds
\beqn
E(\A) \in \{0\}\cup [\hbar, +\infty).
\eeqn
\end{thm}

\begin{proof}
We claim that $\hbar$ being the $\epsilon$ of Proposition \ref{prop48} satisfies the condition of this theorem. Indeed, suppose there is an ASD instanton over $\M$ with $E(\A) < \hbar$. For all $r>0$ and $R = e^s r$ with $s$ sufficiently large, the restriction of $\A$ to $\M_R$ satisfies the hypothesis of Proposition \ref{prop48}. Then 
\beqn
E(\A; \M_r) = E(\A; \M_{R e^{-s}} ) \leq C e^{-\delta_P s} E (\A; \M_R) \leq C \epsilon e^{-\delta_P s}. 
\eeqn
As $s$ can be arbitrarily large, one has $E(\A; \M_r) = 0$. Hence $\A$ is a trivial solution. 
\end{proof}

\section{Compactness modulo energy blowup}\label{section5}

In this section we prove a compactness theorem modulo bubbling. It is the analogue of Theorem \ref{thm33} for the domain being $\M = {\bm R}\times M$. We first set up the problem. Recall that for any open subset $S \subset {\bm H}$, $\M_S$ denotes 
\beqn
\M_S = ( S \times \Sigma ) \cup ( ( S \cap \partial {\bm H}) \times M_0)
\eeqn 
where the two parts are glued over the common boundary $( S\cap \partial {\bm H}) \times \Sigma$. Let $S_i \subset {\bm H}$ be an exhaustive sequence of open subsets. They may or may not intersect the boundary of ${\bm H}$. Let $\rho_i \to +\infty$ be a sequence of positive numbers diverging to infinity. Recall that $\varphi_\rho: {\bm H} \to {\bm H}$ is the map corresponding to multiplying by $\rho$. In the proof of the main theorem of this paper we will only use the case that $S_i = {\bm H}$. 

\begin{thm}\label{thm51}
Suppose $\A_i$ is a sequence of ASD instantons on $\M_{\varphi_{\rho_i}(S_i)}$ with
\beqn
\limsup_{i \to \infty} E (  \A_i; \M_{\varphi_{\rho_i}(S_i)} ) < +\infty. 
\eeqn
Then there exist a subsequence (still indexed by $i$) and a holomorphic map 
\beqn
\tilde {\bm u}_\infty = ( u_\infty, W_\infty, \gamma_\infty, m_\infty),
\eeqn
from ${\bm H}$ to $R_\Sigma$ with mass (see Definition \ref{defn28}) satisfying the following conditions. 
\begin{enumerate}

\item For each $z \in {\bm H}$, for sufficiently small $r>0$, the limit 
\beqn
\lim_{i \to \infty} E \big( \A_i; \M_{\varphi_{\rho_i}(B_r^+(z))} \big)
\eeqn
exists. Moreover, one has the convergence
\beq\label{eqn51}
\lim_{r \to 0} \lim_{i\to \infty} E \big( \A_i; \M_{\varphi_{\rho_i}(B_r^+(z))} \big) = m_\infty (z).
\eeq

\item Over $\varphi_{\rho_i}(S_i) \times \Sigma$ we write $\A_i = d_{\bm H} + \phi_i ds + \psi_i dt + B_i$. Then for any precompact open subset $K \subset {\bm H} \setminus W_\infty$, for $i$ sufficiently large, $B_i (z)$ is in the domain of the Narasimhan--Seshadri map ${\it NS}_2$ (see Definition \ref{nsmap}) for all $z \in \varphi_{\rho_i}(K)$, hence projects down to a sequence of holomorphic maps $u_i: \varphi_{\rho_i}(K) \to R_\Sigma$. Further $u_i \circ \varphi_{\rho_i}$ converges to $u_\infty|_K$ in the sense of Definition \ref{defn25}. 

\item There is no energy lost in the following sense: 
\beq\label{eqn52}
\lim_{R \to \infty} \lim_{i \to \infty} E \big( \A_i; \M_{\rho_i R} \big) = E(u_\infty) + \int_{\bm H} m_\infty. 
\eeq
\end{enumerate}
\end{thm}

These two theorems motivates the following notion of convergence. For simplicity we restrict to the case $S_i = {\bm H}$.

\begin{defn}\label{defn52}
Let $\A_i$ be a sequence of ASD instantons over $\M$. Let $\tilde {\bm u}_\infty$ be a holomorphic map from ${\bm H}$ to $R_\Sigma$ with mass in the sense of Definition \ref{defn28}. Suppose $\{\rho_i\}$ be a sequence of positive numbers diverging to $\infty$. $\A_i$ is said to converge to $\tilde {\bm u}_\infty$ along with $\{ \rho_i\}$ if conditions (a), (b), and (c) of Theorem \ref{thm51} hold. 
\end{defn}

\subsection{Energy blowup threshold}\label{subsection51}

First we define the notion of energy blowup.

\begin{defn}[Energy blowup] Let $\rho_i, \A_i$ satisfy the assumptions of Theorem \ref{thm51}. For each $w \in {\bm H}$, we say that {\it energy blows up at $w$} if 
\beqn
\lim_{r\to 0} \limsup_{i \to \infty} E \big( \A_i; \M_{\varphi_{\rho_i}(B_r^+ (w))} \big) > 0.
\eeqn
\end{defn}

\begin{lemma}[Energy blowup threshold]  \label{lemma54}
There exists $\hbar > 0$ satisfying the following property. Let $\rho_i$ and $\A_i$ be as in Theorem \ref{thm51}. Then for all $w \in {\bm H}$ there holds
\beq\label{eqn53}
\lim_{r \to 0} \limsup_{i \to \infty} E\big( \A_i; \M_{\varphi_{\rho_i}( B_r^+(w))} \big) \in \{0\} \cup [\hbar, +\infty).
\eeq
\end{lemma}

\begin{proof}
When energy blows up at $w \in {\rm Int} {\bm H}$, \eqref{eqn53} follows from the bubbling analysis in \cite[Proof of Theorem 9.1]{Dostoglou_Salamon} (or via Proposition \ref{prop49} as argued below). We consider the case when $w \in \partial {\bm H}$. We claim that $\hbar$ being the $\epsilon$ of Proposition \ref{prop48} satisfies the property. Indeed, suppose on the contrary
\beqn
0 < \lim_{r \to 0} \limsup_{i \to \infty} E(\A_i; \M_{\varphi_{\rho_i}(B_r^+(w))})  < \epsilon.
\eeqn
Then there exists $r_0 >0$ such that 
\beqn
\limsup_{i \to \infty} E(\A_i; \M_{\varphi_{\rho_i}(B_{r_0}^+(w))} ) < \epsilon.
\eeqn
Then by Proposition \ref{prop48} for $r \ll r_0$, one has
\beqn
E(\A_i; \M_{\varphi_{\rho_i }(B_r^+(w))}) \leq C \left( \frac{r}{r_0} \right)^{\delta_P}. 
\eeqn
This contradicts the assumption that energy blows up at $w$.  
\end{proof}

\subsection{Locating the blowup points}

One can use the existence of energy blowup threshold to select a subsequence for which energy blowup happens at only finitely many points. However the current case is slightly more involved than the case of pseudoholomorphic curves and the case of Theorem \ref{thm33}. This is because near the boundary of ${\bm H}$, we do not have the equivalence between energy blowup and  energy density blowup.

\begin{lemma}
There exist a subsequence (still indexed by $i$) and a finite subset $W_\infty \subset {\bm H}$ such that (for this subsequence) energy blows up exactly at points in $W_\infty$. Moreover, for each $w\in W_\infty$ and each positive integer $n$, the limit
\beq\label{eqn54}
\lim_{i \to \infty} E \left( \A_i; \M_{\varphi_{\rho_i}(B_{1/n}^+(w))} \right) \in [\hbar, +\infty)
\eeq
exists.
\end{lemma}

\begin{proof}
We construct, inductively, for each positive integer $n$, the following objects.
\begin{itemize}
\item  A subsequence (still indexed by $i$) of the subsequence constructed in the $(n-1)$-st step. (For $n=1$, a subsequence of the original sequence.)

\item  A finite subset $W_n \subset {\bm H}$ and  $r_w^\bullet, r_w^\circ \in (0, \frac{1}{n})$ for each $w \in W_n$ with $r_w^\bullet < r_w^\circ$.
\end{itemize}
These objects are required to satisfy the following conditions. 
\begin{enumerate}
\item[(C1)]  For $w', w''\in W_n$, $w' \neq w''$, one has 
\beqn
\ov{B_{r_{w'}^\circ}^+(w')} \cap \ov{ B_{r_{w''}^\circ}^+(w'') } = \emptyset.
\eeqn

\item[(C2)]  If $n \geq 2$, then 
\beqn
\bigcup_{w \in W_n} \ov{B_{r_w^\circ}^+(w)} \subset \bigcup_{w \in W_{n-1}} B_{r_w^\circ}^+(w).
\eeqn

\item[(C3)]  For the subsequence, energy does not blow up at any point in 
\beqn 
{\bm H} \setminus \bigcup_{w \in W_n} \ov{B_{r_w^\bullet}^+ (w) }.
\eeqn

\item[(C4)]   For each $w \in W_n$, the limit (for the subsequence)
\beqn
\lim_{i \to \infty} E(\A_i; \M_{\varphi_{\rho_i} (B_{r_w^\bullet}^+(w))} ) \in [\hbar, +\infty)
\eeqn
exists.
\end{enumerate}

We start the construction. Suppose $w \in {\bm H}$ is a point where the energy blows up. Choose some $r_w^\bullet \in (0, 1)$, say $r_w^\bullet = \frac{1}{2}$. Then one can choose a subsequence (still indexed by $i$) such that the limit
\beqn
\lim_{i \to \infty} E \left( \A_i; \M_{\varphi_{\rho_i}(B_{r_w^\bullet}^+(w))} \right) \in [\hbar, +\infty)
\eeqn
exists. (Notice that for this subsequence, it is possible that energy no longer blows up at $w$.) Suppose we have find a subsequence (still indexed by $i$), finitely many points $w_1, \ldots, w_s$, and positive numbers $r_{w_1}^\bullet, \ldots, r_{w_s}^\bullet \in (0, 1)$ such that

\begin{enumerate}
\item  For $k \neq l$, there holds
\beqn
\ov{B_{r_{w_k}^\bullet}^+(w_k)} \cap \ov{ B_{r_{w_l}^\bullet}^+(w_l) } = \emptyset.
\eeqn

\item  For each $k$, the limit (for the subsequence)
\beqn
\lim_{i \to \infty} E(\A_i; \M_{\varphi_{\rho_i} (B_{r_{w_k}^\bullet}^+(w_k))} ) \in [\hbar, +\infty)
\eeqn
exists.
\end{enumerate}
Then consider for the current subsequence if energy blows up at any point in 
\beqn
{\bm H} \setminus \bigcup_{k=1}^s \ov{B_{r_{w_k}^\bullet}^+(w_k)}.
\eeqn
If there is some blowup point, then we can find a further subsequence (still indexed by $i$), a point $w$ not in the union of all $\ov{B_{r_{w_k}^\bullet}^+(w_k)}$, and a radius $r_w^\bullet \in (0, 1)$ such that 
\beqn
\ov{ B_{r_w^\bullet}^+(w)} \cap \ov{ B_{r_{w_k}^\bullet}^+(w_k) } = \emptyset,\ \forall k = 1, \ldots, s,
\eeqn
and such that 
\beqn
\lim_{i \to \infty} E\left( \A_i; \M_{\varphi_{\rho_i}( B_{r_w^\bullet}^+(w))} \right) \in [\hbar, +\infty)
\eeqn
exists. By the uniform energy bound and the fact that the regions $\ov{B_{r_k^\bullet}^+(w_k)}$ do not mutually intersect we know the induction must stops at finite time. The energy does not blow up at any point outside. Moreover, we can find slightly bigger numbers
\beqn
r_w^\circ \in (r_w^\bullet, 1)
\eeqn
such that the closed sets $\ov{B_{r_{w_k}^\circ}^+(w_k)}$ for all $w_k$ still have no mutual intersections. This finishes the inductive construction for the $n=1$ case.

Now suppose we have constructed a subsequence (still indexed by $i$), a finite subset $W_n$, and numbers $r_w^\bullet, r_w^\circ$ satisfying the required conditions [C1]---[C4]. Then for the current subsequence, energy blowup can only happen at points in 
\beqn
\bigcup_{w \in W_n} \ov{B_{r_w^\bullet}^+(w)}.
\eeqn
Then following similar procedure, one can find a further subsequence, a finite subset $W_{n+1}$ of the above closed set, and numbers $r_w^\bullet, r_w^\circ \in (0, \frac{1}{n+1})$ satisfying [C1]---[C4].

Notice that we have been choosing a subsequence for each $n$ from the subsequence constructed for $n-1$. Using the diagonal trick we choose a further subsequence. Then the conditions [C1]---[C4] are satisfied for this subsequence and all $n$. Moreover, by the uniform bound on energy, one has the bound for all $n$
\beq\label{eqn55}
\# W_n \leq N:= \left\lfloor \frac{ \displaystyle \limsup_{i \to \infty} E(\A_i)}{\hbar} \right\rfloor.
\eeq

Now consider the limiting set
\beqn
W_\infty:= \lim_{n \to \infty} \bigcup_{w \in W_n} \ov{ B_{r_w^\circ}^+ (w) } := \bigcap_{n\geq 1} \bigcup_{w \in W_n}  \ov{ B_{r_w^\circ}^+(w) } \subset {\bm H}.
\eeqn
By our construction, for the current subsequence, the energy does not blow up at any point not in $W_\infty$. Moreover, we claim that $W_\infty$ is finite and its number of elements is at most $N$ (defined in \eqref{eqn55}). Suppose on the contrary there exist distinct elements $y_0, \ldots, y_N \in W_\infty$. Choose $n$ sufficiently large such that $\frac{2}{n}$ is smaller than the minimal distance between two distinct $y_j$ and $y_l$. Then there must be distinct elements $w_0, \ldots, w_N \in W_n$ such that 
\beqn
y_j \in \ov{B_{r_{w_j}^\circ}^+(w_j)},\ j = 0, 1, \ldots, N.
\eeqn
This contradicts the bound \eqref{eqn55}.

Therefore for the current subsequence, energy blowup can only happen at points in $W_\infty$. Now for each $w \in W_\infty$, choose a subsequence inductively for each $n$. First consider the limit
\beqn
\limsup_{i \to \infty} E(\A_i; \M_{\varphi_{\rho_i}(B_1^+(w))} ).
\eeqn
If it is less than $\hbar$, then remove $w$ from $W_\infty$. If not, then choose a subsequence (still indexed by $i$) such that 
\beqn
\lim_{i \to \infty} E(\A_i; \M_{\varphi_{\rho_i}(B_1^+(w))} ) \in [\hbar, +\infty).
\eeqn
Suppose we have find a subsequence for $n$ such that the limit 
\beqn
\lim_{i \to \infty} E(\A_i; \M_{\varphi_{\rho_i} (B_{1/k}^+(w))} ) \in [\hbar, +\infty)
\eeqn
exists for all $k = 1, \ldots, n$. Then consider the limit 
\beqn
\limsup_{i \to \infty} E(\A_i; \M_{\varphi_{\rho_i} (B_{1/(n+1)}^+(w))} ).
\eeqn
If it is less than $\hbar$, then remove $w$ from $W_\infty$ because $w$ is no longer a blowup point for the current subsequence. If not, then choose a subsequence (still indexed by $i$) such that 
\beqn
\lim_{i \to \infty} E(\A_i; \M_{\varphi_{\rho_i} (B_{1/(n+1)}^+(w))} ) \in [\hbar, +\infty)
\eeqn
exists. Then using the diagonal trick, one can choose a subsequence (still indexed by $i$) and possibly remove certain elements from $W_\infty$ such that for each $w\in W_\infty$ which has not been removed the limit \eqref{eqn54} exists.
\end{proof}

\subsection{Constructing the limiting holomorphic curve}

Next we construct the limiting holomorphic curve. For any precompact open subset 
\beqn
K \subset {\bm H} \setminus W_\infty
\eeqn
(which may intersect the boundary of ${\bm H}$), we claim that 
\beqn
\lim_{i \to \infty} \sup_{z \in \varphi_{\rho_i} (K)} \| F_{B_i(z)}\|_{L^2(\Sigma)} = 0.
\eeqn
Indeed, if this is not true, then some subsequence will bubble off an ${\bm R}^4$-instanton, an instanton over ${\bm C}\times \Sigma$, or an instanton over ${\bm M}$, contradicting that no energy blows up at points away from $W_\infty$. Then for $i$ sufficiently large, $B_i(z)$ is in the domain of the Narasimhan--Seshadri map $\ov{\it NS}_2: {\mc A}_\epsilon^{1,2}(Q) \to R_\Sigma$. By Proposition \ref{prop219}, the map 
\beqn
u_i: \varphi_{\rho_i}( K ) \to R_\Sigma,\ u_i(z):= \ov{\it NS}_2(B_i(z))
\eeqn
is holomorphic. Denote the reparametrized holomorphic maps by
\beqn
u_i':= u_i \circ \varphi_{\rho_i}: K \to R_\Sigma.
\eeqn
One can choose a subsequence (still indexed by $i$) and an exhausting sequence of precompact open subsets $K_i \subset {\bm H} \setminus W_\infty$ such that $u_i'$ is defined over $K_i$.

\begin{lemma}\label{lemma56} 
There is a subsequence of $u_i'$ which converges to a holomorphic map $u_\infty: {\bm H} \setminus W_\infty \to R_\Sigma$ in the sense of Definition \ref{defn25}. Moreover, the energy density function $e_{\rho_i}$ of $\A_i'$ converges to the energy density of $u_\infty$ in $C^0_{\rm loc}( {\rm Int} {\bm H} \setminus W_\infty)$. 
\end{lemma}

\begin{proof}
By Theorem \ref{thm33}, a subsequence of $u_i': K_i \to R_\Sigma$ (still indexed by $i$) converges to a holomorphic map $u_\infty: { \rm Int}{\bm H} \setminus W_\infty \to R_\Sigma$ in $C^\infty_{\rm loc}({\rm Int}{\bm H}\setminus W_\infty)$ and the rescaled energy density function converges in $C^0_{\rm loc}( {\rm Int} {\bm H} \setminus W_\infty)$ to the energy density of $u_\infty$. To prove the boundary convergence, we need to verify that there is no diameter blowup along the boundary. Indeed, by the fact that no energy blowup happens in ${\bm H} \setminus W_\infty$ and Lemma \ref{diameter2}, there holds
\beqn
\lim_{r \to 0} \limsup_{i \to \infty} {\rm diam} (u_i'(B_r^+(z))) = 0,\ \forall z \in \partial {\bm H} \setminus W_\infty.
\eeqn
Then by Proposition \ref{prop26} the map $u_\infty$ extends continuously to the boundary and $u_i'$ converges to $u_\infty$ in the sense of Definition \ref{defn25}. 
\end{proof}

We would like to use Gromov's removal of singularity theorem to extend the limiting holomorphic map over the bubbling points. We first show that the limit $u_\infty$ satisfies the Lagrangian boundary condition. 

\begin{lemma}\label{lemma57}
The holomorphic map $u_\infty$ maps the boundary of ${\bm H} \setminus W_\infty$ into $\iota(L_M)$. 
\end{lemma}

\begin{proof}
For each $s \in \partial {\bm H} \setminus W_\infty$, denote $A_i'(s) = \A_i|_{\{ \rho_i s\} \times M_0 }$. Since there is no energy blow up at $s$, one has 
\beqn
\lim_{i \to \infty} \| F_{A_i'(s)}  \|_{L^2(M_0)} = 0.
\eeqn
Then there is a subsequence of $A_i'(s)$ (still indexed by $i$) which converges modulo gauge transformation to a flat connection $A_\infty(s)$ on $M_0$. Let $B_i'(s)$ be the boundary restriction of $A_i'(s)$. Since the map ${\mc A}(P_0)/ {\mc G}(P_0) \to {\mc A}(Q)/ {\mc G}(Q)$ induced from boundary restriction is continuous, one has
\beqn
\lim_{i \to \infty} [B_i'(s)] = \lim_{i \to \infty} [ A_i'(s)|_\Sigma] =  \Big( \lim_{i \to \infty} [A_i'(s) ]\Big)|_\Sigma \in \iota(L_M).
\eeqn
On the other hand, since $\| F_{B_i'(s)} \|_{L^\infty(\Sigma)} \to 0$, the point $u_i(s) \in R_\Sigma$ has a flat connection representative which is arbitrarily $L^2$-close to $B_i'(s)$. Since $u_i(s)$ converges to $u_\infty(s)$, $u_\infty(s)$ agrees with the above limit which is in $\iota(L_M)$.
\end{proof}

\begin{lemma}\label{lemma58}
There is a subsequence (still indexed by $i$) such that for all $s \in \partial {\bm H} \setminus W_\infty$, the sequence $[A_i'(s)]$ converges in ${\mc A}(P_0)/ {\mc G}(P_0)$ to a flat connection on $P_0 \to M_0$.
\end{lemma}

\begin{proof}
For each connected component $I_\alpha \subset \partial {\bm H} \setminus W_\infty$, pick one point $s_\alpha\in I_\alpha$. Then for the finitely many chosen points $z_\alpha$, we can find a subsequence of $\A_i$, still indexed by $i$, such that $[A_i'(s_\alpha)]$ converges to \textcolor{blue}{some limit $[A(s_\alpha)]$} for all $\alpha$. Then consider the sets
\beqn
I_\alpha^* = \Big\{ s \in I_\alpha\ |\ \textcolor{black}{[A_i'(s)]} \ {\rm converges\ in\ } {\mc A}(P_0)/ {\mc G}(P_0) \Big\}.
\eeqn
Then $s_\alpha \in I_\alpha^*$. Notice that, by the proof of Lemma \ref{lemma57}, for each $s \in I_\alpha^*$ one has
\beq\label{eqn56}
u_\infty(s) = \iota( \lim_{i \to \infty} \textcolor{black}{[ A_i' (s)])}.
\eeq
Moreover, if $u_\infty(s)$ is not a double point of $\iota(L_M)$, then $s \in I_\alpha^*$. The lemma follows if we can show that $I_\alpha^*$ is both open and closed. 

We prove $I_\alpha^*$ is open. Suppose this is not the case. Then there exist $s^* \in I_\alpha^*$ and a sequence $s_k \notin I_\alpha^*$ with $s_k \to s^*$. Since $u_\infty$ is continuous, $u_\infty(s_k)$ converges to $u_\infty(s^*)$. Since $u_\infty(s_k)$ is a double point for all $k$ and the set of double points is discrete, $u_\infty(s^*) = u_\infty(s_k)$ for sufficiently large $k$. Assume $\iota^{-1}(u_\infty(s^*)) = \{ a', a''\} \subset L_M$ and assume that \textcolor{black}{$[A_{i}' (s^*)]$} converges to $a'$. Then for each $k$, there is a subsequence of \textcolor{black}{$[A_{i}' (s_k)]$} which converges to $a''$. By Lemma \ref{hausdorff} which says the space ${\mc A}(P_0)/ {\mc G}(P_0)$ is Hausdorff, we can choose disjoint open neighborhoods ${\mc U}', {\mc U}'' \subset {\mc A}(P_0)/ {\mc G}(P_0)$. Then for each $k$, we can choose $i_k$ inductively such that 1) $i_{k+1} > i_k$; 2) $\textcolor{black}{[A_{i_k}' (s^*)]} \in {\mc U}'$; 3) $\textcolor{black}{[A_{i_k}' ( s_k )]} \in {\mc U}''$. Then there exists a point $w_k$ between $s^*$ and $s_k$ such that $\textcolor{black}{[A_{i_k}' (w_k)]} \notin {\mc U}' \cup {\mc U}''$. Hence a subsequence (still indexed by $k$) of \textcolor{black}{$[A_{i_k}' (w_k)]$} converges to a point in $L_M$ which is neither $a'$ or $a''$. Denote \textcolor{black}{$[B_{i_k}'(w_k)]:= [A_{i_k}' (w_k)]|_{\partial M_0}$}. Then this convergence implies 
\beq\label{eqn57}
\lim_{k \to \infty} u_{i_k}'(w_k) = \lim_{k \to \infty} \textcolor{black}{[B_{i_k}'(w_k) ]} \notin \iota( {\mc U}' \cup {\mc U}'' \cap L_M)
\eeq
which is different from $u_\infty(s^*)$. This contradicts the convergence $u_i'$ towards $u_\infty$ and the continuity of $u_\infty$. This proves the openness of $I_\alpha^*$.

We prove $I_\alpha^*$ is closed. Suppose this is not the case. Then there exists $s^* \notin I_\alpha^*$ and a sequence $s_k \in I_\alpha^*$ with $s_k \to s^*$. Then $u_\infty(s^*)$ is a double point of $\iota(L_M)$. Let ${\mc U}', {\mc U}'' \subset {\mc A}(P_0)/ {\mc G}(P_0)$ be disjoint open neighborhoods of the two preimages of $u_\infty(s^*)$. Then there must be a subsequence of $s_k$ (still indexed by $k$), and one of ${\mc U}'$ and ${\mc U}''$, say ${\mc U}'$, such that for all $k$, 
\beqn
\lim_{i \to \infty} \textcolor{black}{[A_{i}' (s_k)]} \in {\mc U}'.
\eeqn
By \eqref{eqn56} and the non-convergence of \textcolor{black}{$[A_{i}' (s^*)]$}, there must also be a subsequence of $i$ (still indexed by $i$), such that 
\beqn
\lim_{i \to \infty} \textcolor{black}{[A_{i}' (s^*)]} \in {\mc U}''.
\eeqn
Then for each $k$, one can find $i_k$ such that 
\begin{align*}
&\ \textcolor{black}{[A_{i_k}' (s_k)]} \in {\mc U}',\ &\ \textcolor{black}{[A_{i_k}' (s^*)]} \in {\mc U}''.
\end{align*}
Choose a sequence $i_k$ such that $i_{k+1} > i_k$. Then since ${\mc U}'$ and ${\mc U}''$ are disjoint, one can find for each $k$ a point $w_k$ between \textcolor{black}{$s^*$ and $s_k$} with 
\beqn
\textcolor{black}{[A_{i_k}' (w_k)]} \notin {\mc U}' \cup {\mc U}''.
\eeqn
Denote \textcolor{black}{$[B_{i_k}' (w_k)] = [A_{i_k}' (w_k)]|_{\partial M_0}$}. Then \eqref{eqn57} still holds and one derives the same contradiction. This proves the closedness of $I_\alpha^*$. 
\end{proof}

It follows from Lemma \ref{lemma58} that the boundary map can be defined.

\begin{cor}
There exist a map $\gamma_\infty: \partial {\bm H} \setminus W_\infty \to L_M$ and a subsequence (still indexed by $i$) such that
\beq\label{eqn58}
\lim_{i \to \infty} \textcolor{black}{[A_{i}' (s)]} = \gamma_\infty(s),\ \forall s\in \partial {\bm H} \setminus W_\infty
\eeq
and
\beqn
u_\infty |_{\partial {\bm H} \setminus W_\infty } = \iota\circ \gamma_\infty.
\eeqn
\end{cor}

\begin{lemma}\label{lemma510}
$\gamma_\infty$ is continuous.
\end{lemma}

\begin{proof}
One only needs to show the continuity at points whose image is a double point of $\iota(L_M)$. The argument is similar as in the proof of Lemma \ref{lemma58}, using the Hausdorffness of the configuration space ${\mc A}(P_0)/ {\mc G}(P_0)$. The details are left to the reader. 
\end{proof}

\textcolor{black}{Notice that by Lemma \ref{lemma29}, the holomorphic map $u_\infty$ is smooth along the boundary of ${\bm H} \setminus W_\infty$ and $\iota_\infty$ is locally the boundary restriction of $u_\infty$. Hence $\iota_\infty$ is also smooth. Further, by the convergence of $e_{\rho_i}$ to the energy density of $u_\infty$ given by Lemma \ref{lemma56} and the uniform energy bound on $\A_i$, one has 
\beqn
E(u_\infty) \leq \limsup_{i \to \infty} E(\A_i) < \infty.
\eeqn
}

\begin{cor}\label{cor511}
The holomorphic map $u_\infty: {\bm H} \setminus W_\infty \to R_\Sigma$ extends to a holomorphic map with immersed Lagrangian boundary condition with mass, denoted by 
\beqn
\tilde {\bm u}_\infty = (u_\infty, \gamma_\infty, W_\infty, m_\infty).
\eeqn
\end{cor}

\begin{proof}
For each point $w \in W_\infty \cap {\rm Int} {\bm H}$, by Gromov's removal of singularity theorem, $u_\infty$ extends smoothly over $w$. On the other hand, given $w \in W_\infty \cap \partial {\bm H}$ and consider a punctured neighborhood of $w$ that is biholomorphic to a strip $(0, +\infty) \times [0, \pi]$ with coordinates $(s, t)$. Then by the standard elliptic estimate for holomorphic maps with Lagrangian boundary condition, one can prove that the length of the paths $u_\infty(s, \cdot)$ converges to zero. Hence the pair of points $(\gamma_\infty(s, 0), \gamma_\infty(s, \pi))$ is either close to the diagonal of $L_M$ or is close to an ordered double point. In either case one can define a local action for the path $u_\infty(s, \cdot)$ and this path satisfies the isoperimetric inequality similar to \cite[Theorem 4.4.1 (ii)]{McDuff_Salamon_2004}. The isoperimetric inequality implies that as $s \to +\infty$, the paths $u_\infty(s, \cdot)$ converge to a constant path. Hence $u_\infty$ extends continuously to $w$. 
\end{proof}

\subsection{Conservation of energy}

The last step of proving Theorem \ref{thm51} is to prove its assertion (a) and assertion (c). For $r>0$, define 
\beqn
S_r:= B_{1/r}^+ \setminus \bigcup_{w \in W_\infty} \ov{B_r^+(w)} \subset {\bm H},
\eeqn
which is an open subset of ${\bm H}$. When $r= 0$, define 
\beqn
S_0:= {\bm H} \setminus W_\infty.
\eeqn
\textcolor{black}{For each $\delta>0$ define 
\beqn
S_{r, \delta}:= S_r \cap \{ s + {\bm i} t \in {\bm H}\ |\ t < \delta \}.
\eeqn}
The conservation of energy will follow from the following proposition.

\begin{prop}\label{prop512}
For any $\epsilon>0$, there exists $r_\epsilon>0$ such that for all $r \in (0, r_\epsilon]$, one has 
\beqn
\limsup_{i \to \infty} \big| E( \A_i; \M_{\varphi_{\rho_i}(S_r)}) - E(u_\infty; S_r) \big| \leq \epsilon.
\eeqn
\end{prop}

\textcolor{black}{
The proof is provided in the next subsection. One corollary of this proposition and the interior convergence of the energy density (see Lemma \ref{lemma56}) is that no energy escape from the finite part $M_0$. 
}

\textcolor{black}{
\begin{cor}\label{cor513}
For any $r>0$, one has 
\beqn
\lim_{\delta \to 0} \limsup_{i \to \infty} E(\A_i; \M_{\varphi_{\rho_i}(S_{r, \delta})} ) = 0.
\eeqn
\end{cor}}

\begin{proof}
\textcolor{black}{Suppose this is not the case. Then there exist $r_0 >0$ and $\epsilon>0$ such that
\beqn
\lim_{\delta \to 0} \limsup_{i \to \infty} E(\A_i; \M_{\varphi_{\rho_i}(S_{r_0, \delta})} ) >   2 \epsilon.
\eeqn
Choose $r$ smaller than $r_0$ and $r_\epsilon$ which is the one from Proposition \ref{prop512}. For the holomorphic map $u_\infty$, choose $\delta_0$ such that 
\beqn
E(u_\infty; S_{r, \delta_0}) < \epsilon.
\eeqn
Then one has 
\beqn
\limsup_{i \to \infty} E(\A_i; \M_{\varphi_{\rho_i}(S_{r, \delta_0})} ) \geq \limsup_{i \to \infty} E( \A_i; \M_{\varphi_{\rho_i}(S_{r_0, \delta_0})} )  > 2 \epsilon.
\eeqn
Then choose a subsequence (still indexed by $i$) such that for all $i$ there holds
\beqn
E(\A_i; \M_{\varphi_{\rho_i}(S_{r, \delta_0})} ) > 2\epsilon.
\eeqn
Then one has
\beqn
E(\A_i; \M_{\varphi_{\rho_i}(S_r)} ) = E(\A_i; \M_{\varphi_{\rho_i}(S_{r, \delta_0})} ) + E(\A_i; \varphi_{\rho_i}(S_r  \setminus S_{r, \delta_0}) \times \Sigma).
\eeqn
Taking $\limsup$ with respect to $i \to \infty$ on both sides, one obtains
\begin{multline*}
\limsup_{i \to \infty} E(\A_i; \M_{\varphi_{\rho_i}(S_r)} ) > 2 \epsilon + \lim_{i \to \infty} E(\A_i; ( \varphi_{\rho_i}(S_r \setminus S_{r, \delta_0}) \times \Sigma)\\
 = 2\epsilon + E(u_\infty; S_r \setminus S_{r, \delta_0}) > \epsilon + E(u_\infty; S_r).
 \end{multline*}
Here we used the interior uniform convergence of the energy density (see Lemma \ref{lemma56}). This contradicts Proposition \ref{prop512}.
}
\end{proof}

\begin{cor}
\textcolor{black}{
For each $z \in {\bm H}$ and sufficiently small $r$, the limit 
\beq\label{eqn59}
\lim_{i \to \infty} E( \A_i; \M_{\varphi_{\rho_i}(B_r^+(z))} )
\eeq
exists. }
\end{cor}

\begin{proof}
\textcolor{black}{
If $z \notin W_\infty$, then choose $r>0$ such that $B_r^+(z)$ does not intersect $W_\infty$. Then by the interior uniform convergence of the energy density and Corollary \ref{cor513}, one has 
\beqn
\lim_{i \to \infty} E(\A_i; \M_{\varphi_{\rho_i}(B_r^+(z))} ) = E(u_\infty; B_r^+(z)).
\eeqn
On the other hand, consider $z = w \in W_\infty$. So far our construction guarantees that the limit \eqref{eqn59} exists for $r = 1/n$. Choose $n$ sufficiently large so that the regions $B_{1/n}^+(w)$ for all $w \in W_\infty$ do not mutually intersect. Then for all $r < 1/n$, Corollary \ref{cor513} and the interior uniform convergence of the energy density (Lemma \ref{lemma56}) imply that  
\beqn
\lim_{i \to \infty} E(\A_i; \M_{\varphi_{\rho_i}(B_{1/n}^+(w) \setminus B_r^+(w))}) = E(u_\infty; B_{1/n}^+(w) \setminus B_r^+(w)).
\eeqn
Therefore, for all $r< 1/n$ the limit \eqref{eqn59} exists. }
\end{proof}

\textcolor{black}{
Then we can define $m_\infty: {\bm H} \to [0, +\infty)$ by
\beqn
m_\infty(z) = \left\{ \begin{array}{cc} 0,\ z \notin W_\infty,\\
                                     \displaystyle \lim_{r \to 0} \lim_{i \to \infty} E(\A_i; \M_{\varphi_{\rho_i}(B_r^+(z))} ),\ z \in W_\infty.
\end{array}\right.
\eeqn
Therefore, for the current subsequence, the assertion (a) of Theorem \ref{thm51}; the assertion (c) follows immediately.}

\subsection{Proof of Proposition \ref{prop512}}

Now we prove Proposition \ref{prop512}. Using the rescaling $\varphi_{\rho_i}$, one can identify the restriction $\A_i$ with a piecewise smooth connection $\A_i'$ on $\M$. It is defined as follows. Suppose over ${\bm R} \times M_0$ we can write 
\beqn
\A_i|_{{\bm R}\times M_0} = d_s + \eta_i(s) ds + A_i(s).
\eeqn
Then 
\beqn
\A_i'|_{{\bm R}\times M_0} = d_s + \eta_i'(s) ds + A_i'(s) =  d_s + \rho_i \eta_i(\rho_i s) ds + A_i(\rho_i s).
\eeqn
Suppose over ${\bm H}\times \Sigma$ we can write
\beqn
\A_i|_{{\bm H}\times \Sigma} = d_{\bm H} + \phi_i(z) ds + \psi_i(z) dt + B_i(z).
\eeqn
Then 
\beqn
\A_i'|_{{\bm H} \times \Sigma} = d_{{\bm H}} + \phi_i'(z) ds + \psi_i'(z) dt + B_i'(z) = d_{{\bm H}} + \rho_i \phi_i(\rho_i z) ds + \rho_i \psi_i(\rho_i z) dt	+ B_i(\rho_i z).
\eeqn

The basic idea of proving Proposition \ref{prop512} is to compare $\A_i'$ with a connection $\A_\infty$ associated to the holomorphic curve $u_\infty$. As we did in proving Proposition \ref{prop221}, one can find a piecewise smooth connection $\A_\infty$ over $\M$ lifting the holomorphic curve $u_\infty$ (see Definition \ref{lift}). More precisely
\beqn
\A_\infty|_{{\bm R}\times M_0} = d_s + A_\infty (s),\ {\rm where}\ A_\infty (s) \in {\mc A}_{\rm flat}(P_0)
\eeqn
and 
\beqn
\A_\infty|_{{\bm H}\times \Sigma} = d_{\bm H} + \phi_\infty ds + \psi_\infty dt + B_\infty(z),\ {\rm where}\ B_\infty \in {\mc A}_{\rm flat}(Q)
\eeqn
such that over ${\bm H} \times \Sigma$ the triple $(B_\infty, \phi_\infty, \psi_\infty)$ satisfies the equation 
\beqn
\partial_s B_\infty - d_{B_\infty} \phi_\infty + *( \partial_t B_\infty - d_{B_\infty} \psi_\infty) = 0
\eeqn
and such that 
\begin{align*}
&\ [B_\infty(z)] = u_\infty(z) \in R_\Sigma,\ &\ A_\infty(s)|_{\partial M_0} = B_\infty(s).
\end{align*}
Since $u_\infty$ is a smooth holomorphic disk, one may assume that $B_\infty(z)$ depends smoothly on $z \in {\bm H}\cup \{\infty\}$. 

As the first step, we gauge transform $\A_i'$ over $\M_{S_0}$ such that it is close to $\A_\infty$ in a certain sense.

\begin{lemma}\label{lemma515}
After applying a piecewise smooth gauge transformation on $\M_{S_0}$ to each $\A_i'$, the following condition holds: for all $r>0$, 
\beq\label{eqn510}
\lim_{i \to \infty} \sup_{z \in S_r} \| B_i'(z) - B_\infty(z)\|_{L^4(\Sigma)} = 0
\eeq
and
\beq\label{eqn511}
\lim_{i \to \infty} \sup_{s \in S_r \cap \partial {\bm H}} \textcolor{black}{\| A_i'(s) - A_\infty(s)\|_{L^4(M_0)} }= 0.
\eeq
\end{lemma}

\begin{proof}
By our construction of the limit, for each $r>0$, there holds
\beq\label{eqn512}
\lim_{i \to \infty} \sup_{z \in S_r} \| F_{B_i'(z)}\|_{L^4 (\Sigma)} = 0.
\eeq
\textcolor{black}{Then there exists $i_r>0$ such that for all $i \geq i_r$, $B_i'(z)$ lies in the domain of the Narasimhan--Seshadri map ${\it NS}_4$ for all $z \in S_r$. Then one can choose a sequence $r_i>0$ which converges to zero such that for all sufficiently large $i$ and all $z \in S_{r_i}$, $B_i'(z)$ lies in the domain of ${\it NS}_4$. Denote $\uds{B}_i'(z) = {\it NS}_4 (B_i'(z)) \in {\mc A}_{\rm flat}(Q)$ whose gauge equivalence class is the value $u_i'(z) \in R_\Sigma$. Moreover, by \eqref{eqn512} and \eqref{eqn214} one can adjust $r_i \to 0$ such that
\beq\label{eqn513}
\lim_{i \to \infty}\sup_{z \in S_{r_i}} \| \uds{B}_i'(z) - B_i' (z)\|_{L^4(\Sigma)} = 0.
\eeq
On the other hand, by the $C^0$-convergence of $u_i'$ to $u_\infty$ over compact subsets of ${\bm H}\setminus W_\infty$, for a sufficiently small $\epsilon > 0$, one can modify the sequence $r_i \to 0$ such that for all sufficiently large $i$ one has
\beqn
\sup_{z\in S_{r_i}} d_{R_\Sigma}(u_i'(z), u_\infty(z)) \leq \epsilon.
\eeqn
Then for each $z \in S_{r_i}$ there exists a smooth gauge transformation $g_i(z) \in {\mc G}(Q)$ unique up to $\pm 1$ such that 
\beqn
g_i(z)^* \uds{B}_i'(z) - B_\infty(z) \in {\rm ker} d_{B_\infty(z)}^*.
\eeqn
As $g_i(z)$ is canonically defined from $B_i'(z)$ and $B_\infty(z)$, $g_i(z)$ gives a smooth gauge transformation ${\bm g}_i$ on $S_{r_i}\times \Sigma$. We extend ${\bm g}_i$ to a smooth gauge transformation on $S_0 \times \Sigma$ and then extend it to a piecewise smooth gauge transformation on $\M_{S_0}$. We claim that if we replace $\A_i'$ by ${\bm g}_i^* \A_i'$, then \eqref{eqn510} holds. Indeed, by the uniform convergence $u_i' \to u_\infty$ over compact subsets of ${\bm H}\setminus W_\infty$, for all $r>0$, one has 
\beqn
\lim_{i \to \infty} \sup_{z\in S_r} \| g_i(z)^* \uds{B}_i'(z) - B_\infty(z) \|_{L^4 (\Sigma)} = 0.
\eeqn
Therefore, combining with \eqref{eqn513} one obtains that for all small $r>0$,
\begin{multline*}
\lim_{i \to \infty} \sup_{z\in S_r} \| g_i(z)^* B_i'(z) - B_\infty(z) \|_{L^4 }\\
 \leq \lim_{i \to \infty} \sup_{z \in S_r} \| g_i(z)^* B_i'(z) - g_i(z)^* \uds B_i'(z)\|_{L^4 } + \lim_{i \to \infty} \sup_{z\in S_r} \| g_i(z)^* \uds B_i'(z) - B_\infty(z) \|_{L^4 } \\
 = \lim_{i \to \infty} \sup_{z \in S_r} \| B_i'(z) - \uds B_i'(z)\|_{L^4 } + \lim_{i \to \infty} \sup_{z\in S_r} \| g_i(z)^* \uds B_i'(z) - B_\infty(z) \|_{L^4 } = 0.
\end{multline*}
}
	
\textcolor{black}{Now we can assume that for the original sequence $\A_i'$, \eqref{eqn510} is true for all $r > 0$. We would like to construct another sequence of piecewise smooth gauge transformations ${\bm g}_i$ on $\M_{S_0}$ such that replacing $\A_i'$ by $({\bm g}_i)^* \A_i'$ makes \eqref{eqn511} true without altering \eqref{eqn510}. First, for each fixed small $r>0$, because energy does not blowup along $S_r \cap \partial {\bm H}$, one has
\beqn
\lim_{i \to \infty} \sup_{s \in S_r \cap \partial {\bm H}} \| F_{A_i'(s)} \|_{L^4(M_0)} = 0.
\eeqn
Then $[A_i'(s)]$ converges to $[A_\infty(s)]$ in ${\mc A}^{1,4}(P_0)/{\mc G}^{2,4}(P_0)$ uniformly for $s \in S_r \cap \partial {\bm H}$. Then there exists $i_r >0$ such that for $i \geq i_r$ and all $s \in S_r \cap \partial {\bm H}$, there exists a unique (up to $\pm 1$) smooth gauge transformation $g_i (s) \in {\mc G}(P_0)$ such that $g_i(s)^* A_i'(s)$ is in the Coulomb slice through $A_\infty(s)$ and there holds
\beq\label{eqn514}
\lim_{i \to \infty} \sup_{s \in S_r \cap \partial {\bm H}} \| g_i (s)^* A_i'(s) - A_\infty(s) \|_{L^4(M_0)} = 0.
\eeq
One can then find a sequence $r_i>0$ converging to $0$ such that $g_i(s)$ (which is specified by a gauge fixing condition) is defined when $s \in S_{r_i} \cap \partial {\bm H}$. Since $A_i'(s)$ and $A_\infty(s)$ are smooth in $s$, $g_i (s)$ depends smoothly on $s \in S_{r_i} \cap \partial {\bm H}$. Then we obtains a smooth gauge transformation ${\bm g}_i$ over $(S_{r_i} \cap \partial {\bm H}) \times M_0$.}

\textcolor{black}{Given any piecewise smooth extension of ${\bm g}_i$ to $\M_{S_0}$, if we apply this extension to $\A_i'$, then \eqref{eqn514} implies that \eqref{eqn511} is true. Now we would like to choose a suitable extension such that after such a gauge transformation, \eqref{eqn510} (which has been assumed to hold) remains true. Indeed, by restricting \eqref{eqn514} to the boundary one obtains that for all $r>0$
\beqn
\lim_{i \to \infty} \sup_{s \in S_r \cap \partial {\bm H}} \| (g_i(s)|_{\partial M_0})^* B_i'(s) - B_\infty(s)\|_{L^4(\Sigma)} = 0.
\eeqn
Then by Lemma \ref{lemma222}, for each $r$, for sufficiently large $i$ and $s \in S_r \cap \partial {\bm H}$, one can write $g_i(s)|_{\partial M_0} = e^{h_i(s)}$ with
\beq\label{eqn515}
\lim_{i \to \infty} \sup_{s \in S_r \cap \partial {\bm H}} \| h_i(s) \|_{W^{1,4}_{B_i'(s)}(\Sigma)}  = 0.
\eeq
We may then choose $r_i \to 0$ appropriately such that $h_i(s)$ is defined for $s \in S_{r_i}\cap \partial {\bm H}$ and such that $g_i(s)|_{\partial M_0} = e^{h_i(s)}$. Now choose a smooth cut-off function $\nu: [0, +\infty) \to [0,1]$ such that
\beqn
\nu(t) = \left\{ \begin{array}{ll} 1,\ &\ t \leq \frac{1}{4} \min \{ {\rm Im} w\ |\ w \in W_\infty \cap {\rm Int} {\bm H}\},  \\
                                   0,\ &\ t \geq \frac{1}{2} \min\{ {\rm Im} w\ |\ w \in W_\infty \cap {\rm Int} {\bm H}\}.
                                   \end{array} \right.
                                   \eeqn
For each $i$, choose a smooth cut-off function $\nu_i': (-\infty, +\infty) \to [0, 1]$ such that 
\beqn
\nu_i'(s) = \left\{ \begin{array}{ll} 1,\  &\ s \in \partial {\bm H} \cap  S_{r_i},\\
                                      0,\ &\ s \in \partial {\bm H} \setminus S_{\frac{r_i}{2} }.
\end{array} \right.
\eeqn                                   
Then define 
\beqn
g_i(s,t) = e^{\nu(t) \nu_i'(s) h_i(s)}
\eeqn
which is smooth over ${\bm H}\times \Sigma$. Define ${\bm g}_i$ which is equal to $g_i(s)$ on $\{s\}\times M_0$ for $s \in \partial {\bm H} \setminus W_\infty$ and which is equal to $g_i(s, t)$ on $(s,t) \times \Sigma$. Then by \eqref{eqn514} and \eqref{eqn515}, both \eqref{eqn510} and \eqref{eqn511} hold for all $r>0$ if we apply ${\bm g}_i$ to $\A_i'$.}
\end{proof}

Now we can prove the Proposition \ref{prop512}. Notice that by Lemma \ref{lemma21} and Proposition \ref{prop221}, one has
\begin{align*}
&\ E_{\rho_i}( \A_i'; \M_{S_r} ) = {\it CS}_{\bm P} ( \A_i'|_{\partial \M_{S_r}}),\ &\ E(u_\infty; S_r) = {\it CS}_{\bm P}( \A_\infty|_{\partial \M_{S_r}} ).
\end{align*}
Hence Proposition \ref{prop512} can be proved if one can control the relative Chern--Simons action
\beqn
{\it CS}^{\rm rel}( \A_i'|_{\partial \M_{S_r}}, \A_\infty |_{\partial \M_{S_r}}).
\eeqn
We will estimate the relative action over each boundary component. Notice that each $w \in W_\infty \cap {\rm Int} {\bm H}$ corresponds to a boundary component 
\beqn
S^1 \times \Sigma \cong \partial B_r(w) \times \Sigma  \subset \partial \M_{S_r}
\eeqn
and each $w\in W_\infty \cap \partial {\bm H}$ corresponds to a boundary component 
\beqn
N = N^- \cup N^{\rm neck} \cup N^+ \subset \partial \M_{S_r}.
\eeqn
Here $N^- \cong N^+ \cong M_0$. There is another boundary component corresponding to the big half circle $\partial B_{1/r}^+$ which is also diffeomorphic to $N$. \textcolor{black}{We only carry out the detailed argument for the case when $w \in W_\infty \cap \partial {\bm H}$.} The other cases can be treated similarly. Choose $w \in W_\infty \cap \partial {\bm H}$. For each small $r$, denote the restrictions of $\A_i'$ and $\A_\infty$ to the corresponding boundary component by $A_{i,r}$ and $A_{\infty, r}$ respectively, which can be viewed as piecewise smooth connections on $P_N \to N$. Then one can prove the following result.

\begin{lemma}\label{lemma516}
For any $\epsilon>0$, there exists $r_\epsilon>0$ such that for all $r \in (0, r_\epsilon)$, for sufficiently large $i$, one has  
\beq\label{eqn516}
|{\it CS}^{\rm rel}(A_{i, r}, A_{\infty, r}) | \leq \epsilon.
\eeq
\end{lemma}

Before proving Lemma \ref{lemma516} we first do some preparation. Over the neck region $N^{\rm neck} \cong [0, \pi]\times \Sigma$ where $\theta \in [0, \pi]$ is the angular coordinate, we can write a connection $A$ as $A = d_\theta + \eta d\theta + B(\theta)$. For each $p>1$, define 
\beqn
l_p(A):= \int_0^\pi \| \partial_\theta B(\theta) - d_{B(\theta)} \eta(\theta)\|_{L^p(\Sigma)} d \theta.
\eeqn
This is a certain integral of the norm of some component of $F_A$, hence is gauge invariant. Then one has the following estimate. 

\begin{lemma}\label{lemma517}
For any $\epsilon>0$, there exists $r_\epsilon>0$ such that for all $r \in (0, r_\epsilon)$, for sufficiently large $i$
\begin{enumerate}

\item there holds
\beq\label{eqn517}
l_4(A_{i,r}) \leq \epsilon;
\eeq

\item there exists $\tilde r_i \in (\frac{r}{2}, r)$ such that
\beq\label{eqn518}
\| F_{A_{i, \tilde r_i}} \|_{L^2(N)} \leq \epsilon.
\eeq
\end{enumerate}
\end{lemma}

\begin{proof}
Let $a$ be a small number to be determined. Then one can choose $r_a>0$ such that 
\beqn
\limsup_{i \to \infty} \textcolor{black}{ E(\A_i; \M_{\varphi_{\rho_i}( B_{2r_a}^+(w) )}) } \leq m_\infty(w) + a^2.
\eeqn
Then for all $r \in (0, r_a)$, when $i$ is sufficiently large, one has 
\beq\label{eqn519}
E_{\rho_i} (\A_i'; \M_{\frac{r}{2}, r}) \leq a^2.
\eeq
Over $N^{\rm neck}$ write $A_{i, r} = d_\theta + \eta_{i,r} d\theta + B_{i, r}(\theta)$. Then by Corollary \ref{cor36} and Proposition \ref{prop48}, for some $\delta>0$, one has
\beqn
\| \partial_\theta B_{i, r} - d_{B_{i,r}} \eta_{i,r} \|_{L^2(\Sigma)} \leq C a \min \{ \theta, \pi - \theta\}^{-1 + 2 \delta} .
\eeqn
By Lemma \ref{lemma35} one also has 
\beqn
\| \partial_\theta B_{i, r} - d_{B_{i,r}} \eta_{i,r} \|_{L^\infty (\Sigma)} \leq C \min \{ \theta, \pi - \theta\}^{-1}.
\eeqn
Then by interpolation, one has 
\beqn
\| \partial_\theta B_{i, r} - d_{B_{i,r}} \eta_{i,r} \|_{L^4 (\Sigma)} \leq \sqrt{a} \min C \{ \theta, \pi - \theta\}^{-1 + \delta}.
\eeqn
By integrating the above inequality over $\theta \in [0, \pi]$ and choosing $a$ sufficiently small, one can make $l_4(A_{i,r}) \leq \epsilon$. 

\textcolor{black}{To prove (b), we use the same method in the proof of Lemma \ref{lemma45}. Similar to \eqref{eqn48}, one has 
\beqn
E_{\rho_i}(\A_i'; \M_{\frac{r}{2}, r}) \geq \int_{\log r - \log 2}^{\log r} \| F_{A_{i, e^\tau}} \|_{L^2(N)}^2 d \tau.
\eeqn
Then taking $a$ sufficiently small in \eqref{eqn519}, the above estimate shows that there exists some $\tilde r_i \in (\frac{r}{2}, r)$ such that $\| F_{A_{i, \tilde r_i}} \|_{L^2(N)} \leq \epsilon$. }
\end{proof}

\begin{proof}[Proof of Lemma \ref{lemma516}]
For any $r \in (0, r_\epsilon)$ where $r_\epsilon$ is the one of Lemma \ref{lemma517} and $r \in (0, r_\epsilon)$, denote the restriction of $A_{\infty, r}$ to the neck region by $d_\theta + \eta_{\infty, r} d\theta + B_{\infty, r} (\theta)$ and the restriction of $A_{i, r}$ to the neck region by $d_\theta + \eta_{i, r} d\theta + B_{i, r}(\theta)$. Then for each $i$, there exists a unique smooth gauge transformation $g = g_{i,r}$ over $N^{\rm neck}$ of the form $g = e^h$ such that 
\beqn
\tilde A_{i,r}:= g^* A_{i,r} = d_\theta + \tilde B_{i,r} (\theta),\ g |_{\{0\} \times \Sigma} = 1.
\eeqn
Then it follows from Lemma \ref{lemma517} that when $i$ is sufficiently large, for any pair $\theta_1, \theta_2 \in [0, \pi]$, 
\beqn
\| \tilde B_{i,r} (\theta_1) - \tilde B_{i,r} (\theta_2) \|_{L^4 (\Sigma)} \leq \int_{\theta_1}^{\theta_2} \| \partial_\theta \tilde B_{i,r} (\theta) \|_{L^4(\Sigma)} d\theta \leq  l_4(A_{i,r}) \leq \epsilon.	
\eeqn
\textcolor{black}{Notice that when choosing the lift $\A_\infty$ of the limiting map $u_\infty$, one can require that 
\beqn
\lim_{r \to 0} \sup_{\theta_1, \theta_2 \in [0, \pi]} \| B_{\infty, r} (\theta_1) - B_{\infty, r} (\theta_2)\|_{L^4} = 0.
\eeqn}
Then by Lemma \ref{lemma515}, for each $r \in (0, r_\epsilon)$ and $i$ is sufficiently big, for all $\theta \in [0, \pi]$, 
\begin{multline}\label{eqn520}
\| \tilde B_{i,r}(\theta) - B_{i,r} (\theta) \|_{L^4} \\
\leq \| \tilde B_{i,r} (\theta) - \tilde B_{i,r}(0) \|_{L^4} + \| B_{i,r}(0) - B_{\infty,r} (0) \|_{L^4} + \| B_{\infty,r} (0) - B_{\infty,r} (\theta) \|_{L^4} \leq 3\epsilon. 
\end{multline}
Then by Lemma \ref{lemma222}, when $\epsilon$ is sufficiently small, for each $\theta$, 
\beqn
\| h(\theta,\cdot) \|_{W^{1,4 }_{B_{i,r}(\theta)}(\Sigma)} \leq C \epsilon.
\eeqn
\textcolor{black}{Then one can extend $h$ to a continuous and piecewise smooth section of ${\rm ad}P_N$ which is zero on $N^-$ such that 
\beq\label{eqn521}
\| h \|_{W^{1,4}_{A_{i,r}} (N^+)} \leq C \epsilon.
\eeq
By abuse of notation, denote by $\tilde A_{i, r}$ the piecewise smooth connection transformed from $A_{i,r}$ via the extended gauge transformation $g = e^h$. Clearly, $g$ belongs to the identity component ${\mc G}_0^{\rm p.s.} (P_N) \subset {\mc G}^{\rm p.s.}(P_N)$. Therefore $\tilde A$ has no $d\theta$ component over $N^{\rm neck}$. Moreover, \eqref{eqn511} and \eqref{eqn521} imply that 
\beq\label{eqn522}
\| \tilde A_{i, r} - A_{\infty, r} \|_{L^4(N^- \cup N^+)} \leq C \epsilon.
\eeq
One can similarly gauge transform $A_{\infty, r}$ via a transformation in ${\mc G}_0^{\rm p.s.}(P_N)$ to eliminate the $d\theta$ component of $A_{\infty, r}$ over $N^{\rm neck}$ and such that \eqref{eqn520} and \eqref{eqn522} still hold. Then both $\tilde A_{i,r}$ and $A_{\infty, r}$ have no $d\theta$ component over $N^{\rm neck}$. Therefore, for any fixed $r\in (0, r_\epsilon)$, for $i$ sufficiently large one has
\beqn
\| \tilde A_{i, r} - A_{\infty, r} \|_{L^4(N)} \leq C \epsilon.
\eeqn
Indeed, when $i$ is sufficiently large, one can guarantee that 
\beq\label{eqn523}
\sup_{\frac{r}{2} \leq \tilde r \leq r} \| \tilde A_{i, \tilde r} - A_{\infty, \tilde r} \|_{L^4(N)} \leq C \epsilon.
\eeq
Then by the expression of the relative Chern--Simons action (see \eqref{eqn21}) and its invariance under the identity component ${\mc G}_0^{\rm p.s.}(P_N)$, for any $\tilde r \in [\frac{r}{2}, r]$, one has 
\begin{multline*}
| {\it CS}^{\rm rel}(A_{i,  \tilde r}, A_{\infty, \tilde r}) |  = |  {\it CS}^{\rm rel} (\tilde A_{i, \tilde r}, A_{\infty, \tilde r}) | \\
\leq C \int_N \big(  |F_{A_{i, \tilde r}} | |\tilde A_{i, \tilde r} - A_{\infty, \tilde r} | + |F_{A_{\infty, \tilde r}}| | \tilde A_{i, \tilde r} - A_{\infty, \tilde r} | + |\tilde A_{i, \tilde r} - A_{\infty, \tilde r} |^3 \big) d{\rm vol}_N\\
\leq C\Big( \big( \| F_{A_{i, \tilde r}} \|_{L^2(N)} + \| F_{A_{\infty, \tilde r}} \|_{L^2(N)} \big) \| \tilde A_{i, \tilde r} - A_{\infty, \tilde r} \|_{L^2(N)} + \| \tilde A_{i, \tilde r} - A_{\infty, \tilde r}  \|_{L^3(N)}^3\Big).
\end{multline*}
If we set $\tilde r = \tilde r_i$ where $\tilde r_i$ is the one from Lemma \ref{lemma517}, then one has 
\beqn
| {\it CS}^{\rm rel}(A_{i,  \tilde r_i}, A_{\infty, \tilde r_i}) | \leq C \epsilon.
\eeqn
Finally, by the property of the relative Chern--Simons action, one has 
\begin{multline*}
|{\it CS}^{\rm rel}(A_{i, r}, A_{\infty, r})| \leq | {\it CS}^{\rm rel}(A_{i,r}, A_{i, \tilde r_i})| + |{\it CS}^{\rm rel}(A_{i, \tilde r_i}, A_{\infty, \tilde r_i})| + |{\it CS}^{\rm rel}(A_{\infty, \tilde r_i}, A_{\infty, r})|\\
\leq E_{\rho_i} ( \A_i'; \M_{\tilde r_i, r}) + C\epsilon +  E(u_\infty; {\it Ann}^+(\tilde r_i, r)).
\end{multline*}
The right hand side can be bounded by $2C \epsilon$ if we choose $r_\epsilon$ sufficiently small and $i$ sufficiently large.}
\end{proof}

It is similar to prove similar results for other boundary components of $\M_{S_r}$. We omit the details. Hence Proposition \ref{prop512} is proved.

\section{Stable scaled instantons}\label{section6}

In this section we define the Gromov--Uhlenbeck convergence for ASD instantons. In fact we define a notion called {\bf stable scaled instantons}, which is a combination of stable maps in symplectic geometry and ``instantons with Dirac measures'' in gauge theory. The combinatorial model of trees is also used in the study of the vortex equation (see for example \cite{Wang_Xu} \cite{Woodward_Xu}) and holomorphic quilts (see \cite{Bottman_Wehrheim_2018}).

In this and the next sections, we will emphasize on the discussion of instantons over ${\bm R} \times M$. The situation for instantons over ${\bm C}\times \Sigma$ can be dealt with similarly and most of the details will be left to the reader.

\subsection{Trees}

Let us fix a few notations. A tree, usually denoted by $\Gamma$, consists of a set of vertices $V_\Gamma$ and a set of edges $E_\Gamma$. %, and a set of leaves $L_\Gamma$.\footnote{In some convention a leaf is regarded as a ``semi-infinite edge'' which is only attached to a single vertex, while other edges are regarded as ``finite'' edges.} 
One can associate to each tree a 1-complex whose $0$-cells are vertices and whose $1$-cells are edges. %, while leaves are regarded as combinatorial decorations. 
A {\bf rooted tree} is a tree $\Gamma$ together with a distinguished vertex $v_\infty \in V_\Gamma$ called the {\bf root}. A {\bf rooted subtree} of a rooted tree $(\Gamma, v_\infty)$ is a subtree which contains $v_\infty$. %A {\bf ribbon tree} is a tree $\Gamma$ together with an isotopy class of embeddings of $\Gamma$ into the complex plane. 
A {\bf based tree} is a tree $\Gamma$ with a rooted subtree $\uds \Gamma$ 
%with $\uds\Gamma$ equipped with the structure of a ribbon tree. 
called the {\bf base} of $\Gamma$. A based tree can be used to model stable holomorphic disks such that vertices in the base correspond to disk components and vertices not in the base correspond to sphere components. 

We only consider rooted trees and often skip the term ``rooted.'' Notice that the root $v_\infty \in V_\Gamma$ induces a natural partial order $\leq$ among all vertices: $v \leq v'$ if $v$ is closer to the root $v_\infty$. We write $v' \succ v$ if $v \leq v'$ and $v$, $v'$ are adjacent; in this case we denote the corresponding edge by $e = e_{v'\succ v} \in E_\Gamma$. An edge $e = e_{v' \succ v}$ is called a {\bf boundary edge} if it is an edge of $\uds\Gamma$, i.e., if $v', v \in V_{\uds \Gamma}$. Edges which are not boundary edges are called {\bf interior edges} of $\Gamma$.

\begin{defn}
Consider the set $\{1, \infty\}$ ordered as $1 \leq \infty$. A {\bf scaling} on a based tree $\Gamma$ is a map ${\mf s}: V_\Gamma \to \{1, \infty\}$ satisfying the following condition. 
\begin{itemize}
\item Denote $V_\Gamma^1 = {\mf s}^{-1}(1)$ and $V_\Gamma^\infty = {\mf s}^{-1}(\infty)$. Then $V_\Gamma^\infty$ forms a (possibly empty) rooted subtree of $\Gamma$ and vertices in $V_\Gamma^1$ are all disconnected from each other. 
\end{itemize}
A based tree $\Gamma$ with a scale ${\mf s}$ is called a {\bf scaled tree}. 
\end{defn}

Each vertex of a scaled tree is supposed to support an instanton or a holomorphic map over a certain domain, which we specify as follows. Consider a scaled tree $\Gamma = (\Gamma, {\mf s})$ with possibly empty base $\uds \Gamma$. For each $v \in V_{\uds\Gamma}$, define $\M_v = {\bm R}\times M$, $S_v = {\bm H}$, and $\ov{S_v} = {\bm H} \cup \{\infty\} \cong {\bm D}^2$; for each $v  \in V_\Gamma \setminus V_{\uds \Gamma}$, define $\M_v = {\bm C} \times \Sigma$, $S_v = {\bm C}$, and $\ov{S_v} = {\bm C} \cup \{\infty\} \cong {\bm S}^2$. For each $v$, there is an $SO(3)$-bundle ${\bm P}_v \to \M_v$ specified previously. 

\begin{defn}\label{defn62}
A {\bf stable scaled instanton} modelled on a scaled tree $(\Gamma, {\mf s})$ consists of a collection 
\beqn
{\mc C} = \Big( \big\{ ( \a_v, {\bm m}_v ) \ |\ v \in V_\Gamma^1 \big\},\ \big\{ w_e \ |\ e \in E_\Gamma \big\},\ \big\{ {\bm u}_v = (u_v, W_v, \gamma_v) \ |\ v \in V_\Gamma^\infty \big\} \Big)
\eeqn
where the symbols denote the following objects.
\begin{itemize}

\item For each $v \in V_\Gamma^1$, $\a_v$ is a gauge equivalence class of ASD instantons on $\P_v \to \M_v$ and ${\bm m}_v$ is a positive measure on $\M_v$ with finite support.

\item For each edge $e  = e_{v' \succ v} \in E_\Gamma$, $w_e$ is a point of $S_v$. 

\item For each $v \in V_\Gamma^\infty$, $W_v$ is the subset of $S_v$ defined by
\beq\label{eqn61}
W_v:= \{ w_e\ |\ e = e_{v'\succ v} \in E_\Gamma \}
\eeq
and ${\bm u}_v$ is a holomorphic map from $S_v$ to $R_\Sigma$ with boundary in $\iota(L_M)$ (see Definition \ref{defn28}). We call $w_e \in W_v$ a {\bf boundary node} resp. {\bf interior node} if $e$ is a boundary resp. interior edge. 
\end{itemize}
These objects also satisfy the following conditions.
\begin{enumerate}
\item If $e = e_{v' \succ v}$ is a boundary resp. interior edge, then $w_e \in \partial S_v$ resp.  $w_e \in {\rm Int} S_v$.

\item For every $v \in V_\Gamma^\infty$, the collection of points defining $W_v$ in \eqref{eqn61} are distinct. 

\item The measure ${\bm m}_v$ takes value in $4\pi^2 {\bf Z}$.

\item {\bf (Matching Condition)} For each interior edge $e = e_{v'\succ v}$ the evaluation at infinity of $\a_{v'}$ or ${\bm u}_{v'}$, which is a point of $R_\Sigma$, is equal to the evaluation of ${\bm u}_v$ at $w_e$. For each boundary edge $e = e_{v'\succ v}$ the evaluation at infinity of $\a_{v'}$ or ${\bm u}_{v'}$, which is a point in $L_N \cong \Delta_{L_M} \cup R_{L_M}$ is equal to the transpose of ${\rm ev}_{w_e}({\bm u}_v)$ (see Definition \ref{defn28} and Definition \ref{evaluation} for the definitions of different evaluations.)

\item {\bf (Stability Condition)} If $\a_v$ has zero energy, then ${\bm m}_v \neq 0$; if $v \in V_{\uds\Gamma}^\infty$ and $u$ is a constant map, then the number of boundary nodal points plus twice of the interior nodal points attached to $S_v$ is at least two; if $v \in V_\Gamma^\infty \setminus V_{\uds\Gamma}^\infty$ and $u_v$ is a constant map, then the number of nodal points on $S_v$ is at least two.
\end{enumerate}
\end{defn}

A typical configuration of a stable scaled instanton has been shown in Figure \ref{figure4}.

One can see that on the combinatorial level a stable scaled instanton is very similar to a stable affine vortex over $\C$ (when $\Gamma$ has an empty base) or over $\H$ (when $\Gamma$ has a nonempty base). These two constructions appeared in \cite{Ziltener_book} and \cite{Wang_Xu} respectively. We can define a notion of equivalence among stable scaled instantons by incorporating the translation symmetry of ASD instantons and the conformal invariance of holomorphic maps. The details are left to the reader. It is then easy to check that the automorphism group of a stable scaled instanton is finite.

\subsection{Sequential convergence}

Now we define the notion of convergence of a sequence of ASD instantons over $\M = {\bm R} \times M$ towards a stable scaled instanton. Let 
\beqn
{\mc C} = \Big( \big\{ ( \a_v, {\bm m}_v ) \ |\ v \in V_\Gamma^1 \big\},\ \big\{ w_e \ |\ e \in E_\Gamma \big\},\ \big\{ {\bm u}_v  \ |\ v  \in V_\Gamma^\infty \big\} \Big)
\eeqn
be a stable scaled instanton modelled on a based scaled tree $(\Gamma, {\mf s})$. For each $v \in V_\Gamma^\infty$, we have the set of nodes $W_v \subset S_v$ defined by \eqref{eqn61}. Then we can define a measure $m_v$ on $S_v$ supported in $W_v$ as follows. For each $w_e \in W_v$ with $e = e_{v'\succ v}$, the mass $m_v$ at $w_e$ is the sum of the energy of all components in ${\mc C}$ labelled by $v'' \in V_\Gamma$ with $v'' \geq v'$. Denote 
\beqn
\tilde {\bm u}_v = ({\bm u}_v, m_v)
\eeqn
which is a holomorphic curve with mass (see discussion after Definition \ref{defn28}).

\begin{defn}[Convergence towards stable scaled instantons] \label{defn63}
Let $\a_i = [\A_i]$ be a sequence of gauge equivalence classes of ASD instantons on $\P \to \M$. Let $(\Gamma, {\mf s})$ be a based scaled tree. Let 
\beqn
{\mc C} = \Big( \big\{ ( \a_v, {\bm m}_v ) \ |\ v \in V_\Gamma^1 \big\},\ \big\{ w_e\  |\ e \in E_\Gamma \big\},\ \big\{ {\bm u}_v  \ |\ v  \in V_\Gamma^\infty \big\}\Big)
\eeqn
be a stable scaled instanton modelled on $(\Gamma, {\mf s})$. We say that $\a_i$ converges (modulo gauge transformation) to ${\mc C}$ if the following conditions are satisfied. 
\begin{enumerate}

\item For each $v \in V_{\uds \Gamma}^1$, there exist a sequence of real translations $\varphi_{i, v}(z) = z + Z_{i, v }$ such that $\varphi_{i, v }^* \a_i$ converges in the Uhlenbeck sense to $( \a_v, {\bm m}_v)$ over the manifold $\M_v = {\bm R} \times M$ (see Definition \ref{defn22}).

\item For each $v \in V_\Gamma^1 \setminus V_{\uds \Gamma}^1$, there exist a sequence of complex translations $\varphi_{i, v}(z) =  z + Z_{i, v}$ satisfying the following properties.
\begin{enumerate}

\item ${\bf Im} Z_{i, v} \to +\infty$; 

\item Viewing $\varphi_{i, v }$ as a diffeomorphism from ${\bm C} \times \Sigma$ to itself, for all $R>0$, $\varphi_{i, v}^* (\a_i|_{\varphi_{i, v}(B_R) \times \Sigma})$ converges to $( \a_v |_{B_R \times \Sigma}, {\bm m}_v |_{B_R \times \Sigma})$ in the Uhlenbeck sense.
\end{enumerate}

\item For each $v \in V_{\uds\Gamma}^\infty$, there exists a sequence of real affine transformations \textcolor{black}{$\varphi_{i, v}(z) = \rho_{i, v} z + Z_{i, v}$} satisfying the following properties.
\begin{enumerate}

\item $\rho_{i, v} \in {\bm R}$, $Z_{i, v } \in {\bm R}$, $\rho_{i, v} \to +\infty$;

\item Denoting the translation $z \mapsto z + Z_{i, v}$ by $\psi_{i, v }$ and viewing it as a diffeomorphism from $\M$ to itself, $\psi_{i, v}^* \a_i$ converges to $\tilde {\bm u}_v$ along with $\{\rho_{i, v}\}$ in the sense of Definition \ref{defn52}.
\end{enumerate}

\item For each $v \in V_\Gamma^\infty \setminus V_{\uds\Gamma}^\infty$, there exist a sequence of complex affine transformations \textcolor{black}{$\varphi_{i, v}(z) = \rho_{i, v}  z + Z_{i, v} $} satisfying the following properties. 
\begin{enumerate}

\item $\rho_{i, v} \in {\bm R}$, $Z_{i, v} \in \C$, $\rho_{i, v} \to +\infty$, and ${\bf Im} Z_{i, v} \to +\infty$;

\item Denoting the translation $z \mapsto z + Z_{i, v}$ to be $\psi_{i, v}$ and viewing it as a diffeomorphism from ${\bm C} \times \Sigma$ to itself, then for all sufficiently big $R>0$, the sequence of pullbacks $\psi_{i, v}^* ( \a_i|_{\varphi_{i, v}(B_R)\times \Sigma})$, which are instantons over $B_{\rho_{i, v} R} \times \Sigma$, converge to $\tilde {\bm u}_v |_{B_R}$ along with $\{\rho_{i, v}\}$ in the sense of Definition \ref{defn34}.
\end{enumerate}

\item The reparametrizations $\varphi_{i, v}$ mentioned above can all be viewed as M\"obius transformations on the complex plane. One can also view all $S_v$ (being either ${\bm H}$ or ${\bm C}$) as subsets of ${\bm C}$. For each edge $e = e_{v'\succ v} \in E_\Gamma$, the map $\varphi_{i, v}^{-1} \circ \varphi_{i, v'}$ converges uniformly with all derivatives on compact sets to the constant $w_e\in S_v \subset {\bm C}$. 

\item There is no energy lost, i.e., 
\beqn
\lim_{i \to \infty} E ( \A_i ) = \sum_{v \in V_\Gamma^\infty} E (u_v) + \sum_{v  \in V_\Gamma^1} \Big( E ( \a_v ) +  |{\bm m}_v | \Big).
\eeqn
\end{enumerate}
\end{defn}

Now we state the precise form of the main theorem (Theorem \ref{thm14}) of this paper. 

\begin{thm}\label{thm64}
Given a sequence of ASD instantons on ${\bm R} \times M$ with uniformly bounded energy, there exists a subsequence which converges to a stable scaled instanton in the sense of Definition \ref{defn63}.
\end{thm}

The rest of this paper provides a proof of Theorem \ref{thm64}.

\begin{rem}
It is easy to extend our definitions and our compactness theorem to the case of instantons over $\C \times \Sigma$. For example, when $\Gamma$ is a scaled tree with empty base, the object defined in Definition \ref{defn62} is a possible limiting configuration of a sequence of instantons over $\C \times \Sigma$. One can then prove the corresponding compactness theorem (Theorem \ref{thm15}) for such instantons in a routine way. 
\end{rem}

\section{Proof of the compactness theorem}\label{section7}

Now we start to prove the compactness theorem (Theorem \ref{thm64}). The basic strategy is similar to the proof of Gromov compactness for pseudoholomorphic \textcolor{black}{spheres} used in \cite{McDuff_Salamon_2004} \textcolor{black}{(see also \cite{Frauenfelder_disk} for the case of pseudoholomorphic disks and \cite{Wang_Xu} for adiabatic limit of disk vortices). The construction of the limiting object is inductive and each inductive step is based on the ``soft rescaling'' argument in which one can verify the stability condition of the limiting object (see Subsection \ref{subsection72}). The matching condition is proved by using the annulus lemma (see Subsection \ref{subsection73}).} The inductive construction will stop after finitely many steps because of various energy quantization phenomena. We reset the value of $\hbar>0$ such that it is smaller than the following positive numbers.  
\begin{enumerate}

\item The $\epsilon$ of the annulus lemma (Proposition \ref{annulus} and Corollary \ref{cor47}).

\item The threshold of energy blowup given by Lemma \ref{lemma54}.

\item The minimal energy of nonconstant holomorphic spheres in $R_\Sigma$.

\item The minimal energy of nonconstant holomorphic disks in $R_\Sigma$ with boundary in $\iota(L_M)$ having at most two switching points.

\item The minimal energy of nonconstant ASD instanton over ${\bm C}\times \Sigma$ (the existence of the minimal energy is proved in \cite[Proof of Theorem 9.1]{Dostoglou_Salamon} and \cite{Wehrheim_2006}).

\item The minimal energy of a nontrivial ASD instanton over $\M$ (cf. Theorem \ref{quantization}). 

\end{enumerate}

\textcolor{black}{We give a brief account of the general induction scheme which is standard in symplectic geometry but not necessarily in gauge theory. By using Theorem \ref{thm51} for a suitable sequence $\rho_i$, we construct the root component of the limiting object to which a subsequence converges modulo bubbling. Then depending on the rate of energy blowup and the distance from the boundary of the holomorphic curve, one can ``zoom in'' with an appropriate rate at the bubbling point to extract another component with possibly additional bubbling points. In each step, the component we will construct could be either a (combinatorially) stable component with possibly zero energy or an unstable component with positive energy. In the former case, the total number of bubbling points strictly increases and each bubbling point has an amount of escaped energy at least $\hbar$; in the latter case, the energy of the unstable component also has a lower bound $\hbar$. Because of the uniform energy bound of the sequence, the induction necessarily stops after finitely many steps.}

% component, such as a sphere with three or more nodal points, which will increase the total number of blowup points in the constructed configuration, or construct a (combinatorially) unstable component with positive energy, such as a holomorphic disk with boundary in $\iota(L_M)$ having two switching points and no interior node. The {\it a priori} bound of the total energy and the various energy quantization phenomena listed above guarantee the termination of the induction after finitely many steps.

Now we start the construction. Fix a sequence of ASD instantons $\A_i\in {\mc A}(\P)$ satisfying 
\beqn
\sup_i E (\A_i) < +\infty.
\eeqn 
By Theorem \ref{quantization} we may assume that 
\beq\label{eqn71}
E  ( \A_i) \geq \hbar\ \ ( \forall i\geq 1).
\eeq
Otherwise there is a subsequence of trivial instantons.

It is also convenient to fix the ambiguity caused by the translation invariance. Recall the notation $B_r^+(z)$: for any $z \in {\bm H}$ and $r>0$, $B_r^+(z)$ is the intersection of the radius $r$ open ball $B_r(z) \subset {\bm C}$ with the upper half plane ${\bm H}$. For each $i\geq 1$ and $Z \in {\bm R}$, let $R = R(Z)$ be the minimal real number satisfying
\beqn
E \big( \A_i; \M \setminus \M_{B_R^+(Z)} \big) = \frac{\hbar }{2}.
\eeqn
By the decay of energy, when $Z$ approaches to $\pm \infty$, the $R (Z)$ satisfying the above condition approaches to $\infty$. It is also easy to see that $R(Z)$ depends on $Z$ continuously. Then one can choose for each $i$ a number $Z_i$ (which might not be unique) such that $R_i = R(Z_i)$ is smallest possible. Then by using a translation of $\M$ in the ${\bm R}$-direction, we may assume $Z_i = 0$ for all $i$. Hence one has
\beq\label{eqn72}
E \big( \A_i; \M \setminus \M_{R_i} \big) = \frac{\hbar }{2}.
\eeq

By taking a subsequence, we may assume that either $R_i$ stays bounded or diverges to infinity. We first prove that there is no energy blowup in the area further away from $\M_{R_i}$. 

\begin{prop}\label{prop71}
For any $\rho_i$ with $\rho_i \gg R_i$, one has 
\beq\label{eqn73}
\lim_{i \to \infty} E \big( \A_i; \M \setminus \M_{\rho_i} \big) = 0.
\eeq
\end{prop}

\begin{proof} Follows from the annulus lemma (Proposition \ref{annulus}).
\end{proof}

\subsection{Constructing the root component}

Now we are going to construct the root component of the limiting object. A simple situation is when $R_i$ is bounded. Then by Theorem \ref{thm23}, a subsequence of $\A_i$ converges to an ASD instanton $\A_\infty$ with a bubbling measure ${\bm m}_\infty$ in the Uhlenbeck sense. Then $( [\A_\infty], {\bm m}_\infty )$ is the limiting object. Indeed, the scaled tree underlying this limiting object has only one vertex, and Proposition \ref{prop71} is equivalent to the no-energy-lost condition of Definition \ref{defn63}.

From now on we assume that $R_i$ is unbounded and we take a subsequence such that $R_i$ diverges to infinity. Then by applying Theorem \ref{thm51} for $R_i$ and $\A_i$, one obtains a subsequence (still indexed by $i$) and a holomorphic curve with mass $\tilde {\bm u}_\infty = ({\bm u}_\infty, m_\infty)$ from ${\bm H}$ to $R_\Sigma$ such that $\A_i$ converges to $\tilde {\bm u}_\infty$ along with $\{R_i\}$ in the sense of Definition \ref{defn52}. Proposition \ref{prop71} and \eqref{eqn52} imply that 
\beq\label{eqn74}
\lim_{i \to \infty} E (\A_i) = E (u_\infty ) +  | m_\infty |.
\eeq

We would like to show that this limiting component is stable. 

\begin{lemma}
If $u_\infty$ is a constant map, the support of $m_\infty$ contains either an interior point of $\H$, or contains at least two points in the boundary of $\H$.
\end{lemma}

\begin{proof}
The assumption \eqref{eqn71} and the equality \eqref{eqn74} imply that if $u_\infty$ is constant, then $W_\infty = {\rm Supp} m_\infty \neq \emptyset$. Suppose $W_\infty \subset \partial \H$ and contains only one element $w_0$. Then all energy is concentrated at $w_0$, which means that 
\beqn
\lim_{r \to 0} \lim_{i \to \infty} E \Big( \A_i; \M \setminus \M_{\varphi_{R_i}(B_r^+ (w_0))} \Big) = 0.
\eeqn
Since the total energy is at least $\hbar$, there exists a sequence of $r_i \to 0$ such that for all $i$ sufficiently large,
\beqn
E  \Big( \A_i; \M \setminus \M_{\varphi_{R_i} (B_{r_i}^+ (w_0))} \Big) = E \Big( \A_i; \M \setminus \M_{B_{r_i R_i}^+( R_i w_0)} \Big) = \frac{\hbar }{2}.
\eeqn
Since $r_i R_i \ll R_i$, this contradicts the choice of $R_i$ (see \eqref{eqn72}). 
\end{proof}

\subsection{Soft rescaling}\label{subsection72}

What we have done is constructing the root component of the limiting stable scaled instanton. Next we inductively construct all other components using the ``soft-rescaling'' method (cf.\cite[Section 4.7]{McDuff_Salamon_2004} \cite{Frauenfelder_disk} \cite{Ziltener_book} \cite{Wang_Xu}).

Let $v_\infty$ denote the root component and $m_\infty$ the bubble measure. Suppose
\beqn
W_\infty = {\rm Supp} m_\infty = \{ w_k\ |\ k =1, \ldots, k_\infty \}.
\eeqn
For each $w_k \in {\rm Supp} m_\infty$, one has that 
\beq\label{eqn75}
\lim_{r \to 0} \lim_{i \to \infty} E \Big( \A_i;  \textcolor{black}{ \M_{\varphi_{R_i}( B_r^+ (w_k))} } \Big) = m_\infty (w_k) \geq \hbar.
\eeq
Then choose $r_0 >0$ such that $B_{2r_0}^+ (w_k) \cap W_\infty = \{w_k \}$ and for sufficiently big $i$ there holds
\beqn
m_\infty (w_k ) - \frac{\hbar}{4} \leq E \Big( \A_i;  \M_{\varphi_{R_i} \textcolor{black}{(B_{r_0}^+ (w_k )) }} \Big) = E \Big( \A_i; \M_{B_{R_i r_0}^+( R_i w_k)} \Big) \leq m_\infty (w_k) + \frac{ \hbar }{4}.
\eeqn
Then for each $w \in B_{r_0}^+ (w_k)$ and $i$, there exists $r_i(w) > 0$ such that
\beqn
E \Big( \A_i; \M_{\varphi_{R_i} (B_{r_i(w)}^+ (w)) }  \Big) = m_\infty (w_k) - \frac{\hbar}{2}.
\eeqn
Notice that the minimal $r_i(w)$ depends continuously on $w$. Therefore, one chooses $w_{i, k} \in B_{r_0}^+ (w_k)$ such that 
\beq\label{eqn76}
r_{i,k}:= r_i(w_{i,k}) = \min_{w \in B_{r_0}^+ (w_k)} r_i(w).
\eeq
Since energy blows up at $w_k$, $w_{i, k}$ converges to $w_k$ and $r_i(w_{i, k})$ converges to zero. The point $w_{i,k}$ can be viewed as the point near $w_k$ where energy blows up in the fastest rate (an analogue of the local maximum of the energy density). 

The following argument bifurcates in different cases. Suppose 
\beqn
w_{i,k} = s_{i,k} + {\bm i} t_{i,k}.
\eeqn
By choosing a suitable subsequence, we are in one of the following situations.

\vspace{0.1cm}

\noindent {\bf Case I.} Suppose we are in the situation where 
\begin{align}\label{eqn77}
&\ w_k \in {\rm Int} \H,\ &\ \lim_{i \to \infty} r_{i,k} R_i < \infty.
\end{align}
Then choose a sequence $\rho_i \to \infty$ but grows slower than $R_i$. Then when $i$ is large, $B_{\rho_i}(R_i w_{i,k})  \subset {\bm C}$ is contained in the upper half plane. Define the sequence of translations on ${\bm C}$ by 
\beqn
\varphi_{i,k}(z) = z + R_i w_{i,k}.
\eeqn
Via $\varphi_{i,k}$, the sequence of restrictions of $\A_i$ to $B_{\rho_i} ( R_i w_{i,k}) \times \Sigma$ can be identified with a sequence $\A_{i,k}'$ of solutions to the ASD equation over $B_{\rho_i} \times \Sigma$. The total energy of $\A_{i, k}'$ is uniformly bounded, hence a subsequence (still indexed by $i$) converges in the Uhlenbeck sense to a limiting ASD instanton $\A_k$ over $\M_k:= \C \times \Sigma$ with a bubbling measure ${\bm m}_k$ on $\M_k   $.

We prove that no energy is lost. 
\begin{lemma}\label{lemma73}
There holds 
\beqn
m_\infty (w_k) = E ( \A_k ) + \int_{\M_k} {\bm m}_k.
\eeqn
\end{lemma}

\begin{proof}
Choose $\epsilon>0$. By assumption, there exists $r_\epsilon>0$ such that for $i$ sufficiently large, there holds
\beqn
E \Big( \A_{i, k}'; B_{R_i r_\epsilon}  \times \Sigma \Big) \leq m_\infty (w_k)  + \epsilon.
\eeqn
Then the Uhlenbeck convergence $\A_{i, k}' \to ( \A_k, {\bm m}_k)$ implies that
\beqn
E ( \A_k ) + |{\bm m}_k| = \lim_{R \to \infty} \lim_{i \to \infty} E \Big( \A_{i, k}'; B_R  \times \Sigma \Big) \leq \lim_{i \to \infty} E \Big( \A_{i,k}; B_{R_i r_\epsilon}  \times \Sigma \Big) \leq m_\infty (w_k) + \epsilon.
\eeqn
On the other hand, we claim that when $\epsilon$ is small enough, there exists $\tau_\epsilon>0$ such that for $i$ sufficiently large, there holds
\beq\label{eqn78}
E \Big( \A_{i,k}'; B_{\tau_\epsilon}  \times \Sigma \Big) \geq m_\infty (w_k) - 2 \epsilon.
\eeq
If this is not true, then there exist a sequence $\tau_i \to \infty$ such that 
\beqn
E \Big( \A_{i, k}'; B_{\tau_i}  \times \Sigma \Big) = m_\infty (w_k) - 2 \epsilon.
\eeqn
There are also a sequence $\tau_i' > \tau_i$ such that
\beq\label{eqn79}
E \Big(  \A_{i,k}'; B_{\tau_i'}  \times \Sigma \Big) = m_\infty (w_k) - \epsilon.
\eeq
Hence
\beq\label{eqn710}
E \Big( \A_{i,k}'; (B_{\tau_i'}  \setminus B_{\tau_i} ) \times \Sigma \Big) = \epsilon.
\eeq
On the other hand, choose $R_0>1$ which is bigger than the limit of $r_{i,k} R_i$. Then one has 
\begin{multline}\label{eqn711}
E \Big( \A_{i,k}'; ( B_{R_i r_\epsilon} \setminus B_{R_0}  ) \times \Sigma \Big)  = E \Big( \A_{i,k}'; B_{R_i r_\epsilon} \times \Sigma \Big) - E \Big( \A_{i,k}'; B_{R_0} \times \Sigma \Big)\\
 \leq E \Big( \A_{i,k}'; B_{R_i r_\epsilon} \times \Sigma \Big) - E \Big( \A_{i,k}'; B_{r_{i, k} R_i} \times \Sigma \Big)   \\
 \leq m_\infty( w_k) + \epsilon - (m_\infty(w_k) - \frac{ \hbar }{2}) = \frac{ \hbar}{2} + \epsilon.
\end{multline}

We would like to show that \eqref{eqn710} and \eqref{eqn711} contradict Corollary \ref{cor47}. To see this, notice that the conformal radius of $B_{R_i r_\epsilon}  \setminus B_{R_0} $ diverges to infinity. Since $\tau_i \to +\infty$, one has $\tau_i/ R_0 \to \infty$. Moreover, we claim that $\tau_i'/ R_i r_\epsilon \to 0$. Indeed, if this is not the case, then there is a subsequence (still indexed by $i$) such that $\tau_i' / R_i$ converges to a positive number $3r'>0$. Then \eqref{eqn79} implies that 
\begin{multline*}
\lim_{i \to \infty} E \Big( \A_i; \varphi_{R_i}( B_{r'} (w_k)) \times \Sigma \Big) \leq \lim_{i \to \infty} E \Big( \A_i; \varphi_{R_i} ( B_{2r'} (w_{i,k}) ) \times \Sigma \Big) \\
 \leq \lim_{i \to \infty} E \Big( \A_{i,k}'; B_{\tau_i'} \times \Sigma \Big)  = m_\infty(w_k) - \epsilon.
\end{multline*}
This contradicts \eqref{eqn75}. Hence the claim that $\tau_i' / R_i r_\epsilon \to 0$ is true. Then by Corollary \ref{cor47}, the energy on the smaller annular region $(B_{\tau_i'}\setminus B_{\tau_i} ) \times \Sigma$ should be exponentially smaller than the energy on the larger annular region $(B_{R_i r_\epsilon} \setminus B_{R_0}) \times \Sigma$, contradicting \eqref{eqn710}. 

Therefore, when $\epsilon>0$ is small enough, the claim of the existence of $\tau_\epsilon$ satisfying \eqref{eqn78} is true. Then one has 
\beqn
E ( \A_k) + |{\bm m}_k| \geq \lim_{i \to \infty} E \Big( \A_{i,k}'; B_{\tau_\epsilon}  \times \Sigma \Big) \geq m_v(w_k)- 2\epsilon.
\eeqn
Since $\epsilon$ is an arbitrary small number, this lemma is proved. 
\end{proof}

The following lemma follows directly from Lemma \ref{lemma73}.

\begin{lemma}\label{lemma74}
If $\A_k$ is a trivial instanton, then ${\bm m}_k \neq 0$. 
\end{lemma}

\vspace{0.1cm}

\noindent {\bf Case II.} Suppose we are in the situation where
\begin{align*}
&\ w_k \in \partial \H,\ &\ \lim_{i \to \infty} (r_{i,k} + t_{i,k}) R_i < \infty. 
\end{align*}
Then choose a sequence $\rho_i \to \infty$ but grows slower than $R_i$. Define the sequence of real translations by (recall $s_{i,k} = {\bf Re} w_{i,k}$)
\beqn
\varphi_{i,k}(z) = z + R_i s_{i,k}.
\eeqn
Via $\varphi_{i,k}$, the sequence of restrictions of $\A_i$ to $\M_{B_{\rho_i}^+ ( R_i s_{i, k})}$ can be identified with a sequence $\A_{i,k}'$ of solutions to the ASD equation over $\M_{B_{\rho_i}^+}$. The energy of $\A_{i,k}'$ is uniformly bounded. Then a subsequence of $\A_{i,k}'$ converges in the Uhlenbeck sense to an ASD instanton $\A_k'$ over $\M_k = \M = {\bm R} \times M$ with a bubbling measure ${\bm m}_k$. Similar to Lemma \ref{lemma73}, one use the annulus lemma (Proposition \ref{annulus}) for the case $N = M^{\rm double}$ to prove that no energy is lost, i.e., 
\beqn
m_\infty (w_k) = E ( \A_k) + \int_{\M_k} {\bm m}_k.
\eeqn
Similar to the above case, one has the following facts.

\begin{lemma}\label{lemma75}
If $\A_k$ is trivial, then ${\bm m}_k \neq 0$. 
\end{lemma}

\vspace{0.1cm}

\noindent {\bf Case III.} Suppose we are in the situation where
\begin{align*}
&\ w_k \in {\rm Int} \H,\ &\ \textcolor{black}{\lim_{i \to \infty} \tau_{i,k} = \infty\ {\rm where}\ \tau_{i,k}:= r_{i,k} R_i.}
\end{align*}
In this case, choose a sequence $\rho_i\to \infty$ which grows faster than $\tau_{i,k}$ but slower than $R_i$. Then when $i$ is large, the disk $B_{\rho_i}(R_i w_{i,k})\subset {\bm C}$ is contained in the upper half plane. Define a sequence of affine transformations $\varphi_{i,k}$ and a sequence of translations $\psi_{i,k}$ by
\begin{align*}
&\ \varphi_{i,k}(z) = \tau_{i,k} \left( z + \frac{ w_{i,k} }{r_{i,k}} \right),\ &\ \psi_{i,k}(z) = z +  \frac{ w_{i,k} }{r_{i,k}}.
\end{align*}
Via $\psi_{i,k}$, the sequence of restrictions of $\A_i$ to $B_{\rho_i} ( R_i w_{i,k}) \times \Sigma$ can be identified with a sequence $\A_{i,k}'$ of solutions to the ASD equation over $B_{\rho_i} \times \Sigma$. By Theorem \ref{thm33} a subsequence of $\A_i'$ converges \textcolor{black}{(in the sense of Definition \ref{defn34})} to a holomorphic map $\tilde {\bm u}_k = ({\bm u}_k, m_k)$ with mass from ${\bm C}$ to $R_\Sigma$ along with the sequence $\{ \tau_{i,k}\}$. Using similar method as proving Lemma \ref{lemma73} one can prove that 
\beqn
m_\infty (w_k) = E ( {\bm u}_k) + \int_{\bm C} m_k.
\eeqn

Moreover, we verify the stability condition as follows.

\begin{lemma}\label{lemma76}
When ${\bm u}_k$ is a constant, the support of $m_k$ contains at least two points.
\end{lemma}

\begin{proof}
If this is not the case, then $m_k$ is supported at a single point $z_0 \in \C$ with $m_k(z_0) = m_\infty (w_k)$. This implies that 
\beqn
\lim_{r \to \infty} \lim_{i \to \infty} E \Big( \A_{i, k}'; B_{r \tau_{i,k}}(z_0) \times \Sigma \Big)  = m_k(z_0) = m_\infty(w_k).
\eeqn
Since $m_k(z_0) \geq \hbar$, there is a sequence $\delta_i \to 0$ such that for large $i$
\beqn
E \Big( \A_{i,k}'; B_{\delta_i \tau_{i,k}}(z_0) \times \Sigma \Big) = E \Big( \A_i; \M_{\varphi_{R_i} ( B_{r_i'}( w_{i,k} + z_0) ) } \Big) = m_\infty(w_k) - \frac{\hbar}{2}.
\eeqn
Here 
\beqn
r_i' = \frac{ \delta_i \tau_{i,k}}{R_i} =  \delta_i r_{i,k}
\eeqn
which is smaller than $r_{i,k}$. This contradicts the choice of $w_{i,k}$ and $r_{i,k}$ (see \eqref{eqn76}). Hence the support of $m_k$ contains at least two points. 
\end{proof}

\vspace{0.1cm}

\noindent {\bf Case IV.} Suppose we are in the situation where 
\begin{align*}
&\ w_k \in \partial \H,\ &\ \textcolor{black}{\lim_{i \to \infty} \tau_{i,k} = \infty\ {\rm where}\ \tau_{i,k}:=R_i (r_{i,k} + t_{i,k} ).} 
\end{align*}
Define the sequence of \textcolor{black}{translations $\psi_{i,k}$} by
\beqn
\textcolor{black}{\psi_{i,k} (z) = z + R_i s_{i,k}}.
\eeqn
Denote the pull-back of $\A_k$ via $\psi_{i,k}$ by $\A_{i,k}'$, which are still ASD instantons over $\M$. Without loss of generality, we may assume that $R_i s_{i,k} = 0$ for all $i$. Then by Theorem \ref{thm51}, a subsequence of $\A_{i,k}'$ (still indexed by $i)$ converges to a holomorphic map with mass $\tilde {\bm u}_k = ({\bm u}_k, m_k)$ from ${\bm H}$ to $R_\Sigma$ along with the sequence $\{ \tau_{i,k}\}$. Moreover, one has
\beqn
\lim_{R \to \infty} \lim_{i \to \infty} E \Big( \A_{i,k}'; \M_{\varphi_{\tau_{i,k}}(B_R^+)} \Big) = E ( u_k ) + \int_{\bm H} m_k.
\eeqn
As before one can prove that the left hand side above is $m_\infty (w_k)$, hence
\beqn
m_\infty(w_k) = E(u_k) + \int_{{\bm H}} m_k.
\eeqn

We verify the stability condition as follows. 

\begin{lemma}\label{lemma77}
Suppose $u_k$ is a constant map. 
\begin{enumerate}
\item If $\displaystyle\lim_{i \to \infty} t_{i,k} / r_{i,k} < \infty$, then ${\rm Supp} m_k$ contains at least two points;

\item If $\displaystyle \lim_{i \to \infty} t_{i,k}/ r_{i,k} = \infty$, then ${\rm Supp} m_k$ contains at least one point in the interior of $\H$. 
\end{enumerate}
\end{lemma}

\begin{proof}
Suppose $\displaystyle \lim_{i \to \infty} t_{i,k} /r_{i,k} < \infty$ and ${\rm Supp} m_k$ contains only one point $z_0 \in \H$ with $m_k(z_0) = m_\infty (w_k)$. This implies that 
\beqn
\lim_{r \to 0} \lim_{i \to \infty} E \Big( \A_{i,k}'; \M_{\varphi_{\tau_{i,k}}(B_{r}^+(z_0) )} \Big) = m_\infty (w_k).
\eeqn
Then there exists a sequence $\delta_i \to 0$ satisfying 
\beqn
E \Big( \A_i; \M_{\varphi_{R_i} ( B_{\tau_{i,k} \delta_i/R_i}^+ ( \tau_{i,k} z_0/R_i))} \Big)  =  E \Big( \A_{i,k}'; \M_{\varphi_{\tau_{i,k}} (B_{\delta_i}^+(z_0) ) }  \Big) = m_\infty (w_k) - \frac{\hbar}{2}.
\eeqn
However we have
\beqn
\frac{\tau_{i,k}\delta_i}{R_i} = \delta_i (r_{i,k} + t_{i,k}) \ll r_{i,k},
\eeqn
which contradicts the choice of $r_{i,k}$ (see \eqref{eqn76}).

On the other hand, suppose $\displaystyle \lim_{i \to \infty} t_{i,k} /r_{i,k} = \infty$. Then $\tau_{i,k} \gg r_{i,k} R_i$. Still assume for simplicity that $s_{i,k} = {\bf Re} w_{i,k} = 0$. Then for any $r>0$, for $i$ sufficiently large, the disk centered at $R_i w_{i,k} = {\bm i} R_i t_{i,k}$ of radius $r_{i,k} R_i$ is contained in the disk centered at ${\bm i} \tau_{i,k}$ of radius $r \tau_{i,k}$. This implies that $m_k$ is supported at the single point ${\bm i} \in {\rm Int} \H$. \end{proof}

\subsection{Bubble connects}\label{subsection73}

The above discussion of four different cases allows us to inductively construct the limiting object. The process stops after finitely many steps due to the energy quantization phenomena (see the discussion at the beginning of this section). This finishes the construction of a collection of ASD instantons with measures $([\A_v], {\bm m}_v )$, a collection of holomorphic maps ${\bm u}_v$ from $S_v$ to $R_\Sigma$ with boundary lying in $\iota(L_M)$. They satisfy 
\beqn
\lim_{i \to \infty} E ( \A_i) = \sum_{v \in V_\Gamma^1}  \Big( E (\A_v ) + |{\bm m}_v | \Big) + \sum_{v \in V_\Gamma^\infty}  E (u_v).
\eeqn
Lemma \ref{lemma74}, \ref{lemma75}, \ref{lemma76}, and \ref{lemma77} imply that the collection satisfies the stability condition of Definition \ref{defn62}. To show that the limiting object satisfies the definition of stable scaled instantons (Definition \ref{defn62}), it remains to prove that this collection form a stable scaled instanton, i.e., proving ``bubbles connect'' or more precisely, the matching condition of Definition \ref{defn62}. We only prove this condition for a boundary node $w_e \in W_v$ corresponding to an edge $e_{v' \succ v} \in E_\Gamma$. The case for interior node can be proved in a similar way (see Definition \ref{defn62} for the notions of boundary and interior nodes). By the construction of the limit, there are two sequences of M\"obius transformations 
\begin{align*}
&\ \varphi_{i,v}(z) = \rho_{i,v} ( z + Z_{i,v}),\ &\ \varphi_{i,v'} = \rho_{i,v'} (z + Z_{i,v'})
\end{align*}
with $Z_{i,v}, Z_{i,v'}\in {\bm R}$ and $\rho_{i,v} \gg \rho_{i,v'}$. By the translation invariance, one can assume that $Z_{i,v'} \equiv 0$. Then one has the following lemma (recall the notion of the diameter in Subsection \ref{subsection43}.)

\begin{lemma}[Bubble connects] For sufficiently large $s$ one has
\beq\label{eqn712}
\lim_{i \to \infty} \sup_{z \in {\it Ann}^+( e^s \rho_{i, v'}, e^{-s} \rho_{i,v})} \| F_{\A_i} \|_{L^2(\{z \}\times \Sigma)} = 0
\eeq
and 
\beq\label{eqn713}
\lim_{s \to \infty} \lim_{i \to \infty} {\rm diam} ( \A_i; \M_{ e^s \rho_{i,v'}, e^{-s} \rho_{i,v}} ) = 0.
\eeq
\end{lemma}

\begin{proof}
By the inductive construction of the limiting object, there exists $s_0>0$ such that for all sufficiently large $i$, there holds
\beqn
E(\A_i; \M_{e^{s_0} \rho_{i,v'}, e^{-s_0} \rho_{i, v}}) \leq \frac{\hbar}{2}. 
\eeqn
Then the annulus lemma (Proposition \ref{annulus}), one has 
\beqn
\lim_{s \to \infty} \lim_{i \to \infty} E(\A_i; \M_{e^s \rho_{i,v'}, e^{-s} \rho_{i,v}}) = 0.
\eeqn
Then \eqref{eqn712} follows from standard estimate of the ASD equation. Hence the diameter in the limit \eqref{eqn713} is defined for sufficiently large $s$ and sufficiently large $i$. Then \eqref{eqn713} follows then from the diameter bound given by Lemma \ref{diameter1} and Lemma \ref{diameter2}.
\end{proof}

The above lemma implies that the evaluations of the two sides of the node corresponding to the edge $e_{v' \succ v}$ (which are points in $L_N$) agree after being mapped to $\iota(L_M) \subset R_\Sigma$. One can use the ASD equation on $\partial {\it Ann}^+(e^s \rho_{i, v'}, e^{-s} \rho_{i, v}) \times M_0$ and the fact that the energy of the solution restricted to this region shrinks to zero to prove that the evaluation (as a point of $L_N$) \textcolor{black}{of the limiting holomorphic map at one side of the node after transposing coincides with the evaluation at the other side of the node}. This finishes the proof of the matching condition for the edge $e_{v' \succ v}$. 

We declare that we have finished the proof of Theorem \ref{thm64} (the same as Theorem \ref{thm14}). As we have addressed previously, the case of instantons over $\C\times \Sigma$ (Theorem \ref{thm15}) can be proved using a similar and more simplified method.

\appendix

\section{Proof of Theorem \ref{thm33}}\label{appendixa}

The following lemma is similar to part of \cite[Theorem 3.1.1]{Duncan_thesis} (identical to \cite[Theorem 3.6]{Duncan_2012}). However as the statement is weaker and the proof is slightly different from Duncan's proof. 

\begin{lemma}\label{lemmaa1}
For any \textcolor{black}{$p\geq 2$}, the linearization of the Narasimhan--Seshadri map ${\it NS}_p: {\mc A}_{\epsilon_p}^{1,p}(Q) \to {\mc A}_{\rm flat}^{1,p}(Q)$, which is a map
\beqn
{\it DNS}_p: {\mc A}_{\epsilon_p}^{1,p}(Q) \times W^{1,p} (\Lambda^1\otimes {\rm ad}Q) \to W^{1,p}(\Lambda^1 \otimes {\rm ad}Q),
\eeqn
extends to a continuous map
\beqn
{\it DNS}_p: {\mc A}_{\epsilon_p}^{1,p}(Q) \to {\rm End}(L^p( \Lambda^1 \otimes {\rm ad}Q)).
\eeqn
\end{lemma}

\begin{proof}
Consider the map 
\beqn
{\mc F}: {\mc A}_{\epsilon_p}^{1,p}(Q) \times W^{2,p}({\rm ad} Q) \to L^p({\rm ad} Q)
\eeqn
defined by
\beqn
{\mc F}(B, h) = * F_{(e^{{\bm i} h})^* B}
\eeqn
\textcolor{black}{(see the explicit formula \eqref{eqn212})}. Then $h_B$ is the unique solution (subject to the conditions of the implicit function theorem) to ${\mc F}(B, h) = 0$. Moreover, there holds 
\beqn
\frac{\partial {\mc F}}{\partial B} + \frac{\partial {\mc F}}{\partial h} \frac{\partial h_B}{\partial B} = 0.
\eeqn
The partial derivative $\partial {\mc F}/ \partial h$ is a second order operator on $h$, which, as an invertible linear map from $W^{2,p}$ to $L^p$, depends on $B$ continuously. Hence it extends to an invertible operator from $W^{1,p}$ to $W^{-1,p}$ depending on $B\in W^{1,p}$ continuously. On the other hand, the partial derivative $\partial {\mc F}/ \partial B$ is a first order differential operator which, as a bounded operator from $W^{1,p}$ to $L^p$, depends continuously on $B \in W^{1,p}$. Hence it is also a bounded operator from $L^p$ to $W^{-1,p}$. Therefore, one has 
\beqn
\frac{\partial h_B}{\partial B} = - \left( \frac{\partial {\mc F}}{\partial h} \right)^{-1} \circ \frac{\partial {\mc F}}{\partial B}: L^p (\Lambda^1 \otimes {\rm ad} Q) \to W^{1,p}_B( {\rm ad} Q)
\eeqn
is bounded and depends on $B$ continuously. 
\end{proof}

Now we would like to compare the derivative of the Narasimhan--Seshadri map and the orthogonal projection onto harmonic forms. \textcolor{black}{Restrict to the $p = 2$ case.} For each $B \in {\mc A}^{1,2}(Q)$, one has the Hodge decomposition 
\beqn
L^2(\Sigma, \Lambda^1\otimes {\rm ad}Q) \simeq {\mc H}_B \oplus {\rm Im} d_B \oplus {\rm Im} d_B^*
\eeqn
where ${\mc H}_B$ is the subspace of harmonic ${\rm ad}Q$-valued 1-forms. Let 
\beqn
\pi_B: L^2(\Sigma, \Lambda^1 \otimes {\rm ad}Q) \to L^2(\Sigma, \Lambda^1 \otimes {\rm ad} Q)
\eeqn
be the orthogonal projection onto ${\mc H}_B$. On the other hand, suppose $\bar B = {\it NS}_2 (B)\in {\mc A}_{\rm flat}^{1,2}(Q)$ is defined and $\bar B$ represent a point $b \in R_\Sigma$. Then one has the natural identification 
\beqn
T_b R_\Sigma \cong {\mc H}_{\bar B}.
\eeqn
Via this identification, one can see that the derivative of $\ov{\it{NS}}_2: {\mc A}_{\epsilon_2}^{1,2}(Q) \to R_\Sigma$ at $B$ coincides with 
\beqn
D_B \ov{\it{NS}}_2 = \pi_{\bar B} \circ D_B {\it{NS}}_2: W^{1,2}_B(\Lambda^1 \otimes {\rm ad} Q) \to  {\mc H}_{\bar B} \subset L^2(\Lambda^1 \otimes {\rm ad} Q).
\eeqn
By Lemma \ref{lemmaa1}, $D_B\ov{\it{NS}}_2: L^2(\Lambda^1 \otimes {\rm ad}Q) \to L^2(\Lambda^1\otimes {\rm ad}Q)$ depends on $B$ continuously. 

The following corollary is an analogue of \cite[Corollary 3.1.13]{Duncan_thesis} (identically \cite[Corollary 3.17]{Duncan_2012}). Notice the difference in the choice of Sobolev exponents.

\begin{cor}\label{cora2}
For $\epsilon>0$ sufficiently small, the map 
\begin{align}\label{eqna1}
&\ {\mc A}_{\epsilon}^{1,2}(Q) \to [0, +\infty),\ &\ B \mapsto \| \pi_B - D_B \ov{\it{NS}}_2 \|_{L^2\to L^2}
\end{align}
is continuous and converges to zero as $\| F_B \|_{L^2(\Sigma)} \to 0$.
\end{cor}

\begin{proof}
\textcolor{black}{The projection of $\xi \in L^2 (\Sigma, \Lambda^1 \otimes {\rm ad} Q)$ is of the form}
\beq\label{eqna2}
\pi_B(\xi) = \xi - d_B h_\xi' - * d_B h_\xi''.
\eeq
\textcolor{black}{The condition that $\pi_B(\xi)$ is orthogonal to ${\rm Im} d_B \oplus {\rm Im} d_B^*$ is equivalent to that }
\begin{align}\label{eqna3}
&\ d_B^* d_B h_\xi' + * [F_B, h_\xi''] = d_B^* \xi,\ &\  d_B^* d_B h_\xi'' - * [F_B, h_\xi']  = - * d_B \xi.
\end{align}
\textcolor{black}{The left hand side is a linear operator $L_B$ on $h_\xi = (h_\xi', h_\xi'')$ which is bounded from $W^{1,2}$ to $W^{-1, 2}$. With respect to the operator norm, $L_B$ is continuous in $B\in {\mc A}_\epsilon^{1,2}$. Moreover, when $B$ is flat, $L_B$ coincides with the Laplacian $d_B^* d_B: W^{1,2} \to W^{-1, 2}$ which is invertible. As shown by Lemma \ref{lemma217}, when $\| F_B \|_{L^2}$ is sufficiently small, $B$ is $W^{1,2}$-close to a flat connection. Therefore when $\epsilon$ is sufficiently small, the equation \eqref{eqna3} has a unique solution $h_\xi = (h_\xi', h_\xi'') \in W^{1,2}$ to \eqref{eqna3} and there exists $C>0$ such that for all $B \in {\mc A}_\epsilon^{1,2}(Q)$ and $\xi \in L^2(\Sigma, \Lambda^1\otimes {\rm ad} Q)$}
\beq\label{eqna4}
\| h_\xi \|_{W^{1,2}_B} \leq C \left( \| d_B^* \xi \|_{W^{-1, 2}_B} + \| d_B \xi \|_{W^{-1, 2}_B} \right) \leq C \left( \| d_B^* \xi \|_{L^2} + \| d_B \xi \|_{L^2} \right).
\eeq
Therefore, the family of linear operators $\pi_B:L^2 \to L^2$ depends continuously on $B$. Together with Lemma \ref{lemmaa1}, the function \eqref{eqna1} is continuous. Moreover, when $B\in {\mc A}_{\rm flat}^{1,2}(Q)$, we know that $\pi_B = D_B \ov{\it{NS}}_2$. Therefore the function \eqref{eqna1} converges to zero as $\| F_B \|_{L^2} \to 0$. 
%where $h'$ and $h''$ are the unique solutions to 
%\begin{align*}
%&\ d_B d_B^* h' = d_B^* \xi,\ &\ d_B^* d_B h'' = - * d_B \xi.
%\end{align*}
%When $B \in {\mc A}_{\rm flat}(Q)$, $\Delta_B = d_B^* d_B: W^{2,2}_B({\rm ad}Q) \to L^2({\rm ad} Q)$ is invertible. Then when $B \in {\mc A}_\epsilon^{1,2}(Q)$ for $\epsilon$ sufficiently small, $\Delta_B: W_B^{2,2}({\rm ad} Q) \to L^2({\rm ad} Q)$ is still invertible. So 
%\beq\label{eqna2}
%\pi_B (\xi) = \xi - d_B \Delta_B^{-1} d_B^* \xi + * d_B \Delta_B^{-1} * d_B \xi
%\eeq
%which depends on $B\in {\mc A}_\epsilon^{1,2}(Q)$ continuously. Together with Lemma \ref{lemmaa1}, the function \eqref{eqna1} is continuous. When $B \in {\mc A}_{\rm flat}^{1,2}(Q)$, we know that $\pi_B = D_B \ov{\it{NS}}_2$. When $\| F_B \|_{L^2(Q)}$ is small, Lemma \ref{lemma217} shows that $B$ is $W^{1,2}$-close to a flat connection. Therefore the function \eqref{eqna1} converges to zero as $\| F_B \|_{L^2(\Sigma)} \to 0$. 
\end{proof}

We now turn to the proof of Theorem \ref{thm33}. We first prove it for the special case when the energy density does not blow up.

\begin{prop}\label{propa3}
Let $S$, $S_i$, $\A_i$, and $\rho_i$ be as in Theorem \ref{thm33}. Suppose for all compact subsets $K \subset S$
\beq\label{eqna5}
\limsup_{i \to \infty} \| e_{\rho_i} \|_{L^\infty(K)} < \infty.
\eeq
Then the conclusion of Theorem \ref{thm33} holds with $W_\infty = \emptyset$. 
\end{prop}

\begin{proof}

For any precompact open subset $K \subset S$, the hypothesis implies that 
\beqn
\limsup_{i \to \infty} \| \mu_{\A_i'} \|_{L^\infty( K)} = 0.
\eeqn
Hence when $i$ is sufficiently large the restriction $\A_i'|_K$ satisfies hypothesis of Proposition \ref{prop219} and hence defines a sequence of holomorphic maps
\beqn
u_i': K \to R_\Sigma,\ u_i'(z):= \ov{\it NS}_2(B_i'(z)).
\eeqn
We would like to show that $u_i'$ converges to a limiting holomorphic curve without bubbling. First we estimate the energy density function $e_i'(z) = |\partial_s u_i'(z)|^2$. One has 
\beq\label{eqna6}
\begin{split}
&\  \big| \| \pi_{B_i'(z)} ( \A_{i,s}' ) \|_{L^2(\{z\}\times \Sigma)}^2 - | \partial_s u_i'(z) |^2 \big| \\
= &\ \big| \| \pi_{B_i'(z)} ( \A_{i,s}' )\|_{L^2(\{z\}\times \Sigma)}^2 - \| (D_{B_i'(z)} \ov{\it NS}_2 ) ( \A_{i,s}') \|_{L^2(\{z\} \times \Sigma)}^2 \big|\\
\leq &\ C \| \A_{i,s}'\|_{L^2(\{z\}\times \Sigma)} \| \pi_{B_i'(z)}  (\A_{i,s}') - ( D_{B_i'(z)} \ov{\it NS}_2 ) ( \A_{i,s}') \|_{L^2(\Sigma)} \\
\leq &\ C \| \pi_{B_i'(z)} - D_{B_i'(z)} \ov{\it{NS}}_2 \|_{L^2 \to L^2}  \| \A_{i,s}' \|_{L^2(\{z\}\times \Sigma)}^2.
\end{split}
\eeq
Then since $ \| \A_{i,s}'\|_{L^2(\{z\}\times \Sigma)}^2 \leq e_{\rho_i}(z)$ is uniformly bounded over compact subsets of $S$, by Corollary \ref{cora2}, the energy density of $u_i'$ is also uniformly bounded over compact subsets of $S$. Therefore for the sequence $u_i'$ there is no energy blowup. Then by Gromov compactness (Proposition \ref{prop26}), a subsequence of $u_i'$ (still indexed by $i$) converges to a holomorphic map $u_\infty: S \to R_\Sigma$. 

We verify that $e_{\rho_i}$ converges to the energy density $e_\infty$ of $u_\infty$ in $C^0_{\rm loc}$. Recall that 
\beqn
e_{\rho_i}(z) = \| \A_{i,s}' \|_{L^2(\{z\}\times \Sigma)}^2 + \rho_i^2 \| \mu_{\A_i'}\|_{L^2(\{z\}\times \Sigma)}^2.
\eeqn
For any compact $K \subset S$, sufficiently large $i$, and $z \in K$, \textcolor{black}{by \eqref{eqna2} and \eqref{eqna4}, one has
\begin{multline*}
\| \A_{i,s}' - \pi_{B_i'(z)} (\A_{i,s}')  \|_{L^2(\{z\}\times \Sigma)} \\
 \leq  \| d_{B_i'(z)} h_{\A_{i, s}'}' \|_{L^2(\{z\}\times \Sigma)} +  \| d_{B_i'(z)}^* h_{\A_{i,s}'}'' \|_{L^2(\{z\}\times \Sigma)} \\
 \leq C \left( \| d_{B_i'(z)} \A_{i, s}' \|_{L^2(\{z\}\times \Sigma)} + \| d_{B_i'(z)}^* \A_{i,s}' \|_{L^2(\{z\}\times \Sigma)} \right)   .
\end{multline*}
%Since the Laplacian $\Delta_{B_i'(z)}$ is bounded from below, for some $C>0$ one has
%\beqn
%\| \A_{i,s}' -  \pi_{B_i'(z)} ( \A_{i,s}' ) \|_{L^2(\Sigma)} \leq C\Big( \| d_{B_i'(z)}^* %\A_{i,s}' \|_{L^2(\Sigma)} + \| d_{B_i'(z)} \A_{i,s}'\|_{L^2(\Sigma)}\Big).
%\eeqn
By Theorem \ref{thm31} and \eqref{eqn34}}, for some compact subset $K'\subset S$ whose interior contains $K$, one has 
\beq\label{eqna7}
\sup_{ z\in K} \| \A_{i,s}' - \pi_{B_i'(z)} (\A_{i,s}' ) \|_{L^2(\{z \} \times \Sigma)}^2 \leq C \rho_i^{-2} E_{\rho_i} (\A_i'; K' \times \Sigma).
\eeq
Moreover, by Theorem \ref{thm31} and \eqref{eqn36} (for some $p>2$), one has 
\beq\label{eqna8}
\sup_{z \in K} \rho_i^2 \| \mu_{\A_i'} \|_{L^2(\{z \}\times \Sigma)}^2 \leq C \rho_i^{-2\varepsilon}
\eeq
for certain $\varepsilon>0$. Then by \eqref{eqna6}, \eqref{eqna7}, and \eqref{eqna8}, one obtains that for all $z \in K$, 
\beqn
\begin{split}
&\  \big| e_{\rho_i}(z) - |\partial_s u_i'(z)|^2 \big| \\
\leq &\ \big| \| \pi_{B_i'(z)}( \A_{i,s}' ) \|_{L^2(\{z\}\times \Sigma)}^2 - |\partial_s u_i'(z)|^2 \big| + \| \A_{i,s}' - \pi_{B_i'(z)} ( \A_{i,s}' ) \|_{L^2(\Sigma)}^2 +  \rho_i^2 \| \mu_{\A_i'} \|_{L^2(\{z\}\times \Sigma)}^2 \\
\leq &\ C \| \pi_{B_i'(z)} - D_{B_i'(z)} \ov{\it{NS}}_2 \|_{L^2 \to L^2} \| \A_{i,s}'\|_{L^2(\{z\}\times \Sigma)}^2 + C \rho_i^{-2\varepsilon}.
\end{split}
\eeqn
By Corollary \ref{cora2} and the fact that $F_{B_i'(z)} \to 0$ uniformly over $K' \times \Sigma$, one has
\beqn
\lim_{i \to \infty} \| e_{\rho_i} - |\partial_s u_i'|^2\|_{L^\infty(K)} \to 0.
\eeqn
Since $u_i'$ converges to $u_\infty$ in $C^\infty_{\rm loc}$, it follows that $e_{\rho_i}$ converges to $e_\infty$ in $C^0_{\rm loc}$.
\end{proof}

In the constructive proof of Theorem \ref{thm33} one needs the following lemma.

\begin{lemma}[Hofer's lemma]\cite[Lemma 4.6.4]{McDuff_Salamon_2004}\label{hoferlemma}
Let $X$ be a metric space. Let $x \in X$ and $d > 0$. Suppose the closed ball $\ov{B_{2d}(x)}$ is complete. Let $f: X \to [0, +\infty)$ be a continuous function. Then there exist $\xi\in B_{2d}(x)$ and $\varepsilon \in (0, d]$ such that
\begin{align*}
&\ \sup_{\ov{B_\varepsilon(\xi)} } f \leq 2 f(\xi),\ &\ \varepsilon f(\xi) \geq d f(x).
\end{align*}
\end{lemma}

Now we prove the general case of Theorem \ref{thm33}. Define $\hbar > 0$ to be the smaller than 1) the minimal energy of nontrivial ${\bm R}^4$-instantons 2) the minimal energy of nontrivial ASD instantons over ${\bm C}\times \Sigma$ (whose existence is proved in \cite[Proof of Theorem 9.1]{Dostoglou_Salamon}\cite{Wehrheim_2006}) and 3) the minimal energy of nontrivial holomorphic spheres in $R_\Sigma$. 

\begin{proof}[Proof of Theorem \ref{thm33}]

{\bf Step 1.} We prove the following claim. 

\vspace{0.2cm}

\noindent {\it Claim A.} There exist a subsequence (still indexed by $i$), a finite subset $W_+ \subset S$, and a function $m_+: W_+ \to [\hbar, +\infty)$ satisfying the following condition.
\begin{enumerate}
\item For any compact subset $K \subset S \setminus W_+$ one has 
\beq\label{eqna9}
\limsup_{i \to \infty} \rho_i^{-2} \| e_{\rho_i}^\infty \|_{L^\infty(K)} < \infty.
\eeq

\item  \textcolor{black}{For each $w \in W_+$ one has
\beq\label{eqna10}
\lim_{r \to 0} \liminf_{i \to \infty} E_{\rho_i}(\A_i'; B_r(w) \times \Sigma) = m_+(w).
\eeq}
\end{enumerate}

\vspace{0.1cm}

\noindent {\it Proof of Claim A.} We construct inductively a finite subset $W_+^k \subset S$ with $k$ elements and a function $m_+^k:W_+^k \to [\hbar, +\infty)$ such that \eqref{eqna10} is satisfied for $W_+ = W_+^k$ for all $k$. For $k = 0$, define $W_+^0 = \emptyset$. Suppose we have constructed $W_+^k$ for some $k \geq 0$. Then for all compact subset $K \subset S \setminus W_+^k$, consider the limit \eqref{eqna9}. \textcolor{black}{If it is always zero}, then our inductive construction stops and the claim is proved. If the limit \eqref{eqna9} is not true for some $K$, then there exist a subsequence (still indexed by $i$), a sequence of points $w_i \in S_i\setminus W_+^k$ converging to $w \in S \setminus W_+^k$ such that
\beq\label{eqna11}
\lim_{i \to \infty} \rho_i^{-2} e_{\rho_i}^\infty(w_i) \to \infty.
\eeq
\textcolor{black}{Then there exists $r>0$ such that for all sufficiently large $i$,} 
\beqn
\ov{B_r(w_i)} \subset S_i \setminus W_+^k.
\eeqn
Define $\psi_i: B_{\rho_i r} \to B_r(w_i)\subset S_i$ by
\beqn
\textcolor{black}{\psi_i(z) = w_i + \frac{z}{\rho_i}.}
\eeqn
Then $\psi_i$ pulls back $\A_i'$ to a solution $\A_i$ to the un-scaled ASD equation over $B_{\rho_i r} \times \Sigma$. \eqref{eqna11} and the energy bound imply that 
\begin{align*}
&\ \limsup_{i \to \infty} \| F_{\A_i}\|_{L^2(B_{\rho_i r}\times \Sigma)} < \infty,\ &\ \| F_{\A_i}\|_{L^\infty(\{0\}\times \Sigma)} \to \infty.
\end{align*}
Then by the standard bubbling analysis for the ASD equation, for some subsequence (still indexed by $i$), a nontrivial ${\bm R}^4$-instanton bubbles off at some point in $\{0\}\times \Sigma$. Then for all small $r>0$, one has
\beqn
\liminf_{i \to \infty} E_{\rho_i}(\A_i'; B_r(w) \times \Sigma) \in [\hbar, +\infty).
\eeqn
\textcolor{black}{As the above limit is monotonic in $r$, the limit}
\beqn
\textcolor{black}{m_+(w):= \lim_{r\to 0} \liminf_{i \to \infty} E_{\rho_i}(\A_i'; B_r(w) \times \Sigma)\in [\hbar, +\infty) }
\eeqn
exists. Then $W_+^{k+1}$ and $m_+(w)$ are defined by adding to $W_+^k$ this point and the above value. Then the induction can be carried out. Lastly, since there always holds
\beqn
k \hbar \leq  \sum_{w \in W_+^k} m_+(w) \leq \limsup_{i \to \infty} E_{\rho_i}(\A_i') < \infty,
\eeqn
the induction must stops at a finite $k$. \hfill {\it End of the proof of Claim A.}

\vspace{0.2cm}

\noindent {\bf Step 2.} We prove the following claim.

\vspace{0.2cm}

\noindent {\it Claim B.} There is a subsequence (still indexed by $i$), a finite subset $W_0 \subset S \setminus W_+$, and a function $m_0: W_0 \to [\hbar, +\infty)$ satisfying the following conditions. 
\begin{enumerate}
\item For all compact subsets $K \subset S \setminus (W_+ \cup W_0)$, there holds
\beq\label{eqna12}
\limsup_{i \to \infty} \rho_i^{-2} \| e_{\rho_i}\|_{L^\infty(K)}  = 0.
\eeq

\item  For each $w \in W_0$, one has
\beqn
\textcolor{black}{ \lim_{r \to 0} \liminf_{i \to \infty} E_{\rho_i}(\A_i'; B_r(w) \times \Sigma) = m_0(w).}
\eeqn
\end{enumerate}

\vspace{0.1cm}

\noindent {\it Proof of Claim B.} We construct inductively a subset $W_0^k \subset S \setminus W_+$ with $k$ elements and a subsequence (still indexed by $i$) such that the above condition (b) is satisfied for $W_0 = W_0^k$. For $k = 0$ define $W_0^0 = \emptyset$. Suppose we have found $W_0^k$. Consider the limit \eqref{eqna12} for all compact subsets $K \subset S \setminus (W_+ \cup W_0^k)$. If it is always \textcolor{black}{zero}, then the induction stops and the claim is proved. Otherwise, there exist a subsequence (still indexed by $i$), a sequence of points $w_i \in S_i \setminus (W_+ \cup W_0^k)$ converging to a limit $w \in S \setminus (W_+ \cup W_0^k)$ such that
\begin{align}\label{eqna13}
&\ \lambda_i^2:= e_{\rho_i}(w_i) \to \infty,\ &\ \lim_{i \to \infty} \rho_i^{-2} \lambda_i^2 > 0.
\end{align}
By Claim A and \eqref{eqna9}, $\rho_i^{-2} \lambda_i^2$ converges to a finite limit. 

For each $R>0$, for $i$ sufficiently large, one has 
\beqn
B_{\rho_i^{-1} R}(w_i) \subset S_i \subset S. 
\eeqn
Consider the map 
\beqn
\psi_i: B_R \to \varphi_{\rho_i}( S_i ),\ z \mapsto  w_i + \frac{z}{\rho_i}.
\eeqn
The $\psi_i$ pulls back $\A_i'$ to a sequence of solutions $\A_i$ to the un-scaled ASD equation over $B_R \times \Sigma$. Condition \eqref{eqna9} implies that 
\beqn
\limsup_{i \to \infty} \| F_{\A_i}\|_{L^\infty(B_R \times \Sigma)} < \infty.
\eeqn
Then a subsequence (still indexed by $i$) converges up to gauge transformation to a solution to the un-scaled ASD equation over $B_R \times \Sigma$. Pushing $R$ to infinity, 
Moreover, the condition \eqref{eqna13} implies that
\beqn
\lim_{i \to \infty} \| F_{\A_i}\|_{L^2(\{0\}\times \Sigma)} > 0.
\eeqn
Hence the limiting solution is nontrivial, hence its energy is at least $\hbar$. Then we define $W_0^{k+1}:= W_0^k \cup \{w\}$ and \textcolor{black}{the limit}
\beqn
\textcolor{black}{m_0(w):= \lim_{r \to 0} \liminf_{i\to \infty} E_{\rho_i}( \A_i'; B_r(w) \times \Sigma)}
\eeqn
exists. Since the total energy is bounded, this induction stops at a finite $k$. Then the assertion of Claim B holds for $W_0 = W_0^k$. \hfill {\it End of the proof of Claim B.}

\vspace{0.2cm}

\noindent {\bf Step 3.} We prove the following claim. 

\vspace{0.2cm}

\noindent {\it Claim C.} There exist a subsequence (still indexed by $i$), a finite subset $W_- \subset S \setminus (W_+ \cup W_0)$, and a function $m_-: W_- \to (0, +\infty)$ satisfying the following conditions. 

\begin{enumerate}
\item For all compact subsets $K \subset S \setminus (W_+ \cup W_0 \cup W_-)$, one has 
\beq\label{eqna14}
\limsup_{i \to \infty} \| e_{\rho_i}\|_{L^\infty(K)} < \infty.
\eeq
 
\item For each $w \in W_-$, there holds
\beqn
\textcolor{black}{
\lim_{r \to 0} \liminf_{i \to \infty} E(\A_i'; B_r(w) \times \Sigma) = m_- (w).}
\eeqn
\end{enumerate}

\vspace{0.1cm}

\noindent {\it Proof of Claim C.} Similar as the proofs of Claim A and B, we construct a finite subset $W_-^k \subset S \setminus (W_+ \cup W_0)$ and the function $m_-: W_-^k \to (0, +\infty)$ inductively such that the above condition (b) holds. Starting from $W_-^0 = \emptyset$, suppose we have constructed $W_-^k$. Then consider the limit \eqref{eqna14} for all compact subsets $K \subset S\setminus (W_+ \cup W_0 \cup W_-^k)$. If it is always finite, then the induction stops and we have proved this claim. If not, then there exist a subsequence (still indexed by $i$), a sequence $w_i \in S_i \setminus (W_+ \cup W_0 \cup W_-^k)$ converging to a limit \textcolor{black}{$w \in S \setminus (W_+ \cup W_0 \cup W_-^k)$} such that 
\beqn
\lambda_i^2:= e_{\rho_i}(w_i)  \to +\infty.
\eeqn
Claim A and B imply that $\lambda_i$ grows slower than $\rho_i$. We would like to construct a holomorphic sphere bubble for a certain subsequence. 

Indeed, choose $d>0$ such that $\ov{B_{2d}(w_i)} \subset S_i \setminus (W_+ \cup W_0 \cup W_-^k)$.  For each $i$, apply Hofer's lemma for $f_i = \sqrt{ e_{\rho_i}}$. Then there exist $w_i' \in B_{2d}(w_i)$ and $\varepsilon_i \in (0, d]$ such that 
\begin{align}\label{eqna15}
&\ \sup_{B_{\varepsilon_i}(w_i')} e_{\rho_i} \leq 4 e_{\rho_i}(w_i'),\ &\ \varepsilon_i^2 e_{\rho_i}(w_i') \geq d^2 e_{\rho_i}(w_i).
\end{align}
Denote $\lambda_i':= \sqrt{ e_{\rho_i}(w_i')}$. Then \textcolor{black}{by \eqref{eqna12} and \eqref{eqna15} there holds}
\beqn
\lambda_i \leq \lambda_i' \ll \rho_i.
\eeqn
%implies that $\lambda_i'$ is comparable to $\rho_i$, hence is also comparable to $\lambda_i$.
Choose a sequence $R_i \to \infty$ which grows slower than $\lambda_i$. Then the map
\beqn
\psi_i: B_{R_i} \to S_i \setminus (W_+ \cup W_0 \cup W_-^k),\ z \mapsto w_i' + \frac{z}{\lambda_i'}
\eeqn
is defined for all $i$. By \eqref{eqna15}, one has 
\beq\label{eqna16}
\psi_i(B_{R_i}) = B_{R_i / \lambda_i'}(w_i') \subset B_{R_i \varepsilon_i/ d \lambda_i}(w_i') \subset B_{\varepsilon_i}(w_i').
\eeq
Then we obtained a sequence of solutions to the rescaled ASD equation
\begin{align*}
&\ \A_{i,s}'' + * \A_{i, t}'' = 0,\ &\  \kappa_{\A_i''} + (\rho_i'')^2 * \mu_{\A_i''} = 0
\end{align*}
over $B_{R_i} \times \Sigma$ where $\rho_i'':= \rho_i/ \lambda_i' \to \infty$. Let $e_{\rho_i''}$ be the energy density function of $\A_i''$. The first inequality of \eqref{eqna15} and \eqref{eqna16} implies that 
\beqn
\limsup_{i \to \infty} e_{\rho_i''}  \leq 4.
\eeqn
Moreover, $e_{\rho_i''}(0) = 1$. Then for all sufficiently large $i$ and all $z\in B_{R_i}$, $B_i''(z)$ is contained in the domain of $\ov{\it{NS}}_2$. So we obtain a sequence of holomorphic maps 
\begin{align*}
&\ u_i'': B_{R_i} \to R_\Sigma,\ &\ u_i''(z):= \ov{\it{NS}}_2( B_i''(z)).
\end{align*}
Then by Proposition \ref{propa3}, a subsequence of $u_i''$ converges to in $C^\infty_{\rm loc}$ to a holomorphic map $u_\infty'': {\bm C} \to R_\Sigma$. Moreover the convergence of energy density and the fact that $e_{\rho_i''}(0) = 1$ implies that $u_\infty''$ is a nontrivial holomorphic map with finite energy. Therefore, by Gromov's removal of singularity theorem, $u_\infty''$ extends to a holomorphic sphere in $R_\Sigma$, hence the energy of $u_\infty''$ is at least $\hbar$. Then we define 
\begin{align*}
&\ W_-^{k+1} = W_-^k \cup \{w\},\ &\ \textcolor{black}{m_-(w):= \lim_{r \to 0} \liminf_{i \to \infty} E_{\rho_i}(\A_i'; B_r(w) \times \Sigma).} 
\end{align*}
Similar to the previous claims, the induction stops at a finite $k$. \hfill {\it End of Proof of Claim C.}

\vspace{0.2cm}

\noindent {\bf Step 4.} We construct the holomorphic curve. Denote $W_\infty = W_+ \cup W_0 \cup W_-$ and let $m_\infty: W_\infty \to [\hbar, +\infty)$ be the union of $m_+$, $m_0$, and $m_-$. Take an arbitrary precompact open subset $K \subset S \setminus W_\infty$. Then \eqref{eqna14} implies that 
\beqn
\lim_{i \to \infty} \sup_{z \in K} \| F_{B_i'(z)} \|_{L^2(\Sigma)} = 0.
\eeqn
Therefore, for sufficiently large $i$ and all $z \in K$, $B_i'(z)$ is contained in the domain of $\ov{\it{NS}}_2$. Hence by Proposition \ref{prop219}, one obtains a sequence of holomorphic curves
\begin{align*}
&\ u_i': K \to R_\Sigma,\ &\ u_i'(z):= \ov{\it{NS}}_2(B_i'(z)).
\end{align*}
One can choose a subsequence (still indexed by $i$) and an exhaustive sequence of precompact open subsets $K_i \subset S \setminus W_\infty$ such that $K_i \subset K_{i+1}$ and $u_i'$ is defined over $K_i$. Then by condition (a) and Proposition \ref{propa3}, for a further subsequence (still indexed by $i$), the sequence $u_i': K_i \to R_\Sigma$ of holomorphic curves converges to a holomorphic map 
\beqn
u_\infty: S \setminus W_\infty \to R_\Sigma.
\eeqn
The convergence of energy density and the uniform energy bound imply that 
\beq\label{eqna17}
\lim_{i \to \infty} E_{\rho_i}(\A_i'; K \times \Sigma) = E(u_\infty; K)\ (\forall\ {\rm compact}\ K \subset S \setminus W_\infty).
\eeq
In particular, $u_\infty$ has finite energy. Then by Gromov's removal of singularity theorem, $u_\infty$ extends smoothly to a holomorphic map $u_\infty: S \to R_\Sigma$. 

\vspace{0.2cm}

\noindent \textcolor{black}{{\bf Step 5.} We prove the remaining claims of Theorem \ref{thm33}. First we prove that there exists a subsequence (still indexed by $i$) such that
\beq\label{eqna18}
\lim_{i \to \infty} E_{\rho_i}(\A_i'; B_r(z) \times \Sigma)\ {\rm exists}\	 (\forall\ r>0\ {\rm sufficiently\ small\ and}\ \forall z\in S).
\eeq
Indeed, if $z\notin W_\infty$, then by the uniform convergence of the energy density towards the energy density of the limiting holomorphic curve, we see that the above limit converges to the energy of $u_\infty$ restricted to $B_r(z)$. On the other hand, choose $r_0>0$ such that
\beqn
B_{r_0}(w) \cap B_{r_0}(w') = \emptyset,\ \forall w, w' \in W_\infty,\ w \neq w'.
\eeqn
Then one can find a subsequence (still indexed by $i$) such that for each $w \in W_\infty$, the limit 
\beqn
\lim_{i \to \infty} E_{\rho_i}(\A_i'; B_{r_0}(w)\times \Sigma)
\eeqn
exists. Then for all $r\in (0, r_0)$ and $w\in W_\infty$, by \eqref{eqna17}, one has
\beqn
\lim_{i \to \infty} E_{\rho_i}(\A_i'; (B_{r_0}(w) \setminus B_r(w)) \times \Sigma) = E(u_\infty; B_{r_0}(w) \setminus B_r(w)).
\eeqn
Hence the limit in \eqref{eqna18} exists for $z = w \in W_\infty$ and sufficiently small $r>0$. }

Now extend $m_\infty: W_\infty \to {\mb R}$ to a measure on $S$ by setting $m_\infty(z) = 0$ if $z \notin W_\infty$. Then by the definition of $m_\infty$ and \eqref{eqna18}, one has 
\beq\label{eqna19}
m_\infty(z) = \lim_{r\to \infty} \lim_{i \to \infty} E_{\rho_i}(\A_i'; B_r(z) \times \Sigma).
\eeq
Hence item (a) of Theorem \ref{thm33} is verified.

Notice that item (b) of Theorem \ref{thm33} has been established when we construct the limiting holomorphic curve. Now we verify item (c) of Theorem \ref{thm33}. From the definition of $m_\infty$, \eqref{eqna17}, and \eqref{eqna18}, if $K\subset S$ is any compact subset whose interior containing $W_\infty$ \textcolor{black}{and $r>0$ is sufficiently small}, then one has
\beqn
\begin{split}
&\ \lim_{i \to\infty} E_{\rho_i}(\A_i'; K \times \Sigma)\\
 = &\ \lim_{i \to \infty} \left( E_{\rho_i}( \A_i'; (K \setminus B_r(W_\infty)) \times \Sigma ) +  E_{\rho_i} (\A_i'; B_r(W_\infty) \times \Sigma) \right) \\
= &\ \lim_{i \to \infty} E_{\rho_i}( \A_i'; (K \setminus B_r(W_\infty)) \times \Sigma ) +  \lim_{i \to \infty}  E_{\rho_i} (\A_i'; B_r(W_\infty) \times \Sigma)  \\
= &\  E(u_\infty; K \setminus B_r(W_\infty)) +  \lim_{i \to \infty}  E_{\rho_i} (\A_i'; B_r(W_\infty) \times \Sigma) .
\end{split}
\eeqn
Let $r$ go to $0$, using \eqref{eqna19}, one obtains that
\beqn
\lim_{i \to \infty} E_{\rho_i}(\A_i'; K \times \Sigma) = E(u_\infty; K) + \int_K m_\infty.
\eeqn

Lastly, from the above construction and Proposition \ref{propa3}, $W_\infty= \emptyset$ if and only if $e_{\rho_i}$ is uniformly bounded over all compact subsets of $S$.
\end{proof}

\bibliographystyle{amsalpha}

\bibliography{../../../../mathref}

\end{document}